\documentclass[a4paper,11pt]{article}

\usepackage[applemac]{inputenc}
\usepackage{enumerate}
\usepackage{lmodern}
\usepackage[T1]{fontenc}
\usepackage{verbatim}
\usepackage{textcomp}
\usepackage[english]{babel}
\usepackage[a4paper,vmargin={3.5cm,3.5cm},hmargin={2.5cm,2.5cm}]{geometry}
\usepackage[font=sf, labelfont={sf,bf}, margin=1cm]{caption}
\usepackage[pagebackref=false,
pdftex]{hyperref}
\usepackage[pdftex]{color,graphicx}

\usepackage{amsmath,amsfonts,amssymb,amsthm,mathrsfs}
\usepackage{colonequals} 
\usepackage{stmaryrd}     

\usepackage{cleveref}
  \crefname{theorem}{Theorem}{Theorems}
  \crefname{lemma}{Lemma}{Lemmas}
  \crefname{remark}{Remark}{Remarks}
  \crefname{proposition}{Proposition}{Propositions}
  \crefname{definition}{Definition}{Definitions}
  \crefname{corollary}{Corollary}{Corollaries}
  \crefname{section}{Section}{Sections}
  \crefname{figure}{Figure}{Figures}

\usepackage{color}

\newtheorem{theorem}{Theorem}[]

\newtheorem{proposition}[theorem]{Proposition}
\newtheorem{lemma}[theorem]{Lemma}

\theoremstyle{definition}

\def\a{^{(\alpha)}}
\def\noi{\noindent}
\def\cc{{\mathcal C}}

\def\t{{\mathcal T}}
\def\n{{\mathcal N}}
\def\v{{\mathcal V}}
\def\ve{\varepsilon}
\def\wt{\widetilde}
\def\wh{\widehat}

\def\bd{{\bf d}}
\def\la{\longrightarrow}
\def\T{{\mathbb T}}
\def\R{{\mathbb R}}
\def\P{{\mathbb P}}
\def\E{{\mathbb E}}
\def\N{{\mathbb N}}
\def\Z{{\mathbb Z}}

\def\build#1_#2^#3{\mathrel{\mathop{\kern 0pt#1}\limits_{#2}^{#3}}}

\title{The harmonic measure of balls in critical Galton-Watson trees with infinite variance offspring distribution}

\author{Shen LIN
\thanks{E-mail address: \texttt{shen.lin.math@gmail.com}}\\
\textit{\small Universit\'e Paris-Sud XI}  } 

\date{}

\begin{document}

\maketitle

\begin{abstract}
We study properties of the harmonic measure of balls in large critical Galton-Watson trees whose offspring distribution is in the domain of attraction of a stable distribution with index $\alpha\in (1,2]$. Here the harmonic measure refers to the hitting distribution of height $n$ by simple random walk on the critical Galton-Watson tree conditioned on non-extinction at generation~$n$. For a ball of radius $n$ centered at the root, we prove that, although the size of the boundary is roughly of order $n^{\frac{1}{\alpha-1}}$, most of the harmonic measure is supported on a boundary subset of size approximately equal to $n^{\beta_{\alpha}}$, where the constant $\beta_{\alpha}\in (0,\frac{1}{\alpha-1})$ depends only on the index $\alpha$. Using an explicit expression of $\beta_{\alpha}$, we are able to show the uniform boundedness of $(\beta_{\alpha}, 1<\alpha\leq 2)$. These are generalizations of results in a recent paper of Curien and Le Gall~\cite{CLG13}.


\smallskip
\noindent{\bf AMS 2010 Classification Numbers.} 60J80, 60G50, 60K37. 
\end{abstract}

\section{Introduction}
Recently, Curien and Le Gall have studied in~\cite{CLG13} the properties of harmonic measure on generation $n$ of a critical Galton-Watson tree, whose offspring distribution has finite variance and which is conditioned to have height greater than $n$. They have shown the existence of a universal constant $\beta<1$ such that, with high probability, most of the harmonic measure on generation $n$ of the tree is concentrated on a set of approximately $n^{\beta}$ vertices, although the number of vertices at generation $n$ is of order $n$. Their approach is based on the study of a similar continuous model, where it is established that the Hausdorff dimension of the (continuous) harmonic measure is almost surely equal to $\beta$.

In this paper, we continue the above work by extending their results to the critical Galton-Watson trees whose offspring distribution has infinite variance. To be more precise, let $\rho$ be a non-degenerate probability measure on $\Z_{+}$ with mean one, and we assume throughout this paper that $\rho$ is in the domain of attraction of a stable distribution of index $\alpha\in(1,2]$, which means that
\begin{equation}
\label{eq:stable-attraction}
\sum\limits_{k\geq 0}\rho(k)r^{k}= r+(1-r)^{\alpha}L(1-r)\qquad \mbox{ for any } r\in [0,1),
\end{equation} 
where the function $L(x)$ is slowing varying as $x \to 0^{+}$. We point out that the finite variance condition for $\rho$ is sufficient for the previous statement to hold with $\alpha=2$. When $\alpha\in (1,2)$, by results of~\cite[Chapters XIII and XVII]{F71}, the condition (\ref{eq:stable-attraction}) is satisfied if and only if the tail probability 
\begin{displaymath}
\sum\limits_{k\geq x} \rho(k)=\rho([x,+\infty))
\end{displaymath}
varies regularly with exponent $-\alpha$ as $x \to +\infty$. See e.g.~\cite{BGT87} for the definition of regularly varying functions. 
 
Under the probability measure $\P$, for every integer $n\geq 0$, we let $\mathsf{T}^{(n)}$ be a Galton-Watson tree with offspring distribution $\rho$, conditioned on non-extinction at generation $n$. Conditionally given the tree $\mathsf{T}^{(n)}$, we consider simple random walk on $\mathsf{T}^{(n)}$ starting from the root. The probability distribution of the first hitting point of generation $n$ by random walk will be called the harmonic measure $\mu_{n}$, which is supported on the set $\mathsf{T}^{(n)}_{n}$ of all vertices of $\mathsf{T}^{(n)}$ at generation $n$. 
 
Let $q_{n}>0$ be the probability that a critical Galton-Watson tree $\mathsf{T}^{(0)}$ survives up to generation~$n$. It is shown in~\cite{S68} that, as $n\to \infty$, the probability $q_{n}$ decreases as $n^{-\frac{1}{\alpha-1}}$ up to multiplication by a slowly varying function, and $q_{n}\#\mathsf{T}^{(n)}_{n}$ converges in distribution to a non-trivial limit distribution on $\R_{+}$, whose Laplace transform can be written explicitly in terms of parameter $\alpha$. The following theorem generalizes the result~\cite[Theorem 1]{CLG13} in the finite variance case ($\alpha=2$) to all $\alpha \in (1,2]$.

\begin{theorem}
\label{thm:dim-discrete}
If the offspring distribution $\rho$ has mean one and belongs to the domain of attraction of a stable distribution of index $\alpha\in (1,2]$, there exists a constant $\beta_{\alpha}\in(0,\frac{1}{\alpha-1})$, which only depends on $\alpha$, such that for every $\delta>0$, we have the convergence in $\mathbb{P}$-probability
\begin{equation}
\label{eq:dim-discrete} 
\mu_n\Big(\big\{v\in \mathsf{T}^{(n)}_n\colon n^{-\beta_{\alpha}-\delta}\leq \mu_n(v) \leq n^{-\beta_{\alpha}+\delta}\big\}\Big)  \xrightarrow[n\to\infty]{(\mathbb{P})} 1\,.  
\end{equation}
Consequently, for every $\ve\in(0,1)$, there exists, with $\P$-probability tending to $1$ as $n\to\infty$, a~subset $A_{n,\ve}$ of  $\mathsf{T}^{(n)}_n$ such that $\# A_{n,\ve}\leq n^{\beta_{\alpha}+\delta}$ and $\mu_n(A_{n,\ve})\geq 1-\ve$. Conversely, the maximal $\mu_n$-measure of a set of cardinality bounded by $n^{\beta_{\alpha}-\delta}$ tends to $0$ as $n\to\infty$, in $\P$-probability.
\end{theorem}

The last two assertions of the preceding theorem are easy consequences of the convergence~(\ref{eq:dim-discrete}), as explained in~\cite{CLG13}.

We observe that the hitting distribution $\mu_{n}$ of generation $n$ by simple random walk on $\mathsf{T}^{(n)}$ is unaffected if we remove the branches of $\mathsf{T}^{(n)}$ that do not reach height $n$. Thus in order to establish the preceding result, we may consider simple random walk on $\mathsf{T}^{*n}$, the reduced tree associated with $\mathsf{T}^{(n)}$, which consists of all vertices of $\mathsf{T}^{(n)}$ that have at least one descendant at generation $n$. 

When the critical offspring distribution $\rho$ has infinite variance, scaling limits of the discrete reduced trees $\mathsf{T}^{*n}$ have been studied in~\cite{V77} and~\cite{Y80}. If we scale the graph distances by the factor $n^{-1}$, the discrete reduced trees $n^{-1}\mathsf{T}^{*n}$ converge to a random compact rooted $\R$-tree $\Delta^{(\alpha)}$ that we now describe. For every $\alpha\in(1,2]$, we define the $\alpha$-offspring distribution $\theta_{\alpha}$ as follows. For $\alpha=2$, we let $\theta_{2}=\delta_{2}$ be the Dirac measure at~2. If $\alpha<2$, $\theta_{\alpha}$ is the probability measure on $\mathbb{Z}_{+}$ given by
\begin{eqnarray*}
\theta_{\alpha}(0) &= &\theta_{\alpha}(1)\; = \;0, \\
\theta_{\alpha}(k) &= &\frac{\alpha\, \Gamma(k-\alpha)}{k!\,\Gamma(2-\alpha)}\,=\,\frac{\alpha(2-\alpha)(3-\alpha)\cdots(k-1-\alpha)}{k!}\,, \quad \forall k\geq 2,
\end{eqnarray*}
where $\Gamma(\cdot)$ is the Gamma function. We let $U_{\varnothing}$ be a random variable uniformly distributed over $[0,1]$, and let $K_{\varnothing}$ be a random variable distributed according to $\theta_{\alpha}$, independent of $U_{\varnothing}$. To construct $\Delta^{(\alpha)}$, one starts with an oriented line segment of length $U_{\varnothing}$, whose origin will be the root of the tree. We call $K_{\varnothing}$ the offspring number of the root $\varnothing$. Correspondingly, at the other end of the first line segment, we attach the origins of $K_{\varnothing}$ oriented line segments with respective lengths $U_{1},U_{2},\ldots,U_{K_{\varnothing}}$, such that, conditionally given $U_{\varnothing}$ and $K_{\varnothing}$, the variables $U_{1},U_{2},\ldots,U_{K_{\varnothing}}$ are independent and uniformly distributed over $[0,1-U_{\varnothing}]$. This finishes the first step of the construction. In the second step, for the first of these $K_{\varnothing}$ line segments, we independently sample a new offspring number $K_{1}$ distributed as $\theta_{\alpha}$, and attach $K_{1}$ new line segments whose lengths are again independent and uniformly distributed over $[0,1-U_{\varnothing}-U_{1}]$, conditionally on all the random variables appeared before. For the other $K_{\varnothing}-1$ line segments, we repeat this procedure independently. We continue in this way and after an infinite number of steps we get a random non-compact rooted $\R$-tree, whose completion is the random compact rooted $\R$-tree $\Delta^{(\alpha)}$. We will call $\Delta^{(\alpha)}$ the reduced stable tree of parameter $\alpha$. See Section~\ref{sec:treedelta} for a more precise description. Notice that all the offspring numbers involved in the construction of $\Delta^{(2)}$ are a.s.~equal to 2, which correspond to the  binary branching mechanism. In contrast, this is no longer the case when $1<\alpha<2$. 

We denote by $\mathbf{d}$ the intrinsic metric on $\Delta^{(\alpha)}$. By definition, the boundary $\partial \Delta^{(\alpha)}$ consists of all points of $\Delta\a$ at height 1. As the continuous analogue of simple random walk, we can define Brownian motion on $\Delta\a$ starting from the root and up to the first hitting time of $\partial \Delta\a$. It behaves like linear Brownian motion as long as it stays inside a line segment of $\Delta\a$. It is reflected at the root of $\Delta\a$ and when it arrives at a branching point, it chooses each of the adjacent line segments with equal probabilities. We define the (continuous) harmonic measure $\mu_{\alpha}$ as the (quenched) distribution of the first hitting point of $\partial \Delta\a$ by Brownian motion. 

\begin{theorem}
\label{thm:dim-stable}
For every index $\alpha\in(1,2]$, with the same constant $\beta_{\alpha}$ as in Theorem~\ref{thm:dim-discrete}, we have $\P$-a.s.~$\mu_{\alpha}(\mathrm{d}x)$-a.e., 
\begin{equation}
\label{eq:loc-dim-har}
\lim_{r\downarrow 0} \frac{\log \mu_{\alpha}(\mathcal{B}_{\bd}(x,r))}{\log r}=\beta_{\alpha}\,,
\end{equation}
where $\mathcal{B}_{\bd}(x,r)$ stands for the closed ball of radius $r$ centered at $x$ in the metric space $(\Delta^{(\alpha)},\bd)$. Consequently, the Hausdorff dimension of $\mu_{\alpha}$ is $\P$-a.s.~equal to $\beta_{\alpha}$.
\end{theorem}

According to Lemma 4.1 in \cite{LPP95}, the last assertion of the preceding theorem follows directly from~(\ref{eq:loc-dim-har}). As another direct consequence of~(\ref{eq:loc-dim-har}), we have that $\P$-a.s.~for $\mu_{\alpha}(\mathrm{d}x)$-a.e.~$x\in \partial \Delta^{(\alpha)}$, $\mu_{\alpha}(\mathcal{B}_{\bd}(x,r))\to 0$ as $r\downarrow 0$, which is equivalent to non-atomicity of $\mu_{\alpha}$. 

Since it has been proved in \cite[Theorem 1.5]{DLG06} that the Hausdorff dimension of $\partial \Delta^{(\alpha)}$ with respect to $\bd$ is a.s.~equal to $\frac{1}{\alpha-1}$, the previous theorem implies that the harmonic measure has a.s.~strictly smaller Hausdorff dimension than that of the whole boundary of the reduced stable tree. This phenomenon of dimension drop has been shown in~\cite[Theorem 2]{CLG13} for the special case of binary branching $\alpha=2$. 

We prove Theorem~\ref{thm:dim-stable} in Section~\ref{sec:stationary+ergodic}, where our approach is different and shorter than the one developed in~\cite{CLG13} for the special case $\alpha=2$.

Notice that the Hausdorff dimension of the boundary $\partial \Delta^{(\alpha)}$ increases to infinity when $\alpha \downarrow 1$. However, it is an interesting fact that the Hausdorff dimension of the harmonic measure remains bounded when $\alpha \downarrow 1$.
\begin{theorem}
\label{thm:bdd-dimension}
There exists a constant $C>0$ such that for any $\alpha \in (1,2]$, we have $\beta_{\alpha}<C$.
\end{theorem}

Our proof of Theorem~\ref{thm:bdd-dimension} relies on the fact that the constant $\beta_{\alpha}$ in Theorems~\ref{thm:dim-discrete} and~\ref{thm:dim-stable} can be expressed in terms of the conductance of $\Delta\a$. Informally, if we think of the random tree $\Delta\a$ as a network of resistors with unit resistance per unit length, the effective conductance between the root and the boundary $\partial \Delta\a$ is a random variable which we denote by $\mathcal{C}\a$. From a probabilistic point of view, it is the mass under the Brownian excursion measure for the excursion paths away from the root that hit height 1. Following the definition of $\Delta\a$ and the above electric network interpretation, the distribution of $\mathcal{C}\a$ satisfies the recursive distributional equation
\begin{equation}
\label{eq:rde}
\mathcal{C}^{(\alpha)}\,\overset{(\mathrm{d})}{=\joinrel=}\, \bigg (U + \frac{1-U}{ \mathcal{C}\a_1+ \mathcal{C}\a_2+\cdots+\mathcal{C}\a_{N_{\alpha}}}\bigg)^{-1},
\end{equation} 
where $(\mathcal{C}\a_{i})_{i\geq1}$ are i.i.d.~copies of $\mathcal{C}\a$, the integer-valued random variable $N_{\alpha}$ is distributed according to $\theta_{\alpha}$, and $U$ is uniformly distributed over $[0,1]$. All these random variables are supposed to be independent. 

\begin{proposition}
\label{prop:dim-formula}
For any $\alpha \in (1,2]$, the distribution $\gamma_{\alpha}$ of the conductance $\mathcal{C}\a$ is characterized in the class of all probability measures on $[1,\infty)$ by the distributional equation~(\ref{eq:rde}). The constant $\beta_{\alpha}$ appearing in Theorems \ref{thm:dim-discrete} and~\ref{thm:dim-stable} is given by 
\begin{equation}
\label{eq:beta-value}
\beta_{\alpha}= \frac{1}{2}\bigg(\frac{\big(\int \gamma_{\alpha}(\mathrm{d}s) s\big)^{2}}{\iint \gamma_{\alpha}(\mathrm{d}s)\gamma_{\alpha}(\mathrm{d}t)\frac{st}{s+t-1}}-1\bigg).
\end{equation}
\end{proposition}

Interestingly, formula~(\ref{eq:beta-value}) expresses the exponent $\beta_{\alpha}$ as the same function of the distribution $\gamma_{\alpha}$, for all $\alpha\in (1,2]$. In the course of the proof, we obtain two other formulas for $\beta_{\alpha}$ (see~(\ref{eq:beta2}) and~(\ref{eq:beta1}) below), but they both depend on $\alpha$ in a more complicated way, which also involves the distribution $\theta_{\alpha}$.

The paper is organized as follows. In Section 2 below, we study the continuous model of Brownian motion on $\Delta\a$. A formal definition of the reduced stable tree $\Delta\a$ is given in Section~\ref{sec:treedelta}. In Section~\ref{sec:ctgwtree} we explain how to relate $\Delta\a$ to an infinite supercritical continuous-time Galton-Watson tree $\Gamma\a$, and we reformulate Theorem~\ref{thm:dim-stable} in terms of Brownian motion with drift $1/2$ on $\Gamma\a$. Properties of the law of the random conductance $\mathcal{C}\a$, including the first assertion of Proposition~\ref{prop:dim-formula}, are discussed in Section~\ref{sec:conductance}, and Section~\ref{sec:coupling} gives the coupling argument that allows one to derive Theorem~\ref{thm:bdd-dimension} from formula~(\ref{eq:beta-value}). Section~\ref{sec:stationary+ergodic} is devoted to the proofs of Theorem~\ref{thm:dim-stable} and of formula~(\ref{eq:beta-value}). We emphasize that our approach to Theorem~\ref{thm:dim-stable} is different from the one used in~\cite{CLG13} when $\alpha=2$. In fact we use an invariant measure for the environment seen by Brownian motion on $\Gamma\a$ at the last passage time of a node of the $n$-th generation, instead of the last passage time at a height $h$ as in~\cite{CLG13}. We then apply the ergodic theory on Galton-Watson trees, which is a powerful tool initially developed in~\cite{LPP95}.
  
In Section 3 we proceed to the discrete setting concerning simple random walk on the discrete reduced tree $\mathsf{T}^{*n}$. Let us emphasize that, when the critical offspring distribution $\rho$ is in the domain of attraction of a stable distribution of index $\alpha\in(1,2)$, the convergence of discrete reduced trees is less simple than in the special case $\alpha =2$ where we have a.s.~a binary branching structure. See Proposition~\ref{conv-reduced-tree} for a precise statement in our more general setting. Apart from this ingredient, we need several estimates for the discrete reduced tree $\mathsf{T}^{*n}$ to derive Theorem~\ref{thm:dim-discrete} from Theorem~\ref{thm:dim-stable}. For example, Lemma~\ref{lem:level-size} gives a bound for the size of level sets in $\mathsf{T}^{*n}$, and Lemma~\ref{moment-conductance} presents a moment estimate for the (discrete) conductance $\mathcal{C}_{n}(\mathsf{T}^{*n})$ between generations 0 and $n$ in $\mathsf{T}^{*n}$. Although the result analogous to Lemma~\ref{moment-conductance} in~\cite{CLG13} is a second moment estimate, we only manage to give a moment estimate of order strictly smaller than $\alpha$ if the critical offspring distribution $\rho$ satisfies~(\ref{eq:stable-attraction}) with $\alpha\in (1,2]$. Nevertheless, this is sufficient for our proof of Theorem~\ref{thm:dim-discrete}, which is adapted from the one given in~\cite{CLG13}. 

Comments and several open questions are gathered in the last section. Following the work of A\"\i d\'ekon~\cite{Aid11}, we obtain a candidate for the speed of Brownian motion with drift $1/2$ on the infinite tree $\Gamma\a$, expressed by~(\ref{eq:speed}) in terms of the continuous conductance $\mathcal{C}\a$. Nonetheless, the monotonicity properties of this quantity remains open. It would also be of interest to know whether or not the Hausdorff dimension $\beta_{\alpha}$ of the continuous harmonic measure $\mu_{\alpha}$ is monotone with respect to $\alpha \in(1,2]$.

\smallskip
\noindent \textbf{Acknowledgments.} The author is deeply indebted to J.-F.~Le Gall and N.~Curien for many helpful suggestions during the preparation of this paper. 

\section{The continuous setting}

\subsection{The reduced stable tree}
\label{sec:treedelta}

We set 
$$\mathcal{V} = \bigcup_{n=0}^\infty \mathbb{N}^n$$
where by convention $\mathbb{N}=\{1,2,\ldots\}$ and $\mathbb{N}^0=\{\varnothing\}$. If $v=(v_1,\ldots,v_n)\in\mathcal{V}$, we set $|v|=n$ (in particular, $|\varnothing|=0$),  and if $n\geq 1$,
we define the parent 
of $v$ as $\widehat v=(v_1,\ldots,v_{n-1})$ and then say that $v$ is a child of $\wh v$. For two elements $v=(v_1,\ldots,v_n)$ and $v'=(v'_1,\ldots,v'_m)$ belonging to $\mathcal{V}$, their concatenation is $vv'\colonequals (v_1,\ldots,v_n,v'_1,\ldots,v'_m)$. The notions of a descendant and an ancestor of an element of $\v$ are defined in the obvious way, with the convention that every $v\in \v$ is both an ancestor and a descendant of itself. If $v,w\in\v$, $v\wedge w$ is the unique element of $\v$ such that it is a common ancestor of $v$ and $w$, and $|v\wedge w|$ is maximal. 

An infinite subset $\Pi$ of $\mathcal{V}$ is called an infinite discrete tree if there exists a collection of positive integers $k_{v}=k_{v}(\Pi)\in \N$ for every $v\in \mathcal{V}$ such that 
\begin{displaymath}
\Pi=\{\varnothing\}\cup\{(v_{1},\ldots,v_{n})\in \mathcal{V}: v_{j}\leq k_{(v_{1},\ldots,v_{j-1})} \mbox{ for every } 1\leq j\leq n\}.
\end{displaymath}

Recall the definition of the $\alpha$-offspring distribution $\theta_{\alpha}$ for $\alpha\in (1,2]$. It will also be convenient to consider the case $\alpha=1$, where we define $\theta_{1}$ as the probability measure on $\Z_{+}$ given by 
\begin{eqnarray*}
\theta_{1}(0) &= &\theta_{1}(1)\; = \;0, \\
\theta_{1}(k) &= &\frac{1}{k(k-1)}\,, \quad \forall k\geq 2.
\end{eqnarray*}
If $\alpha \in(1,2]$, the generating function of $\theta_{\alpha}$ is given (see e.g.~\cite[p.74]{DLG02}) as 
\begin{equation}
\label{eq:gene-fct}
\sum\limits_{k\geq 0}^{}\theta_{\alpha}(k)\,r^k =\frac{(1-r)^{\alpha}-1+\alpha r}{\alpha-1},\quad \forall r\in (0,1],
\end{equation}
while for $\alpha=1$,
\begin{equation}
\label{eq:gene-fct-1}
\sum\limits_{k\geq 0}^{}\theta_1(k)\,r^k= r+(1-r)\log(1-r) ,\quad \forall r\in (0,1].
\end{equation}
Notice that for $\alpha\in(1,2]$, the mean of $\theta_{\alpha}$ is given by 
\begin{equation*}
m_{\alpha}=\frac{\alpha}{\alpha-1} \in[2,\infty),
\end{equation*}
whereas $\theta_{1}$ has infinite mean.  

For fixed $\alpha\in [1,2]$, we introduce a collection $(K_{\alpha}(v))_{v\in\v}$ of independent random variables distributed according to $\theta_{\alpha}$ under the probability measure $\P$, and define a random infinite discrete tree 
\begin{displaymath}
\Pi^{(\alpha)}\colonequals \{\varnothing\}\cup\{(v_{1},\ldots,v_{n})\in \v\colon v_{j}\leq K_{\alpha}((v_{1},\ldots,v_{j-1})) \mbox{ for every }1\leq j\leq n\}\,.
\end{displaymath}
We point out that $\Pi^{(2)}$ is an infinite binary tree.

Let $(U_v)_{v\in\v}$ be another collection, independent of $(K_{\alpha}(v))_{v\in\v}$, consisting of independent real random variables uniformly distributed over $[0,1]$ under the same probability measure~$\P$. We set now
$$Y_\varnothing=U_\varnothing$$
and then by induction, for every $v\in \Pi^{(\alpha)}\setminus \{\varnothing\}$,
$$Y_v = Y_{\hat v} + U_v(1- Y_{\hat v}).$$
Note that  a.s.~$0\leq Y_v< 1$ for every $v\in \Pi^{(\alpha)}$. Consider then the set
$$\Delta^{(\alpha)}_{0} \colonequals \big(\{\varnothing\}\times [0,Y_\varnothing] \big)\cup  \bigg(\bigcup_{v\in\Pi^{(\alpha)}\backslash\{\varnothing\}} \{v\} \times (Y_{\hat v}, Y_v]\bigg).$$
There is a straightforward way to define a metric $\bd$ on $\Delta^{(\alpha)}_{0}$, so that 
$(\Delta^{(\alpha)}_{0},\bd)$ is a (noncompact) $\R$-tree and, for every $x=(v,r)\in \Delta^{(\alpha)}_{0}$, we have $\bd((\varnothing,0), x)=r$. To be specific, let $x=(v,r)\in\Delta^{(\alpha)}_{0}$ and $y=(w,r')\in\Delta^{(\alpha)}_{0}$:
\begin{enumerate}
\item[$\bullet$] If $v$ is a descendant (or an ancestor) of $w$, we set $\bd(x,y)= |r-r'|$.
\item[$\bullet$] Otherwise, $\bd(x,y)= \bd((v\wedge w,Y_{v\wedge w}),x)+ \bd((v\wedge w,Y_{v\wedge w}),y)
=(r-Y_{v\wedge w})+(r'-Y_{v\wedge w})$.
\end{enumerate}
See \cref{fig:Delta} for an illustration of the tree $\Delta^{(\alpha)}_{0}$ when $\alpha <2$.

\begin{figure}[!h]
 \begin{center}
 \includegraphics[width=13cm]{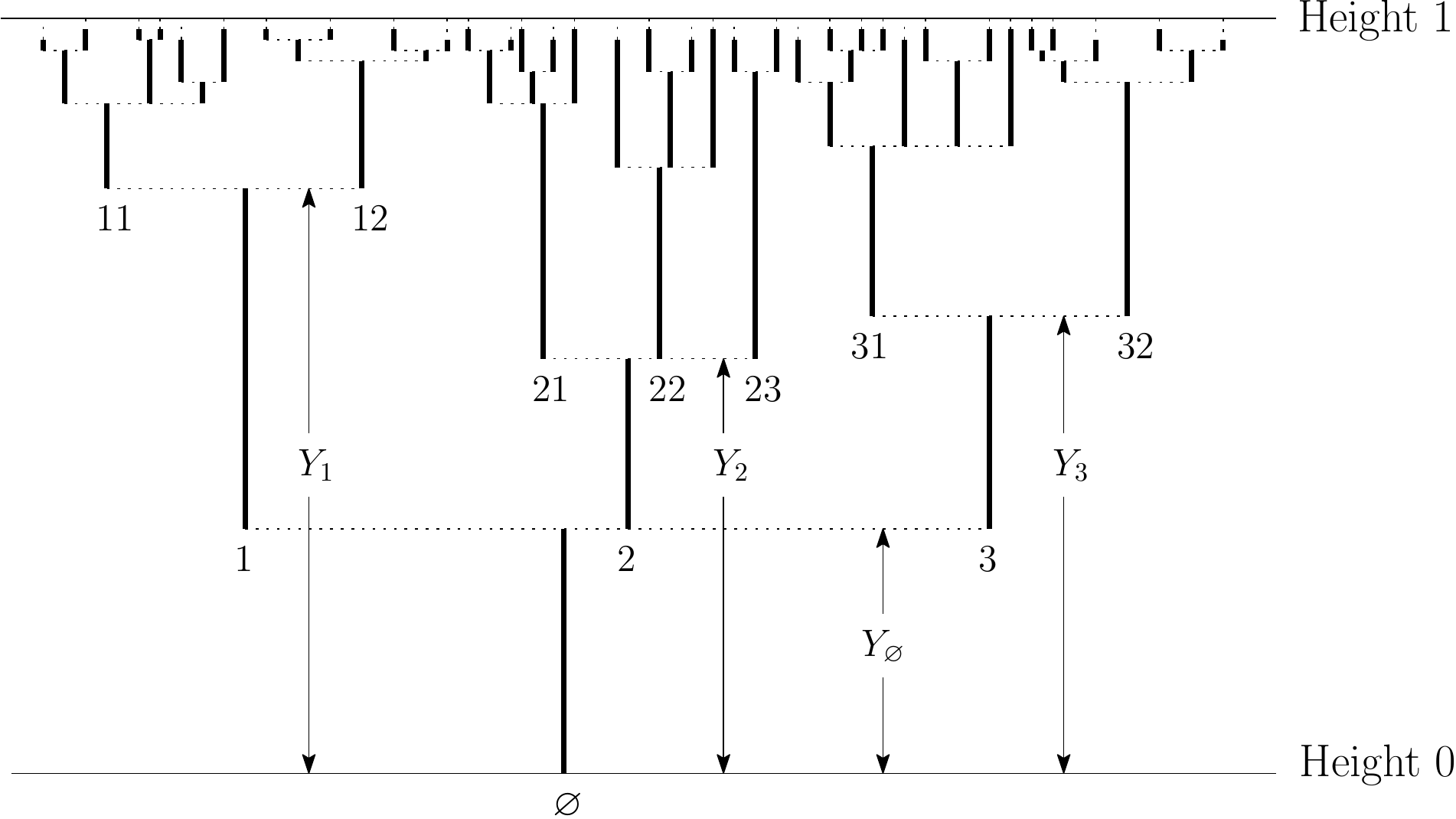}
 \caption{\label{fig:Delta}The random tree $\Delta^{(\alpha)}_{0}$ when $1\leq \alpha<2$}
 \end{center}
 \end{figure}

We let $\Delta^{(\alpha)}$ be the completion of $\Delta^{(\alpha)}_{0}$ with respect to the metric $\bd$. Then
$$ \Delta^{(\alpha)}=\Delta^{(\alpha)}_{0} \cup \partial \Delta^{(\alpha)}$$
where by definition $\partial \Delta^{(\alpha)}\colonequals \{x\in\Delta^{(\alpha)}\colon \bd((\varnothing,0), x)=1\}$, which can be identified with a random subset of $\N^{\N}$. It is immediate to see that $(\Delta^{(\alpha)},\bd)$ is a compact $\R$-tree, which we will call the reduced stable tree of index $\alpha$.

The point $(\varnothing,0)$ is called the root of $\Delta^{(\alpha)}$. For every $x\in\Delta^{(\alpha)}$, we set $H(x)=\bd((\varnothing,0), x)$
and call $H(x)$ the height of $x$. We can  define a genealogical order on $\Delta^{(\alpha)}$ by setting $x\prec y$ if and only if $x$ belongs to the geodesic path from the root to $y$.

For every $\ve\in(0,1)$, we set
$$\Delta_{\ve}^{(\alpha)}\colonequals \{x\in\Delta^{(\alpha)}\colon H(x)\leq 1-\ve\},$$
which is also a compact $\R$-tree for the metric $\bd$. 
The leaves of $\Delta_\ve^{(\alpha)}$ are the points of the form $(v,1-\ve)$ for all $v\in\v$ such that $Y_{\hat v}< 1-\ve\leq Y_v$. The branching points of $\Delta_{\ve}^{(\alpha)}$ are the points of the form $(v,Y_v)$ for all $v\in \mathcal{V}$ such that $Y_v<1-\ve$.  

Now conditionally on $\Delta^{(\alpha)}$, we can define Brownian motion on $\Delta_{\ve}^{(\alpha)}$ starting from the root. Informally, this process behaves like linear Brownian motion as long as it stays on an ``open interval'' of the form $\{v\}\times (Y_{\hat v},Y_v\wedge (1-\ve))$, and it is reflected at the root $(\varnothing,0)$ and at the leaves of $\Delta_{\ve}^{(\alpha)}$. When it arrives at a branching point of the tree, it chooses each of the possible line segments ending at this point with equal probabilities. By taking a sequence $\ve_n=2^{-n}, n\geq 1$ and then letting $n$ go to infinity, we can construct under the same probability measure $P$ a Brownian motion $B$ on $\Delta^{(\alpha)}$ starting from the root, which is defined up to its first hitting time $T$ of $\partial \Delta^{(\alpha)}$. We refer the reader to~\cite{CLG13} for the details of this construction. The harmonic measure $\mu_{\alpha}$ is then the distribution of $B_{T-}$ under~$P$, which is a (random) probability measure on $\partial \Delta^{(\alpha)}\subseteq \N^\N$.

\subsection{The continuous-time Galton-Watson tree}
\label{sec:ctgwtree}

In this subsection, we introduce a new tree which shares the same branching structure as $\Delta^{(\alpha)}$, such that each point of $\Delta^{(\alpha)}$ at height $s\in [0,1)$ corresponds to a point of the new tree at height $-\log(1-s)\in [0,\infty)$ in a bijective way. As it turns out, this new random tree is a continuous-time Galton-Watson tree. 

To define it, we take $\alpha\in [1,2]$ and start with the same random infinite tree $\Pi^{(\alpha)}$ introduced in Section~\ref{sec:treedelta}. Consider now a collection $(V_v)_{v\in\mathcal{V}}$ of independent real random variables exponentially distributed with mean $1$ under the probability measure $\P$. We set
$$Z_\varnothing=V_\varnothing$$
and then by induction, for every $v\in \Pi^{(\alpha)}\setminus\{\varnothing\} $,
$$Z_v = Z_{\hat v} + V_v.$$
The continuous-time Galton-Watson tree (hereafter to be called CTGW tree for short) of stable index~$\alpha$ is the set
$$\Gamma^{(\alpha)} \colonequals\big(\{\varnothing\}\times [0,Z_\varnothing]\big) \cup  \bigg(\bigcup_{v\in \Pi^{(\alpha)}\backslash\{\varnothing\}} \{v\} \times (Z_{\hat v}, Z_v]\bigg),$$
which is equipped with the metric $d$ defined in the same way as $\bd$ in the preceding subsection. For this metric, $\Gamma^{(\alpha)}$ is a non-compact $\R$-tree. For every $x=(v,r)\in\Gamma^{(\alpha)}$, we keep the notation $H(x)=r=d((\varnothing,0),x)$ for the height of the point $x$.  

Now observe that if $U$ is uniformly distributed over $[0,1]$, the random variable $-\log(1-U)$ is exponentially distributed with mean $1$. Hence we may and will suppose that the collection $(V_v)_{v\in\mathcal{V}}$ is constructed from the collection $(U_v)_{v\in\mathcal{V}}$ in the previous subsection via the formula $V_v=-\log(1-U_v)$ for every $v\in\mathcal{V}$. Then, the mapping $\Psi$ defined on $\Delta^{(\alpha)}_{0}$ by 
\begin{displaymath}
\Psi(v,r)\colonequals \big(v,-\log(1-r)\big) \quad \mbox{ for every }(v,r)\in\Delta^{(\alpha)}_{0},
\end{displaymath}
is a homeomorphism from $\Delta^{(\alpha)}_{0}$ onto $\Gamma^{(\alpha)}$.

By stochastic analysis, we can write for every $t\in[0,T)$,
\begin{equation}
\label{BM-CTGW}
\Psi(B_t) =W\Big(\int_0^t (1-H(B_s))^{-2}\,\mathrm{d}s\Big)
\end{equation}
where $(W(t))_{t\geq 0}$ is Brownian motion with constant drift $1/2$ towards infinity on the CTGW tree~$\Gamma\a$ (this process is defined in a similar way as Brownian motion on $\Delta^{(\alpha)}_{\ve}$, except that it behaves like Brownian motion with drift $1/2$ on every ``open interval'' of the tree). Note that again $W$ is defined under the probability measure $P$. From now on, when we speak about Brownian motion on the CTGW tree or on other similar trees, we will always mean Brownian motion with drift $1/2$ towards infinity.

By definition, the boundary of $\Gamma^{(\alpha)}$ is the set of all infinite geodesics in $\Gamma^{(\alpha)}$ starting from the root $(\varnothing,0)$ (these are called geodesic rays), and it can be canonically embedded into $\N^{\N}$. Due to the transience of Brownian motion on $\Gamma^{(\alpha)}$, there is an a.s.~unique geodesic ray denoted by $W_\infty$ that is visited by $(W(t))_{t\geq 0}$ at arbitrarily large times. We say that $W_\infty$ is the exit ray of Brownian motion on $\Gamma^{(\alpha)}$. The distribution of $W_\infty$ under $P$ yields a probability measure $\nu_{\alpha}$ on $\N^{\N}$. Thanks to \eqref{BM-CTGW}, we have in fact $\nu_{\alpha}=\mu_{\alpha}$, provided we think of both $\mu_{\alpha}$ and $\nu_{\alpha}$ as (random) probability measures on $\N^{\N}$. The statement of~\cref{thm:dim-stable} is then reduced to checking that for every $1<\alpha\leq 2$, $\P$-a.s., $\nu_{\alpha}(\mathrm{d}y)$-a.e.
\begin{equation}
\label{eq:1}
\lim_{r\to \infty} \frac{1}{r}\log \nu_{\alpha}(\mathcal{B}(y,r))=-\beta_{\alpha}\,,
\end{equation}
where $\mathcal{B}(y,r)$ denotes the set of all geodesic rays that coincide with $y$ up to height $r$. 

\medskip
\noindent{\bf Infinite continuous trees.} To prove~\eqref{eq:1}, we will apply the tools of ergodic theory to certain transformations on a space of finite-degree rooted infinite continuous trees that we now describe. We let $\T$ be the set of all pairs $(\Pi,(z_v)_{v\in\Pi})$ that satisfy the following conditions:
\begin{enumerate}
\item[(1)] $\Pi$ is an infinite discrete tree, in the sense of Section~\ref{sec:treedelta}.
\item[(2)] We have
\begin{enumerate}
\item[\rm(i)] $z_{v}\in [0,\infty)$ for all $v \in \Pi$\,;
\item[\rm(ii)] $z_{\hat v}< z_v$ for every $v\in\Pi\backslash\{\varnothing\}$\,;
\item[\rm(iii)] for every $\mathbf{v}\in \Pi_{\infty}\colonequals\{(v_1,v_2,\ldots,v_{n},\ldots)\in \N^{\N}\colon (v_{1},v_{2},\ldots,v_{n})\in \Pi, \forall n\geq 1\}$,
$$\lim_{n\to\infty} z_{(v_1,\ldots,v_n)} =+\infty.$$ 
\end{enumerate}
\end{enumerate}
In the preceding definition, we allow the possibility that $z_\varnothing=0$. Notice that property \rm(iii) implies that $\# \{v\in \Pi\colon z_v\leq r\}< \infty$ for every $r>0$.

We equip $\T$ with the $\sigma$-field generated by the coordinate mappings. If $(\Pi,(z_v)_{v\in\Pi})\in\T$, we can consider the associated ``tree''
$$\t\colonequals \big(\{\varnothing\}\times [0,z_\varnothing]\big) \cup  \bigg(\bigcup_{v\in\Pi\backslash\{\varnothing\}} \{v\} \times (z_{\hat v}, z_v]\bigg),$$
equipped with the distance defined as above. The set $\Pi_{\infty}$ is identified with the collection of all geodesic rays in $\Pi$, and will be viewed as the boundary of the tree $\t$. We keep the notation $H(x)=r$ for the height of a point $x=(v,r)\in\t$. The genealogical order on $\t$ is defined as previously and again is denoted by $\prec$. If $\mathbf{u}=(u_1,u_2,\ldots)\in \Pi_{\infty}$, and $x=(v,r)\in \t$, we write $x\prec \mathbf{u}$ if $v=(u_1,u_2,\ldots,u_k)$ for some integer $k\geq 0$.

We will often abuse notation and say that we consider a tree $\t\in \T$: This means that we are given a pair $(\Pi,(z_v)_{v\in\Pi})$ satisfying the above properties, and we consider the associated tree $\t$. In particular, $\t$ has an order structure (in addition to the genealogical partial order) given by the lexicographical order on $\Pi$. Elements of $\T$ will be called infinite continuous trees. Clearly, for every stable index $\alpha \in [1,2]$, the CTGW tree $\Gamma^{(\alpha)}$ can be viewed as a random variable with values in $\T$, and we write $\Theta_{\alpha}(\mathrm{d}\t)$ for its distribution. 

Let us fix $\t=(\Pi,(z_v)_{v\in\Pi})\in \mathbb{T}$. Under our previous notation, $k_{\varnothing}$ is the number of offspring at the first branching point of $\t$. We denote by $\t_{(1)},\t_{(2)},\ldots,\t_{(k_{\varnothing})}$ the subtrees of $\t$ obtained at the first branching point. To be more precise, for every $1\leq i\leq k_{\varnothing}$, we define the shifted discrete tree $\Pi[i]=\{v\in \mathcal{V}\colon iv\in \Pi\}$, and $\t_{(i)}$ is the infinite continuous tree corresponding to the pair
\begin{displaymath}
\Big(\Pi[i],(z_{iv}-z_{\varnothing})_{v\in \Pi[i]}\Big).
\end{displaymath}
Under $\Theta_{\alpha}(\mathrm{d}\t)$, we know by definition that $k_{\varnothing}$ is distributed according to $\theta_{\alpha}$. Moreover, conditionally on $k_{\varnothing}$, the branching property of the CTGW tree states that the subtrees $\t_{(1)},\ldots,\t_{(k_{\varnothing})}$ are i.i.d.~following the same law $\Theta_{\alpha}$.

If $r>0$, the level set of $\t\in \T$ at height $r$ is
$$\t_r= \{x\in\t \colon H(x)=r\}.$$ 
For $\alpha\in (1,2]$, we have the classical result
$$\E\big[\#\Gamma_r^{(\alpha)}\big]= \exp \big(\frac{r}{\alpha-1}\big)=\exp \big((m_{\alpha}-1)r\big)\,,$$ 
which can be derived from the following identity (see e.g.~Theorem 2.7.1 in~\cite{DLG02}) stating that for every $u>0$,
\begin{displaymath}
\E\big[\exp (-u\,\#\Gamma_r^{(\alpha)})\big]=1-\big[1-e^{-r}(1-(1-e^{-u})^{1-\alpha})\big]^{\frac{1}{1-\alpha}}\,.
\end{displaymath}

\subsection{The continuous conductance}
\label{sec:conductance}
Recall that, for $\alpha \in[1,2]$, the random variable $\mathcal{C}^{(\alpha)}$ is defined as the conductance between the root and the set $\partial \Delta^{(\alpha)}$ in the continuous tree $\Delta\a$ viewed as an electric network. One can also give a more probabilistic definition of the conductance. If $\t$ is a (deterministic) infinite continuous tree, the conductance $\mathcal{C}(\t)$ between the root and the boundary $\partial\t$ can be defined in terms of excursion measures of Brownian motion with drift $1/2$ on $\t$. Under this definition, we can set $\mathcal{C}^{(\alpha)}=\mathcal{C}(\Gamma\a)\in [1,\infty)$. For details, we refer the reader to Section 2.3 in~\cite{CLG13}.

In this subsection, we will prove for $\alpha \in (1,2]$ that the law of $ \mathcal{C}\a$ is characterized by the distributional identity~\eqref{eq:rde} in the class of all probability measures on $[1,\infty)$, and discuss some of the properties of this law. For $u \in (0,1), n\in \N$ and $(x_{i})_{i\geq 1}\in [1,\infty)^{\N}$, we define 
\begin{equation*} 
G(u,n,(x_{i})_{i\geq1}) \colonequals \left( u + \frac{1-u}{x_{1}+x_{2}+\cdots+x_{n}}\right)^{-1}, 
\end{equation*} 
so that \eqref{eq:rde} can be rewritten as 
\begin{equation}
\label{rde-bis}
\mathcal{C}\a \overset{(\mathrm{d})}{=} G(U, N_{\alpha}, (\mathcal{C}\a_{i})_{i\geq 1})
\end{equation} 
where $U, N_{\alpha}, (\mathcal{C}\a_{i})_{i\geq 1}$ are as in~\eqref{eq:rde}. Note that (\ref{rde-bis}) also holds for $\alpha=1$. Let $ \mathscr{M}$ be the set of all probability measures on $[1, \infty]$ and let $\Phi_{\alpha} \colon \mathscr{M} \to \mathscr{M}$ map a distribution $\sigma$ to 
\begin{equation*} 
\Phi_{\alpha}(\sigma) = \mathsf{Law} \big(G(U,N_{\alpha},(X_{i})_{i\geq1})\big) 
\end{equation*} 
where $(X_{i})_{i\geq1}$ are independent and identically distributed according to $\sigma$, while $U, N_{\alpha}$ are as in~\eqref{eq:rde}. We suppose in addition that $U, N_{\alpha}$ and $(X_{i})_{i\geq1}$ are independent.   

We write $\gamma_{\alpha}$ for the distribution of $\mathcal{C}\a$, and define for all $\ell \geq 0$ the Laplace transform 
\begin{displaymath}
\varphi_{\alpha}(\ell) \colonequals \E\big[\exp(-\ell\,\mathcal{C}\a/2)\big] =\int_1^\infty e^{-\ell r/2}\,\gamma_{\alpha}(\mathrm{d}r).
\end{displaymath}

\begin{proposition} 
\label{prop:unique}
Let us fix the stable index $\alpha \in (1,2]$.
The law $\gamma_{\alpha}$ of $\mathcal{C}\a$ is the unique fixed point of the mapping $\Phi_{\alpha}$ on $\mathscr{M} $, and we have  $\Phi_{\alpha}^k(\sigma) \to \gamma_{\alpha}$ weakly as $k \to \infty$, for every $\sigma \in \mathscr{M}$. Furthermore,
\begin{enumerate}
\item\label{item:unique2} If $\alpha=2$, all moments of $\gamma_{2}$ are finite, and $\gamma_{2}$ has a continuous density over $[1,\infty)$. The Laplace transform $\varphi_{2}$ solves the differential equation
\begin{equation*} 
2\ell\,\varphi''(\ell) + \ell \varphi'(\ell) + \varphi^2(\ell) - \varphi(\ell)=0. 
\end{equation*}

\item\label{item:unique} If $\alpha\in(1,2)$, only the first and the second moments of $\gamma_{\alpha}$ are finite. The distribution $\gamma_{\alpha}$ has a continuous density over $[1,\infty)$, and the Laplace transform $\varphi_{\alpha}$ solves the differential equation
\begin{equation}
\label{eq:laplace}
2\ell\,\varphi''(\ell) + \ell \varphi'(\ell) + \frac{(1-\varphi(\ell))^{\alpha}+\varphi(\ell)-1}{\alpha-1}=0.
\end{equation}

\end{enumerate}
\end{proposition}
 
\proof The case $\alpha=2$ has been derived in~\cite[Proposition 6]{CLG13} and is listed above for completeness. We will prove the corresponding assertion for $\alpha\in (1,2)$ by similar methods. 

Firstly, the stochastic partial order $\preceq$ on $ \mathscr{M}$ is defined by saying that $ \sigma \preceq \sigma'$ if and only if there exists a coupling $(X,Y)$ of $\sigma$ and $\sigma'$ such that a.s.~$X \leq Y$. It is clear that for any $\alpha\in [1,2]$, the mapping $\Phi_{\alpha}$ is increasing for the stochastic partial order. 
 
We endow the set $\mathscr{M}_{1}$ of all probability measures on $[1, \infty]$ that have a finite first moment with the $1$-Wasserstein  metric
\begin{eqnarray*} 
\mathrm{d}_{1}( \sigma,\sigma') &\colonequals& \inf\big \{ E\big[|X-Y|\big] \colon  (X,Y)\mbox{ coupling of }(\sigma,\sigma')\big\}.  
\end{eqnarray*} 
The metric space $(\mathscr{M}_{1}, \mathrm{d}_{1})$ is Polish and its topology is finer than the weak topology on $ \mathscr{M}_{1}$. From the easy bound
$$G(u,n,(x_{i})_{i\geq1})\leq x_{1}+x_{2}+\cdots+x_{n}$$
and the fact that $\E N_{\alpha}<\infty$ as $\alpha\neq 1$, we immediately see that $\Phi_{\alpha}$ maps $\mathscr{M}_{1}$ into $ \mathscr{M}_{1}$ when $\alpha>1$. We then observe that the mapping $\Phi_{\alpha}$ is strictly contractant on $(\mathscr{M}_{1}, \mathrm{d}_{1})$. To see this, let $(X_{i},Y_{i})_{i\geq1}$ be independent copies of a coupling between $\sigma,\sigma' \in \mathscr{M}_{1}$ under the probability measure $\mathbb{P}$. As in~(\ref{rde-bis}), let $U$ be uniformly distributed over $[0,1]$ and $N_{\alpha}$ be distributed according to $\theta_{\alpha}$. Assume that $U, N_{\alpha}$ and $(X_{i},Y_{i})_{i\geq1}$ are independent under $\mathbb{P}$. Then the two variables $G(U,N_{\alpha},(X_{i})_{i\geq 1})$ and $G(U,N_{\alpha},(Y_{i})_{i\geq 1})$ give a coupling of $ \Phi_{\alpha}(\sigma)$ and $ \Phi_{\alpha}(\sigma')$.  Using the fact that $X_{i},Y_{i} \geq 1$, we have
\begin{eqnarray*} 
 && \left| G(U,N_{\alpha},(X_{i})_{i\geq 1}) - G(U,N_{\alpha},(Y_{i})_{i\geq 1})  \right|\\
 &=& \Big| \Big( U+ \frac{1-U}{X_1+X_2+\cdots+X_{N_{\alpha}}}\Big)^{-1} -  \Big( U+ \frac{1-U}{Y_1+Y_2+\cdots+Y_{N_{\alpha}}}\Big)^{-1} \Big|\\
 &=& \Big|\frac {(X_1+X_2+\cdots+X_{N_{\alpha}} -Y_1-Y_2-\cdots-Y_{N_{\alpha}})(1-U)}{(U(X_1+X_2+\cdots+X_{N_{\alpha}})+1-U)(U(Y_1+Y_2+\cdots+Y_{N_{\alpha}})+1-U)}\Big|\\
 &\leq &  \big(|X_{1}-Y_{1}|+|X_{2}-Y_{2}|+\cdots+|X_{N_{\alpha}}-Y_{N_{\alpha}}|\big)\, \frac{1-U}{(1+(N_{\alpha}-1)U)^2}.
\end{eqnarray*}
Notice that for any integer $k\geq 2$,
\begin{displaymath}
\mathbb{E}\Big[\frac{k(1-U)}{(1+(k-1)U)^{2}}\Big]=1+\frac{k-1-k\log k}{(k-1)^{2}}.
\end{displaymath}
Taking expected values and minimizing over the choice of the coupling between $\sigma$ and $\sigma'$, we get 
\begin{eqnarray*} 
\mathrm{d}_{1}( \Phi_{\alpha}(\sigma), \Phi_{\alpha}(\sigma'))
&\leq& \mathbb{E}\Big[\frac{N_{\alpha}(1-U)}{(1+(N_{\alpha}-1)U)^{2}}\Big]\mathrm{d}_{1}(\sigma,\sigma')\\
&=& \bigg(1+\mathbb{E}\Big[\frac{N_{\alpha}-1-N_{\alpha}\log N_{\alpha}}{(N_{\alpha}-1)^{2}}\Big]\bigg)\mathrm{d}_{1}(\sigma,\sigma') \;=\; c_{\alpha} \mathrm{d}_{1}(\sigma,\sigma')\,,
\end{eqnarray*}
with $c_{\alpha}<1$. So for $\alpha \in (1,2]$, the mapping $\Phi_{\alpha}$ is contractant on $ \mathscr{M}_{1}$ and by completeness it has a unique fixed point $\tilde\gamma_{\alpha}$ in $ \mathscr{M}_{1}$. Furthermore, for every $ \sigma \in \mathscr{M}_{1}$, we have $ \Phi_{\alpha}^k(\sigma) \to \tilde\gamma_{\alpha}$ for the metric $ \mathrm{d}_{1}$, hence also weakly, as $k \to \infty$.
   
Since we know from \eqref{rde-bis} that $\gamma_{\alpha}$ is also a fixed point of $\Phi_{\alpha}$, the equality $\gamma_{\alpha}=\tilde\gamma_{\alpha}$ will follow if we can verify that $\tilde\gamma_{\alpha}$ is the unique fixed point of $\Phi_{\alpha}$ in $\mathscr{M}$. To this end, it will be enough to show that we have $ \Phi_{\alpha}^k(\sigma) \to \tilde\gamma_{\alpha}$ as $k\to\infty$, for every $\sigma\in \mathscr{M}$. 

For any $\alpha \in [1,2]$, we apply $\Phi_{\alpha}$ to the Dirac measure $\delta_{\infty}$ at $\infty$ to see 
\begin{eqnarray*}
\Phi_{\alpha}(\delta_{\infty}) &= &\mathsf{Law} \big(U^{-1}\big)\,, \\
\Phi_{\alpha}^{2}(\delta_{\infty}) &= &\mathsf{Law} \bigg(\bigg(U+\frac{1-U}{U_{1}^{-1}+U_{2}^{-1}+\cdots+U_{N_{\alpha}}^{-1}}\bigg)^{-1}\bigg)\,,
\end{eqnarray*}
where we introduce a new sequence $(U_{i})_{i\geq 1}$ consisting of i.i.d.~copies of $U$, independent of $N_{\alpha}$ and $U$ under $\mathbb{P}$. Thus the first moment of $\Phi_{\alpha}^{2}(\delta_{\infty})$ is given by
\begin{eqnarray*}
&& \sum\limits_{k\geq 2}\theta_{\alpha}(k)\int_{0}^{1}\mathrm{d}u\int_{0}^{1}\mathrm{d}u_{1}\cdots \int_{0}^{1}\mathrm{d}u_{k} \bigg(u+\frac{1-u}{u_{1}^{-1}+u_{2}^{-1}+\cdots+u_{k}^{-1}}\bigg)^{-1}\\
&=& \sum\limits_{k\geq 2}\theta_{\alpha}(k)\int_{0}^{1}\mathrm{d}u_{1}\cdots \int_{0}^{1}\mathrm{d}u_{k} \frac{1}{1-(u_{1}^{-1}+u_{2}^{-1}+\cdots+u_{k}^{-1})^{-1}}\log \Big(\frac{1}{u_{1}}+\frac{1}{u_{2}}+\cdots+\frac{1}{u_{k}}\Big)\\
&\leq& 2\sum\limits_{k\geq 2}\theta_{\alpha}(k)\int_{0}^{1}\mathrm{d}u_{1}\cdots \int_{0}^{1}\mathrm{d}u_{k} \log \Big(\frac{1}{u_{1}}+\frac{1}{u_{2}}+\cdots+\frac{1}{u_{k}}\Big)\,,
\end{eqnarray*}
in which the integrals can be bounded as follows,
\begin{eqnarray*}
&&\int_{0}^{1}\mathrm{d}u_{1}\cdots \int_{0}^{1}\mathrm{d}u_{k} \log \Big(\frac{1}{u_{1}}+\frac{1}{u_{2}}+\cdots+\frac{1}{u_{k}}\Big) \\
&= &k!\int_{0<u_{1}<u_{2}<\cdots<u_{k}<1}\mathrm{d}u_{1}\mathrm{d}u_{2}\cdots \mathrm{d}u_{k} \log\Big(\frac{1}{u_{1}}+\frac{1}{u_{2}}+\cdots+\frac{1}{u_{k}}\Big) \\
&=& k!\int_{0<u_{2}<u_{3}<\cdots<u_{k}<1}\mathrm{d}u_{2}\mathrm{d}u_{3}\cdots \mathrm{d}u_{k} \bigg[u_{2}\log\Big(\frac{2}{u_{2}}+\frac{1}{u_{3}}+\cdots+\frac{1}{u_{k}}\Big)+\frac{\log\Big(2+\frac{u_{2}}{u_{3}}+\cdots+\frac{u_{2}}{u_{k}}\Big)}{u_{2}^{-1}+u_{3}^{-1}+\cdots+u_{k}^{-1}}\bigg]\\
&\leq& k!\int_{0<u_{2}<u_{3}<\cdots<u_{k}<1}\mathrm{d}u_{2}\mathrm{d}u_{3}\cdots \mathrm{d}u_{k} \Big[u_{2}\log \frac{k}{u_{2}}+\frac{\log k}{k-1}\Big]\\
&=& \log k+\frac{1}{2}+\cdots+\frac{1}{k}+\frac{k\log k}{k-1}\;\leq\; (2+\frac{k}{k-1})\log k\,.
\end{eqnarray*}
Using Stirling's formula, we know that $\theta_{\alpha}(k)=O(k^{-(1+\alpha)})$ as $k\to +\infty$. As
\begin{displaymath}
\sum\limits_{k\geq2} (2+\frac{k}{k-1})\frac{\log k }{k^{1+\alpha}}<+\infty
\end{displaymath} 
for all $\alpha\in [1,2]$, we get $\Phi_{\alpha}^2(\delta_{\infty})\in\mathscr{M}_1$. By monotonicity, we have also $\Phi_{\alpha}^2(\sigma)\in\mathscr{M}_1$ for every $\sigma\in\mathscr{M}$, and from the preceding results we get $\Phi_{\alpha}^k(\sigma) \to \tilde\gamma_{\alpha}$ for every $\sigma\in \mathscr{M}$. This implies that $\gamma_{\alpha}=\tilde\gamma_{\alpha}$ is the unique fixed point of $\Phi_{\alpha}$ in $\mathscr{M}$.

For every $t\in \R$ we set $F_{\alpha}(t)=\gamma_{\alpha}([t,\infty])$. For every integer $k\geq 2$, we write $F^{(k)}_{\alpha}(t) = \mathbb{P}( \cc\a_{1}+\cc\a_{2}+\cdots+\cc\a_{k} \geq t)$, where $(\cc\a_{k})_{k\geq1}$ are independent and identically distributed according to $\gamma_{\alpha}$. Then we have, for every $t>1$,
\begin{eqnarray} 
\label{eq:repart} 
F_{\alpha}(t) &=& \mathbb{P}\bigg ( U + \frac{1-U}{ \cc\a_{1} +\cc\a_{2}+\cdots+\cc\a_{N_{\alpha}}} \leq t^{-1} \bigg) \nonumber \\ 
&=& \mathbb{P}\left( U <  t^{-1} \mbox{ and } \frac{t-Ut}{1-Ut} \leq \cc\a_{1} + \cc\a_{2}+\cdots+\cc\a_{N_{\alpha}}  \right)\nonumber \\ 
&=& \mathbb{E}\bigg[\int_{0}^{1/t} \mathrm{d}u \,F^{(N_{\alpha})}_{\alpha}\left ( \frac{t-ut}{1-ut}\right)\bigg]\nonumber \\
&= & \frac{t-1}{t} \int_{t}^\infty  \frac{\mathrm{d}x}{(x-1)^2} \mathbb{E} \Big[F_{\alpha}^{(N_{\alpha})}(x)\Big].  
\end{eqnarray}
By definition, we have $F^{(k)}_{\alpha}(t)=1$ for every $t\in[1,2]$ and $k\geq 2$. It follows from~\eqref{eq:repart} that
\begin{equation}
\label{eq:valueon[1,2]}
F_{\alpha}(t) = \frac{D\a}{t} + 1-D\a,\qquad \forall t\in[1,2],
\end{equation}
where 
$$D\a= 2-\int_2^{\infty} \frac{\mathrm{d}x}{(x-1)^2} \mathbb{E}\Big[F_{\alpha}^{(N_{\alpha})}(x)\Big] \in [1,2].$$
We observe that the right-hand side of \eqref{eq:repart} is a continuous function of $t\in(1,\infty)$, so that $F_{\alpha}$ is continuous on $[1,\infty)$ (the right-continuity at $1$ is obvious from \eqref{eq:valueon[1,2]}). Thus $\gamma_{\alpha}$ has no atom and it follows that all functions $F^{(k)}_{\alpha}, k\geq 2$ are continuous on $[1,\infty)$. By dominated convergence the function $x\mapsto \E[F_{\alpha}^{(N_{\alpha})}(x)]$ is also continuous on $[1,\infty)$. Using \eqref{eq:repart} again we obtain that $F_{\alpha}$ is continuously differentiable on $[1,\infty)$ and consequently $\gamma_{\alpha}$ has a continuous density $f_{\alpha}=-F'_{\alpha}$ with respect to Lebesgue measure on $[1,\infty)$. 

Let us finally derive the differential equation \eqref{eq:laplace}. To this end, we first differentiate \eqref{eq:repart} with respect to $t$ to get that the linear differential equation   
\begin{equation}
\label{eq:edo}
t(t-1) F'_{\alpha}(t) - F_{\alpha}(t)= -\mathbb{E}\big[F_{\alpha}^{(N_{\alpha})}(t)\big]
\end{equation}
holds for $t\in[1,\infty)$. Then let $g\colon [1,\infty) \to \mathbb{R}_{+}$ be a monotone continuously differentiable function. From the definition of $F_{\alpha}$ and Fubini's theorem, we have
$$ \int_{1}^\infty \mathrm{d}t\, g'(t) F_{\alpha}(t) = \mathbb{E}\big[ g( \mathcal{C}\a)\big] - g(1)$$
and similarly
$$\int_{1}^\infty \mathrm{d}t\, g'(t) \mathbb{E}\big[F_{\alpha}^{(N_{\alpha})}(t)\big] = \mathbb{E}\big[ g( \mathcal{C}\a_{1}+ \mathcal{C}\a_{2}+\cdots+\mathcal{C}\a_{N_{\alpha}})\big] - g(1).$$
We then multiply both sides of \eqref{eq:edo} by $g'(t)$ and integrate for $t$ running from $1$ to $\infty$ to get   
\begin{equation} 
\label{eq:star} 
\mathbb{E}\big[ \mathcal{C}\a_{1}( \mathcal{C}\a_{1}-1) g'( \mathcal{C}\a_{1})\big] + \mathbb{E} \big[ g( \mathcal{C}\a_{1})\big] =  \E\big[g( \mathcal{C}\a_{1}+ \mathcal{C}\a_{2}+\cdots+\mathcal{C}\a_{N_{\alpha}})\big]. 
\end{equation} 
When $g(x)=\exp(- x\ell/2 )$ for $\ell >0$, we readily obtain \eqref{eq:laplace} by using the generating function of $N_{\alpha}$ given in (\ref{eq:gene-fct}). Finally, taking $g(x)=x$ in~(\ref{eq:star}), we get
\begin{displaymath}
\E\big[(\mathcal{C}\a)^{2}\big]=\E\big[N_{\alpha}\big]\E\big[\mathcal{C}\a\big]=\frac{\alpha}{\alpha-1}\E\big[\mathcal{C}\a\big]\,.
\end{displaymath}
Nevertheless, by taking $g(x)=x^{2}$ in~(\ref{eq:star}), we see that the third moment of $\mathcal{C}\a$ is infinite since $\E\big[(N_{\alpha})^{2}\big]=\infty$.
\endproof
 
The arguments of the preceding proof also yield the following lemma in the case $\alpha=1$.
\begin{lemma}
\label{lem:conductance1}
The conductance $\mathcal{C}^{(1)}$ of the tree $\Delta^{(1)}$ satisfies the bound
\begin{equation}
\label{eq:C1mean}
\E\big[\mathcal{C}^{(1)}\big] \leq 2\sum\limits_{k\geq2} (2+\frac{k}{k-1})\frac{\log k }{k(k-1)}<+\infty.
\end{equation} 
Additionally, the Laplace transform $\varphi_{1}$ of the law of $\mathcal{C}^{(1)}$ solves the differential equation
\begin{equation*} 
2\ell\,\varphi''(\ell) + \ell \varphi'(\ell) + (1-\varphi(\ell))\log(1-\varphi(\ell))=0. 
\end{equation*}
\end{lemma}

\begin{proof}
The law of $\mathcal{C}^{(1)}$ is a fixed point of the mapping $\Phi_{1}$ defined via~(\ref{rde-bis}) with $\alpha=1$. By the same monotonicity argument that we used above, it follows that the first moment of $\mathcal{C}^{(1)}$ is bounded above by the first moment of $\Phi_{1}^{2}(\delta_{\infty})$, and the calculation of this first moment in the previous proof leads to the right-hand side of (\ref{eq:C1mean}). 

As an analogue to~(\ref{eq:star}), we have
\begin{equation*} 
\mathbb{E}\big[ \mathcal{C}^{(1)}_{1}( \mathcal{C}^{(1)}_{1}-1) g'( \mathcal{C}^{(1)}_{1})\big] + \mathbb{E} \big[ g( \mathcal{C}^{(1)}_{1})\big] =  \E\big[g( \mathcal{C}^{(1)}_{1}+ \mathcal{C}^{(1)}_{2}+\cdots+\mathcal{C}^{(1)}_{N_{1}})\big].
\end{equation*} 
By taking $g(x)=\exp(- x\ell/2)$ and using (\ref{eq:gene-fct-1}), one can then derive the differential equation satisfied by $\varphi_{1}$.
\end{proof}

\subsection{The reduced stable trees are nested}\label{sec:coupling}
In this short subsection, we introduce a coupling argument to explain how \cref{thm:bdd-dimension} follows from the identity (\ref{eq:beta-value}) in~\cref{prop:dim-formula}. 

Recall the definition of the $\alpha$-offspring distribution $\theta_{\alpha}$. From the obvious fact 
\begin{displaymath}
1-\sum\limits_{i=2}^{k-1}\frac{\alpha}{i-\alpha}<0, \qquad \forall \alpha \in(1,2), k\geq 3,
\end{displaymath} 
one deduces that for all $k\geq 3$, 
\begin{displaymath}
\frac{\mathrm{d}}{\mathrm{d}\alpha}\theta_{\alpha}(k)< 0, \qquad \forall \alpha \in(1,2).
\end{displaymath}
This implies that for every $k\geq 3$, $\theta_{\alpha}([2,k])$ is a strictly increasing function of $\alpha\in(1,2)$. Using the inverse transform sampling, we can construct on a common probability space a sequence of random variables $(N_{\alpha},\alpha \in[1,2])$ such that a.s.
\begin{displaymath}
N_{\alpha_{2}}\geq N_{\alpha_{1}} \quad \mbox{for all } 1\leq\alpha_{2}\leq\alpha_{1}\leq 2.
\end{displaymath}
Then following the same procedure explained in~\cref{sec:treedelta}, we can construct simultaneously all reduced stable trees as a nested family. More precisely, there exists a family of compact $\mathbb{R}$-trees $(\bar \Delta\a,\alpha\in[1,2])$ such that 
\begin{eqnarray*}
\bar\Delta^{(\alpha)} &\overset{(\mathrm{d})}{=}& \Delta\a \quad \mbox{for all } 1\leq\alpha\leq2\,;\\
\bar\Delta^{(\alpha_{1})}&\subseteq& \bar\Delta^{(\alpha_{2})} \quad \mbox{for all } 1\leq\alpha_{2}\leq\alpha_{1}\leq2\,.
\end{eqnarray*}
Consequently, the family of conductances $(\bar{\mathcal{C}}\a,\alpha\in[1,2])$ associated with $(\bar \Delta\a,\alpha\in[1,2])$ is decreasing with respect to $\alpha$. In particular, the mean $\E[\mathcal{C}\a]$ is decreasing with respect to $\alpha$, and it follows from~(\ref{eq:C1mean}) that $(\E[\mathcal{C}\a],\alpha\in[1,2])$ is uniformly bounded by the constant
\begin{equation*}
C_{0}\colonequals 2\sum\limits_{k\geq2} (2+\frac{k}{k-1})\frac{\log k }{k(k-1)}<+\infty.
\end{equation*}

\smallskip
\noindent{\bf Proof of~\cref{thm:bdd-dimension}.} For any $\alpha\in (1,2]$, $\gamma_{\alpha}$ is a probability measure on $[1,\infty)$ and
\begin{eqnarray*}
\iint \gamma_{\alpha}(\mathrm{d}s)\gamma_{\alpha}(\mathrm{d}t)\frac{st}{s+t-1} &\geq & \iint \gamma_{\alpha}(\mathrm{d}s)\gamma_{\alpha}(\mathrm{d}t)\frac{st}{s+t}\\
&\geq & \iint \gamma_{\alpha}(\mathrm{d}s)\gamma_{\alpha}(\mathrm{d}t)\frac{st}{2(s\vee t)}\\
&=& \frac{1}{2}\iint \gamma_{\alpha}(\mathrm{d}s)\gamma_{\alpha}(\mathrm{d}t) (s\wedge t) \;\geq\; \frac{1}{2}\,.
\end{eqnarray*}
So we derive from (\ref{eq:beta-value}) that 
\begin{displaymath}
\beta_{\alpha}\leq \frac{1}{2}\Big(2\big(\E\big[\mathcal{C}\a\big]\big)^{2}-1\Big)\leq \frac{1}{2}\big(2\, C_{0}^{2}-1\big)<\infty.  \tag*{$\square$}
\end{displaymath}

\subsection{Proof of Theorem~\ref{thm:dim-stable}}
\label{sec:stationary+ergodic}

The proof of~\cref{thm:dim-stable} given below will follow the approach sketched in~\cite[Section 5.1]{CLG13}. We will first establish the flow property of harmonic measure (Lemma~\ref{flow-property}), and then find an explicit invariant measure for the environment seen by Brownian motion on the CTGW tree $\Gamma\a$ at the last visit of a vertex of the $n$-th generation (Proposition~\ref{invariantbis}). After that, we will rely on arguments of ergodic theory to complete the proof of~\cref{thm:dim-stable} and that of~\cref{prop:dim-formula}. 

\textbf{Throughout this subsection, we fix the stable index $\alpha\in (1,2]$ once and for all}. For notational ease, we will omit the superscripts and subscripts concerning~$\alpha$ in all the proofs involved. Recall that $ \mathbb{P}$ stands for the probability measure under which the CTGW tree $\Gamma\a$ is defined, whereas Brownian motion with drift $1/2$ on the CTGW tree is defined under the probability measure~$P$.

\subsubsection{The flow property of harmonic measure}
\label{sec:flow}
We fix an infinite continuous tree $\t\in \T$, and write as before $\t_{(1)},\t_{(2)},\ldots,\t_{(k_{\varnothing})}$ for the subtrees of $\t$ at the first branching point. Here we slightly abuse notation by writing $W=(W(t))_{t\geq 0}$ for Brownian motion with drift $1/2$ on $\t$ started from the root. As in Section~\ref{sec:ctgwtree}, $W_\infty$ stands for the exit ray of $W$, and the distribution of $W_\infty$ on the boundary of $\t$ is the harmonic measure of $\t$, denoted as $\nu_{\t}$. Let $K$ be the index such that $W_{\infty}$ ``belongs to'' $\t_{(K)}$ and we write $W'_{\infty}$ for the ray of $\t_{(K)}$ obtained by shifting $W_{\infty}$ at the first branching point of $\t$.

\begin{lemma}
\label{flow-property}
Let $j\in\{1,2,\ldots,k_{\varnothing}\}$. Conditionally on $\{K=j\}$, the law of $W_\infty'$ is the harmonic measure of $\t_{(j)}$. 
\end{lemma}

The proof is similar to that of \cite[Lemma 7]{CLG13} and is therefore omitted. 

\subsubsection{The invariant measure and ergodicity}
\label{sec:ergodic}

We introduce the set 
\begin{displaymath}
\T^* \subseteq \T \times \N^{\N}
\end{displaymath}
of all pairs consisting of a tree $\t\in \T$ and a distinguished geodesic ray $\mathbf{v}$ in $\t$. Given a distinguished geodesic ray $\mathbf{v}=(v_1,v_2,\ldots)$ in $\t$, we let $S(\t,\mathbf{v})$ be obtained by shifting $(\t,\mathbf{v})$ at the first branching point of $\t$, that is 
$$ S(\t,\mathbf{v})= (\t_{(v_1)}, \wt{\mathbf{v}}),$$
where $\wt{\mathbf{v}}=(v_2,v_3,\ldots)$ and $\t_{(v_1)}$ is the subtree of $\t$ rooted at the first branching point that is chosen by $\mathbf{v}$. 

Under the probability measure $\P \otimes P$, we can view $(\Gamma\a,W_\infty)$ as a random variable with values in $\T^*$. We write $\Theta^*_{\alpha}(\mathrm{d}\t \mathrm{d} \mathbf{v})$ for the distribution of $(\Gamma\a,W_\infty)$. The next proposition gives an invariant measure absolutely continuous with respect to $\Theta^*_{\alpha}$ under the shift $S$.

\begin{proposition}
\label{invariantbis}
For every $r\geq 1$, set
$$\kappa_{\alpha}(r) \colonequals  \sum\limits_{k=2}^{\infty} k\theta_{\alpha}(k)\!\int\!\gamma_{\alpha}(\mathrm{d}t_{1})\!\int\! \gamma_{\alpha}(\mathrm{d}t_{2})\cdots \!\int\!\gamma_{\alpha}(\mathrm{d}t_{k})\,\frac{r t_{1}}{r+t_{1}+t_{2}+\cdots+t_{k}-1}.$$
The finite measure $\kappa_{\alpha}(\cc(\t))\Theta_{\alpha}^*(\mathrm{d}\t\mathrm{d}\mathbf{v})$ is invariant under $S$.
\end{proposition}

\noindent \textbf{Remark.} The preceding formula for $\kappa_{\alpha}$ is suggested by the analogous formula in~\cite[Proposition 25]{CLG13} for $\alpha=2$.

\proof 
First notice that the function $\kappa$ is bounded, since for every $r\geq 1$, 
\begin{displaymath}
\kappa(r) \leq  \sum\limits_{k=2}^{\infty} k\theta(k) \int t_1\gamma(\mathrm{d}t_1)  <\infty.
\end{displaymath} 

Let us fix $\t\in\T$, then for any $1\leq i\leq k_{\varnothing}$ and any bounded measurable function $g$ on $\N^{\N}$, the flow property of harmonic measure gives that
$$\int \nu_\t(\mathrm{d}\mathbf{v})\,\mathbf{1}_{\{v_1=i\}}\,g(\wt{\mathbf{v}}) = \frac{\cc(\t_{(i)})}{\cc(\t_{(1)}) +\cdots+ \cc(\t_{(k_{\varnothing})})}\int \nu_{\t_{(i)}}(\mathrm{d}\mathbf{u})\,g(\mathbf{u}).$$
Recall that $\Theta^*(\mathrm{d}\t\,\mathrm{d}\mathbf{v})= \Theta(\mathrm{d}\t)\nu_\t(\mathrm{d}\mathbf{v})$ by construction. Let $F$ be a bounded measurable function on $\T^*$. Using the preceding display, we have
\begin{align}
\label{eq:S}
&\int F\circ S(\t,\mathbf{v})\,\kappa(\cc(\t))\,\Theta^*(\mathrm{d}\t\,\mathrm{d}\mathbf{v})   \\ 
&\quad= \sum\limits_{k=2}^{\infty}\theta (k) \sum\limits_{i=1}^{k}\int\! F(\t_{(i)}, \mathbf{u})\kappa(\cc(\t)) \frac{\cc(\t_{(i)})}{\cc(\t_{(1)}) +\cdots+ \cc(\t_{(k)})}\, \Theta(\mathrm{d}\t\!\mid\! k_{\varnothing}=k)\,\nu_{\t_{(i)}}(\mathrm{d}\mathbf{u}).\notag
\end{align}
Observe that under $\Theta(\mathrm{d}\t\,|\, k_{\varnothing}=k)$, the subtrees $\t_{(1)},\t_{(2)},\ldots, \t_{(k)}$ are independent and distributed according to $\Theta$, and furthermore,
$$\cc(\t)= \Big( U + \frac{1-U}{\cc(\t_{(1)}) +\cdots+ \cc(\t_{(k)})}\Big)^{-1},$$
where $U$ is uniformly distributed over $[0,1]$ and independent of $(\t_{(1)},\t_{(2)},\ldots, \t_{(k)})$. Using these observations, together with a simple symmetry argument, we get that the integral~(\ref{eq:S}) is given by
\begin{align*}
&\sum\limits_{k=2}^{\infty} k\theta(k)\int_0^1 \mathrm{d}x \!\int\!  \Theta(\mathrm{d}\t_{1})\cdots \!\int\! \Theta(\mathrm{d}\t_{k})\!\int\!\nu_{\t_{1}}(\mathrm{d}\mathbf{u})\,F(\t_{1},\mathbf{u}) \hspace{3cm}\\
&\qquad \qquad \times \frac{\cc(\t_{1})}{\cc(\t_{1}) +\cdots + \cc(\t_{k})} \,\kappa\Big(\Big( x + \frac{1-x}{\cc(\t_{1}) +\cdots + \cc(\t_{k})}\Big)^{-1}\Big) \\
&\quad = \int \Theta^*(\mathrm{d}\t_{1}\,\mathrm{d}\mathbf{u})\,F(\t_{1},\mathbf{u})\, \bigg[ \sum\limits_{k=2}^{\infty} k\theta(k)\int_0^1 \mathrm{d}x \!\int\! \Theta(\mathrm{d}\t_{2})\cdots\!\int\! \Theta(\mathrm{d}\t_{k})\\
&\qquad \qquad\times \frac{\cc(\t_{1})}{\cc(\t_{1}) +\cdots + \cc(\t_{k})} \,\kappa\Big(\Big( x + \frac{1-x}{\cc(\t_{1}) +\cdots + \cc(\t_{k})}\Big)^{-1}\Big)\bigg].
\end{align*}
The proof is thus reduced to checking that, for every $r\geq 1$, $\kappa(r)$ is equal to
\begin{equation}
\label{eq:invarbistech2}
\sum\limits_{k=2}^{\infty} k\theta(k)\int_0^1\!\mathrm{d}x \!\int\! \Theta(\mathrm{d}\t_{2})\cdots\!\int\! \Theta(\mathrm{d}\t_{k})\frac{r}{r+ \cc(\t_{2}) +\cdots + \cc(\t_{k})} \kappa\Big(\big( x + \frac{1-x}{r+ \cc(\t_{2}) +\cdots + \cc(\t_{k})}\big)^{-1}\Big).
\end{equation}
To this end, we will reformulate the last expression in the following way. Under the probability measure $\P$, we introduce an i.i.d.~sequence $(\cc_i)_{i\geq0}$ distributed according to $\gamma$, and a random variable $N$ distributed according to $\theta$. In addition, under the same probability measure $\P$, let $U$ be uniformly distributed over $[0,1]$, $(\wt \cc_{i})_{i\geq 0}$ be an independent copy of $(\cc_i)_{i\geq0}$, and $\wt N$ be an independent copy of $N$. We assume that all these random variables are independent. Note that by definition, for every $r \geq 1$,
$$\kappa(r)= \E\Big[ \frac{r \wt N\wt \cc_1}{r+\wt \cc_1+\wt \cc_2+\cdots+\wt\cc_{\tilde N} -1}\Big].$$
It follows that (\ref{eq:invarbistech2}) can be written as
\begin{align*}
&\sum\limits_{k=2}^{\infty} k\theta(k)\E\Bigg[ \frac{r}{r+\cc_2+\cdots+\cc_{k}}\,\frac{\Big(U+\frac{1-U}{r+\cc_2+\cdots+\cc_{k}}\Big)^{-1}\wt N \wt \cc_1} {\Big(U+\frac{1-U}{r+\cc_2+\cdots+\cc_{k}}\Big)^{-1}+\wt \cc_1+\wt\cc_{2}+\cdots+\wt \cc_{\tilde N} -1}\Bigg]\\
&\,=r \sum\limits_{k=2}^{\infty} k\theta(k) \E\bigg[\frac{\wt N \wt \cc_1}{(r+\cc_2+\cdots+\cc_{k})\big(1+(\wt \cc_1+\wt\cc_{2}+\cdots+\wt \cc_{\tilde N} -1)(U+\frac{1-U}{r+\cc_2+\cdots+\cc_{k}})\big)}\bigg]\\
&\,=r \sum\limits_{k=2}^{\infty} k\theta(k) \E\bigg[\frac{ \wt \cc_1+\wt\cc_{2}+\cdots+\wt\cc_{\tilde N}}{(\wt \cc_1+\wt\cc_{2}+\cdots+\wt \cc_{\tilde N} -1)(U(r+\cc_2+\cdots+\cc_{k})+1-U)+r+\cc_2+\cdots+\cc_{k}}\bigg]\\
&\,=r \sum\limits_{k=2}^{\infty} k\theta(k) \E\bigg[\frac{ \wt \cc_1+\wt\cc_{2}+\cdots+\wt\cc_{\tilde N}}{(\wt \cc_1+\wt\cc_{2}+\cdots+\wt \cc_{\tilde N} )(U(r+\cc_2+\cdots+\cc_{k}-1)+1)+(r+\cc_2+\cdots+\cc_{k}-1)(1-U)}\bigg]\\
&\,=r \sum\limits_{k=2}^{\infty} k\theta(k) \E\bigg[\frac{1}{(r+\cc_2+\cdots+\cc_{k}-1)\big(U+\frac{1-U}{\tilde \cc_1+\tilde\cc_{2}+\cdots+\tilde\cc_{\tilde N}}\big)+1}\bigg]\\
&\,=r \sum\limits_{k=2}^{\infty} k\theta(k) \E\bigg[\frac{\widetilde{\cc}}{r+\widetilde{\cc}+\cc_2+\cdots+\cc_{k}-1}\bigg]\,=\,\E\bigg[\frac{r  N \widetilde{\cc}}{r+\widetilde{\cc}+\cc_2+\cdots+\cc_{ N}-1}\bigg],
\end{align*}
where 
$$\widetilde{\cc} \colonequals (U + \frac{1-U}{\wt \cc_1 +\cdots+\wt \cc_{\tilde N}})^{-1}$$ 
is independent of $(\cc_{i})_{i\geq 0}$ and $N$. By~(\ref{eq:rde}), the random variable $\widetilde{\cc}$ is also distributed according to~$\gamma$. So the right-hand side of the last long display is equal to $\kappa(r)$, which completes the proof of the proposition. \endproof

We normalize $\kappa_{\alpha}$ by setting
\begin{displaymath}
\widehat{\kappa}_{\alpha}(r)=\frac{\kappa_{\alpha}(r)}{\int \kappa_{\alpha}(\cc(\mathcal{T}))\Theta^{*}_{\alpha}(\mathrm{d}\mathcal{T}\mathrm{d}\mathbf{v})}= \frac{\kappa_{\alpha}(r)}{\int \kappa_{\alpha}(\cc(\mathcal{T}))\Theta_{\alpha}(\mathrm{d}\mathcal{T})}
\end{displaymath}
for every $r\geq 1$. Then $\widehat \kappa_{\alpha}(\cc(\mathcal{T}))\Theta^{*}_{\alpha}(\mathrm{d}\mathcal{T}\mathrm{d}\mathbf{v}) $ is a probability measure on $\mathbb{T}^{*}$ invariant under the shift $S$. To simplify notation, we set $\Upsilon^{*}_{\alpha}(\mathrm{d}\mathcal{T}\mathrm{d}\mathbf{v})\colonequals\widehat \kappa_{\alpha}(\cc(\mathcal{T}))\Theta^{*}_{\alpha}(\mathrm{d}\mathcal{T}\mathrm{d}\mathbf{v})$. Let $\pi_1$ be the canonical projection from $\mathbb{T}^{*}$ onto $\mathbb{T}$. The image of $\Upsilon^{*}_{\alpha}$ under this projection is the probability measure $\Upsilon_{\alpha}(\mathrm{d}\mathcal{T})\colonequals\widehat \kappa_{\alpha}(\cc(\mathcal{T}))\Theta_{\alpha}(\mathrm{d}\mathcal{T})$.

\begin{proposition}
\label{ergodicbis}
The shift $S$ acting on the probability space $(\mathbb{T}^{*},\Upsilon^{*}_{\alpha})$ is ergodic.
\end{proposition}

\begin{proof}
Our arguments proceed in a similar way as in the proof of \cite[Proposition 13]{CLG13}. We define a transition kernel $\mathbf{p}(\mathcal{T},\mathrm{d}\mathcal{T}')$ on $\mathbb{T}$ by setting
\begin{displaymath}
\mathbf{p}(\mathcal{T},\mathrm{d}\mathcal{T}')=\sum\limits_{i=1}^{k_{\varnothing}} \frac{\mathcal{C}(\mathcal{T}_{(i)})}{\mathcal{C}(\mathcal{T}_{(1)})+\cdots+\mathcal{C}(\mathcal{T}_{(k_{\varnothing})})}\,\delta_{\mathcal{T}_{(i)}}(\mathrm{d}\mathcal{T}').
\end{displaymath}
Informally, under the probability measure $\mathbf{p}(\mathcal{T},\mathrm{d}\mathcal{T}')$, we choose one of the subtrees of $\mathcal{T}$ obtained at the first branching point, with probability equal to its harmonic measure. 

For every integer $n\geq 1$, we denote by $S^n$ the mapping on $\mathbb{T}^{*}$ obtained by iterating $n$ times the shift $S$, and then we consider the process $(Z_n)_{n\geq 0}$ on the probability space $(\mathbb{T}^{*},\Upsilon^{*})$ with values in $\mathbb{T}$, defined by $Z_0(\mathcal{T},\mathbf{v})=\mathcal{T}$ and 
\begin{displaymath}
Z_n(\mathcal{T},\mathbf{v})=\pi_1\big( S^n(\mathcal{T},\mathbf{v})\big)
\end{displaymath}
for every $n\geq 1$. According to Proposition~\ref{invariantbis} and the flow property of harmonic measure, the process $(Z_n)_{n\geq 0}$ is a Markov chain with transition kernel $\mathbf{p}$ under its stationary measure $\Upsilon(\mathrm{d}\mathcal{T})$.

We write $\mathbb{T}^{\infty}$ for the set of all infinite sequences $(\mathcal{T}^0,\mathcal{T}^1,\ldots)$ of elements in $\mathbb{T}$, and let $\widehat{\mathbb{T}}^{\infty}$ be the set of all infinite sequences $(\mathcal{T}^0,\mathcal{T}^1,\ldots)$ in $\mathbb{T}^{\infty}$, such that, for every integer $j\geq 1$, $\mathcal{T}^j$ is one of the subtrees of $\mathcal{T}^{j-1}$ above the first branching point of $\mathcal{T}^{j-1}$. Note that $\widehat{\mathbb{T}}^{\infty}$ is a measurable subset of $\mathbb{T}^{\infty}$ and that $(Z_n(\mathcal{T},\mathbf{v}))_{n\geq 0}\in \widehat{\mathbb{T}}^{\infty}$ for every $(\mathcal{T},\mathbf{v})\in \mathbb{T}^{*}$. If $(\mathcal{T}^0,\mathcal{T}^1,\ldots)\in \widehat{\mathbb{T}}^{\infty}$, there exists a geodesic ray $\mathbf{v}$ in $\mathcal{T}^0$ such that $\mathcal{T}^j=S^j(\mathcal{T}^0,\mathbf{v})$ for every $j\geq 1$, and we set $\phi(\mathcal{T}^0,\mathcal{T}^1,\ldots)\colonequals (\mathcal{T}^0,\mathbf{v})$. Notice that $\mathbf{v}$ is a priori not unique, but to make the previous definition rigorous we can take the smallest possible $\mathbf{v}$ in lexicographical ordering (of course for the random trees that we consider later this uniqueness problem does not arise). In this way, we define a measurable mapping $\phi$ from $\widehat{\mathbb{T}}^{\infty}$ into $\mathbb{T}^*$ such that 
\begin{equation}
\label{eq:Psi-erg}
\phi(Z_0(\mathcal{T},\mathbf{v)},Z_1(\mathcal{T},\mathbf{v}),\ldots)=(\mathcal{T},\mathbf{v}), \qquad \Upsilon^*\mbox{-a.s.}
\end{equation}

Now given a measurable subset $A$ of $\mathbb{T}^{*}$ such that $S^{-1}(A)=A$, we aim at proving that $\Upsilon^{*}(A)\in \{0,1\}$. To this end, we consider the pre-image $B=\phi^{-1}(A)$, which is a measurable subset of $\widehat{\mathbb{T}}^{\infty}\subset \mathbb{T}^{\infty}$. Due to the previous constructions, $B$ is shift-invariant for the Markov chain $Z$ in the sense that 
\begin{displaymath}
\{(Z_0,Z_1,\ldots)\in B\}= \{(Z_1,Z_2,\ldots)\in B\},\quad \hbox{a.s.}
\end{displaymath}
Using Proposition 16.2 in~\cite{LP10}, we then obtain a measurable subset $D$ of $\mathbb{T}$, such that 
$$\mathbf{1}_{B}(Z_0,Z_1,\ldots)=\mathbf{1}_{D}(Z_0) \quad \hbox{a.s.},$$
and moreover $\mathbf{p}(\mathcal{T},D)=\mathbf{1}_D(\mathcal{T})$, $\Upsilon(\mathrm{d}\mathcal{T})$-a.s. It follows thus from~(\ref{eq:Psi-erg}) that $\Upsilon^{*}$-a.s.~we have $(\mathcal{T},\mathbf{v})\in A$ if and only if $\mathcal{T}\in D$.

However from the property $\mathbf{p}(\mathcal{T},D)=\mathbf{1}_D(\mathcal{T})$, $\Upsilon(\mathrm{d}\mathcal{T})$-a.s., one can verify that $\Upsilon(D)\in\{0,1\}$. First note that this property also implies that $\mathbf{p}(\mathcal{T},D)=\mathbf{1}_D(\mathcal{T})$, $\Theta(\mathrm{d}\mathcal{T})$-a.s. Hence, $\Theta(\mathrm{d}\mathcal{T})$-a.s., the tree $\mathcal{T}$ belongs to $D$ if and only if each of its subtrees above the first branching point belongs to $D$ (it is clear that that the measure $\mathbf{p}(\mathcal{T},\cdot)$ assigns a positive mass to each of these subtrees). Then, the branching property of the CTGW tree shows that
$$\Theta(D)=  \sum_{k=2}^\infty \theta(k)\, \Theta(D)^k$$
which is only possible if $\Theta(D)=0$ or $1$, or equivalently if $\Upsilon(D)=0$ or $1$. Therefore $\Upsilon^*(A)$ is either 0 or 1, which completes the proof.
\end{proof}

\subsubsection{Proof of Theorem~\ref{thm:dim-stable}}

Having established Proposition~\ref{invariantbis} and Proposition~\ref{ergodicbis}, we can now apply the ergodic theorem to the two functionals on $\mathbb{T}^{*}$ defined as follows. First let $J_n(\mathcal{T},\mathbf{v})$ denote the height of the $n$-th branching point on the geodesic ray $\mathbf{v}$. One immediately verifies that, for every $n\geq 1$,
$$J_n= \sum_{i=0}^{n-1} J_1\circ S^i. $$
If $M=\int \kappa(\mathcal{C}(\mathcal{T})) \Theta^*(\mathrm{d}\mathcal{T}\mathrm{d}\mathbf{v})$, it follows from the ergodicity that $\Theta^{*}$-a.s.,
\begin{equation}
\label{ergodic1}
\frac{1}{n} J_n \build{\longrightarrow}_{n\to\infty}^{} M^{-1}\int J_1(T,\mathbf{v})\kappa(\mathcal{C}(\mathcal{T})) \Theta^{*}(\mathrm{d}\mathcal{T}\mathrm{d}\mathbf{v}).
\end{equation}
Note that the limit can be written as 
\begin{displaymath}
M^{-1} \E\bigg[|\log(1-U)|\,\kappa \Big( \Big(  U + \frac{1-U}{\cc_1 +\cdots+\cc_N} \Big)^{-1}\Big)\bigg]
\end{displaymath}
with the notation used in the proof of Proposition~\ref{invariantbis}. 

Secondly, let $\mathbf{x}_{n,\mathbf{v}}$ denote the $n+$1-st branching point on the geodesic ray $\mathbf{v}$. If $\mathbf{v}=(v_1,v_2,\ldots)$, then $\mathbf{x}_{n,\mathbf{v}}=((v_1,\ldots,v_n),J_{n+1}(\mathcal{T},\mathbf{v}))$ with the notation of Section~\ref{sec:ctgwtree}. We set for every $n\geq 1$,
$$F_n(\mathcal{T},\mathbf{v}) \colonequals \log \nu_{\mathcal{T}}(\{\mathbf{u}\in\partial \mathcal{T} \colon \mathbf{x}_{n,\mathbf{v}}\prec\mathbf{u}\}).$$
By the flow property of harmonic measure (Lemma~\ref{flow-property}), we have
$$F_n=\sum_{i=0}^{n-1} F_1\circ S^i,$$
and by the ergodic theorem, $\Theta^{*}$-a.s.,
\begin{equation}
\label{ergodic2}
\frac{1}{n} F_n \build{\longrightarrow}_{n\to\infty}^{} M^{-1} \int F_1(\mathcal{T},\mathbf{v}) \kappa(\mathcal{C}(\mathcal{T})) \Theta^{*}(\mathrm{d}\mathcal{T}\mathrm{d}\mathbf{v}),
\end{equation}
where the limit can be written as
\begin{displaymath}
M^{-1}\E\bigg[ \frac{N \cc_1}{\cc_1+\cdots+\cc_N}\,\log\Big(\frac{\cc_1}{\cc_1+\cdots+\cc_N}\Big)\,\kappa \Big(\Big(U + \frac{1-U}{\cc_1 +\cdots+\cc_N}\Big)^{-1}\Big)\bigg].
\end{displaymath}
By combining~(\ref{ergodic1}) and~(\ref{ergodic2}), we obtain that the convergence~(\ref{eq:1}) holds with limit 
\begin{equation*}
-\beta= \frac{ \E\Big[ \frac{N \cc_1}{\cc_1+\cdots+\cc_N}\,\log\Big(\frac{\cc_1}{\cc_1+\cdots+\cc_N}\Big)\,\kappa \Big(\Big(U + \frac{1-U}{\cc_1 +\cdots+\cc_N}\Big)^{-1}\Big)\Big]} {\E\Big[|\log(1-U)|\,\kappa \Big(\Big(U + \frac{1-U}{\cc_1 +\cdots+\cc_N}\Big)^{-1}\Big)\Big]}.
\end{equation*}

\begin{proposition}
\label{prop:beta<}
We have $\beta<\frac{1}{\alpha-1}$.
\end{proposition}

\proof 
We use the notation 
$$\mathcal{W}(\t)=\lim_{r\to \infty} e^{-\frac{r}{\alpha-1}}\,\#\t_r \,,$$
which exists $\Theta(\mathrm{d}\t)$-a.s.~by a martingale argument. Since $\sum \theta(k)k\log k< \infty$, the Kesten-Stigum theorem (for CTGW trees, see e.g.~\cite[Theorem III.7.2]{AN}) implies that the previous convergence holds in the $L^{1}$-sense and $\int \mathcal{W}(\t)\Theta(\mathrm{d}\t) =1$. Moreover, $\Theta(\mathcal{W}(\t)=0)=0$ and the Laplace transform 
\begin{displaymath}
\int e^{-u\mathcal{W}(\t)} \Theta(\mathrm{d}\t) = 1-\frac{u}{(1+u^{\alpha-1})^{\frac{1}{\alpha-1}}} \quad \mbox{ for any } u\in (0,\infty)
\end{displaymath}
can be obtained by applying Theorem III.8.3 in~\cite{AN} together with (\ref{eq:gene-fct}). In particular, it follows from a Tauberian theorem (cf.~\cite[Chapter XIII.5]{F71}) that $\int |\log \mathcal{W}(\t)| \Theta(\mathrm{d}\t)<\infty$.

Let $\t_{(1)},\ldots,\t_{(k_{\varnothing})}$ be the subtrees of $\t$ at the first branching point, and let $J(\t)=J_1(\t,\mathbf{v})$ be the height of the first branching point. Then, $\Theta(\mathrm{d}\t)$-a.s.
\begin{equation*}
\mathcal{W}(\t)=e^{-\frac{J(\t)}{\alpha-1}}\big(\mathcal{W}(\t_{(1)})+\cdots+\mathcal{W}(\t_{(k_{\varnothing})})\big),
\end{equation*}
so that we can define a probability measure $w_{\t}$ on $\{1,2,\ldots,k_{\varnothing}\}$ by setting
\begin{displaymath}
w_{\t}(i)=\frac{e^{-\frac{J(\t)}{\alpha-1}} \mathcal{W}(\t_{(i)})}{\mathcal{W}(\t)}, \qquad 1 \leq i\leq k_{\varnothing}.
\end{displaymath}
On the other hand, for $1\leq i\leq k_{\varnothing}$, let $\nu_{\t}^{*}(i)$ denote the mass assigned by the harmonic measure $\nu_{\t}$ to the rays ``contained'' in $\t_{(i)}$, that is, 
\begin{displaymath}
\nu_{\t}^{*}(i)=\int \mathbf{1}_{\{v_1=i\}}\nu_{\t}(\mathrm{d}\mathbf{v})=\frac{\mathcal{C}(\mathcal{T}_{(i)})}{\mathcal{C}(\mathcal{T}_{(1)})+\cdots+\mathcal{C}(\mathcal{T}_{(k_{\varnothing})})} .
\end{displaymath}
By a concavity argument, 
\begin{equation}
\label{eq:concavity}
\sum\limits_{i=1}^{k_{\varnothing}} \nu_{\t}^{*}(i)\log \frac{w_{\t}(i)}{\nu_{\t}^{*}(i)} \leq 0,
\end{equation}
and the inequality is strict with positive $\Theta$-probability. 

Recall that $\Upsilon(\mathrm{d}\t)=M^{-1}\kappa(\mathcal{C}(\t))\Theta(\mathrm{d}\t)$ is the image of the probability measure $\Upsilon^{*}(\mathrm{d}\t \mathrm{d}\mathbf{v})$ under the canonical projection $\pi_1$ from $\T^{*}$ to $\T$. According to the discussion before Proposition~\ref{prop:beta<}, we can write  
\begin{equation*}
\beta= \Big(\int \Upsilon(\mathrm{d}\t)J(\t)\Big)^{-1} \int \Upsilon(\mathrm{d}\t) \sum\limits_{i=1}^{k_{\varnothing}}\nu_{\t}^{*}(i)\log \frac{1}{\nu_{\t}^{*}(i)}, 
\end{equation*}
which by (\ref{eq:concavity}) is strictly smaller than 
\begin{displaymath}
\Big(\int \Upsilon(\mathrm{d}\t)J(\t)\Big)^{-1}\int \Upsilon(\mathrm{d}\t) \sum\limits_{i=1}^{k_{\varnothing}}\nu_{\t}^{*}(i)\log \frac{1}{w_{\t}(i)}.
\end{displaymath}
However, it follows from the definition of $w_{\t}$ that
\begin{align*}
\int \Upsilon(\mathrm{d}\t) \sum\limits_{i=1}^{k_{\varnothing}}\nu_{\t}^{*}(i)\log \frac{1}{w_{\t}(i)} = & \frac{1}{\alpha-1} \int \Upsilon(\mathrm{d}\t)J(\t) + \int \Upsilon(\mathrm{d}\t) \sum\limits_{i=1}^{k_{\varnothing}}\nu_{\t}^{*}(i)\log \frac{\mathcal{W}(\t)}{\mathcal{W}(\t_{(i)})}\\
=& \frac{1}{\alpha-1} \int \Upsilon(\mathrm{d}\t)J(\t) + \int \Upsilon^{*}(\mathrm{d}\t \mathrm{d}\mathbf{v}) \log \frac{\mathcal{W}\circ \pi_{1}(\t,\mathbf{v})}{\mathcal{W}\circ \pi_1(S(\t,\mathbf{v}))}\\
=& \frac{1}{\alpha-1} \int \Upsilon(\mathrm{d}\t)J(\t)\,,
\end{align*}
where in the last equality we used the fact that $\Upsilon^*$ is invariant under the shift $S$, and that $\log \mathcal{W}(\t)$ is integrable under $\Theta(\mathrm{d}\t)$ hence also under $\Upsilon^{*}$. Therefore, we have shown $\beta<\frac{1}{\alpha-1}$ and the proof of Theorem~\ref{thm:dim-stable} is completed. 
\endproof

\subsubsection{Proof of Proposition~\ref{prop:dim-formula}}

We have seen above that  
\begin{equation}
\label{eq:beta2}
\beta= \frac{ \E\Big[ \frac{N \cc_1}{\cc_1+\cdots+\cc_N}\,\log\Big(\frac{\cc_1}{\cc_1+\cdots+\cc_N}\Big)\,\kappa \Big(\Big(U + \frac{1-U}{\cc_1 +\cdots+\cc_N}\Big)^{-1}\Big)\Big]} {\E\Big[\log(1-U)\,\kappa \Big(\Big(U + \frac{1-U}{\cc_1 +\cdots+\cc_N}\Big)^{-1}\Big)\Big]}.
\end{equation}
On account of Proposition~\ref{prop:unique}, the proof of Proposition~\ref{prop:dim-formula} will be completed if we can verify that the preceding expression for $\beta$ is consistent with formula~(\ref{eq:beta-value}). In the following calculations, we will keep using the same notation introduced in the proof of Proposition~\ref{invariantbis}. 

Firstly, the numerator of the right-hand side of~(\ref{eq:beta2}) is equal to 
\begin{align*}
&\E\bigg[ \frac{N \cc_1}{\cc_1+\cdots+\cc_N}\,\log\Big(\frac{\cc_1}{\cc_1+\cdots+\cc_N}\Big)\,\frac{\big(U + \frac{1-U}{\cc_1 +\cdots+\cc_N}\big)^{-1}(\wt\cc_1 +\cdots+\wt\cc_{\tilde N})}{\big(U + \frac{1-U}{\cc_1 +\cdots+\cc_N}\big)^{-1}+\wt\cc_1 +\cdots+\wt\cc_{\tilde N}-1}\bigg]\\
&\quad=\E\bigg[\frac{N \cc_1 (\wt\cc_1 +\cdots+\wt\cc_{\tilde N})\log\frac{\cc_1}{\cc_1+\cdots+\cc_N}}{\cc_1+\cdots+\cc_N+\wt\cc_1 +\cdots+\wt\cc_{\tilde N}-1+U(\cc_1+\cdots+\cc_N-1)(\wt\cc_1 +\cdots+\wt\cc_{\tilde N}-1)}\bigg].
\end{align*}
For every integer $k\geq 2$, we define for $x\in (1,\infty)$ the function
\begin{displaymath}
G_{c_{1},\ldots,c_{k},u}(x)\colonequals \frac{xc_{1}\log \frac{c_{1}}{c_{1}+\cdots+c_{k}}}{c_{1}+\cdots+c_{k}+x-1+(c_{1}+\cdots+c_{k}-1)(x-1)u}\,,
\end{displaymath}
where $u\in (0,1)$ and $c_{1},\ldots,c_{k} \in (1,\infty)$. We can apply (\ref{eq:star}) to get 
\begin{align*}
&\E\big[G_{\cc_{1},\ldots,\cc_{k},U}(\wt\cc_{1}+\cdots+\wt\cc_{\tilde N})\mid \cc_{1},\ldots,\cc_{k},U\big]\\
&\quad= \E\left[\frac{\cc_{0}^{2}\cc_{1}(\cc_{1}+\cdots+\cc_{k})\log \frac{\cc_{1}}{\cc_{1}+\cdots+\cc_{k}}}{\big(\cc_{0}+\cc_{1}+\cdots+\cc_{k}-1+(\cc_{0}-1)(\cc_{1}+\cdots+\cc_{k}-1)U\big)^{2}} \,\middle|\, \cc_{1},\ldots,\cc_{k},U\right].
\end{align*}
With help of the last display, the numerator of the right-hand side of~(\ref{eq:beta2}) becomes
\begin{displaymath}
\E\bigg[\frac{N\cc_{0}^{2}\cc_{1}(\cc_{1}+\cdots+\cc_{N})\log \frac{\cc_{1}}{\cc_{1}+\cdots+\cc_{N}}}{\big(\cc_{0}+\cc_{1}+\cdots+\cc_{N}-1+(\cc_{0}-1)(\cc_{1}+\cdots+\cc_{N}-1)U\big)^{2}}\bigg].
\end{displaymath}
We now integrate with respect to $U$ and recall that for $a,b,c>0$, $\int_{0}^{1} \mathrm{d}u \frac{a}{(b+cu)^{2}}=\frac{a}{b(b+c)}$. So the numerator of the right-hand side of~(\ref{eq:beta2}) coincides with 
\begin{displaymath}
\E\bigg[\frac{N\cc_{0}\cc_{1}\log \frac{\cc_{1}}{\cc_{1}+\cc_{2}+\cdots+\cc_{N}}}{\cc_{0}+\cc_{1}+\cdots+\cc_{N}-1}\bigg].
\end{displaymath}

On the other hand, the denominator of the right-hand side of~(\ref{eq:beta2}) is equal to
\begin{align*}
&\E\bigg[\frac{(\cc_{1}+\cdots+\cc_{N})(\wt\cc_{1}+\cdots+\wt\cc_{\tilde N})\log(1-U)}{\cc_{1}+\cdots+\cc_{N}+(\wt\cc_{1}+\cdots+\wt\cc_{\tilde N}-1)(\cc_{1}+\cdots+\cc_{N})U+(\wt\cc_{1}+\cdots+\wt\cc_{\tilde N}-1)(1-U)}\bigg]\\
&\quad=\,\, \E\bigg[\frac{(\cc_{1}+\cdots+\cc_{N})(\wt\cc_{1}+\cdots+\wt\cc_{\tilde N})\log(1-U)}{\cc_{1}+\cdots+\cc_{N}+\wt\cc_{1}+\cdots+\wt\cc_{\tilde N}-1+(\cc_{1}+\cdots+\cc_{N}-1)(\wt\cc_{1}+\cdots+\wt\cc_{\tilde N}-1)U}\bigg]\\
&\quad=\,\, \E\bigg[\frac{\cc_{0}^{2}(\cc_{1}+\cdots+\cc_{N})^{2}\log(1-U)}{\big(\cc_{0}+\cc_{1}+\cdots+\cc_{N}-1+(\cc_{0}-1)(\cc_{1}+\cdots+\cc_{N}-1)U\big)^{2}}\bigg]\\
&\quad=\,\, -\,\E\bigg[\frac{\cc_{0}^{2}\cc_{1}^{2}\big(-1+\cc_{0}+\cc_{1}-2\cc_{0}\cc_{1}+(\cc_{0}-1)(\cc_{1}-1)U\big)\log(1-U)}{\big(\cc_{0}+\cc_{1}-1+(\cc_{0}-1)(\cc_{1}-1)U\big)^{3}}\bigg],
\end{align*}
where we have repeatedly used~(\ref{eq:star}) in the last two equalities, the first time to replace $\wt\cc_{1}+\cdots+\wt\cc_{\tilde N}$ by $\cc_{0}$, the second time to replace $\cc_{1}+\cdots+\cc_{N}$ by $\cc_{1}$. In order to integrate with respect to $U$, we appeal to the identity that for $a,b,c>0$,
\begin{displaymath}
\int_{0}^{1}\mathrm{d}u \frac{(a+bu)\log(1-u)}{(c+bu)^{3}}=\frac{b(c-a)+(2b+c+a)c\log \frac{c}{b+c}}{2bc(b+c)^{2}}.
\end{displaymath}
Applying this formula, we see that the denominator of the right-hand side of~(\ref{eq:beta2}) coincides with 
\begin{displaymath}
-\E\Big[\frac{\cc_{0}\cc_{1}}{\cc_{0}+\cc_{1}-1}\Big].
\end{displaymath}
We have thus obtained the following formula
\begin{equation}
\label{eq:beta1}
\beta=\frac{\E\Big[\frac{N\cc_{0}\cc_{1}}{\cc_{0}+\cc_{1}+\cdots+\cc_{N}-1}\log \frac{\cc_{1}+\cc_{2}+\cdots+\cc_{N}}{\cc_{1}}\Big]}{\E\Big[ \frac{ \mathcal{C}_{0} \mathcal{C}_{1}}{ \mathcal{C}_{0}+ \mathcal{C}_{1}-1} \Big]}.
\end{equation}

By a symmetry argument, the numerator of the right-hand side of~(\ref{eq:beta1}) is equal to
\begin{align}
& \E\Big[\frac{N\cc_{0}\cc_{1}\log (\cc_{1}+\cc_{2}+\cdots+\cc_{N})}{\cc_{0}+\cc_{1}+\cdots+\cc_{N}-1}\Big]-\E\Big[\frac{N\cc_{0}\cc_{1}\log (\cc_{1})}{\cc_{0}+\cc_{1}+\cdots+\cc_{N}-1}\Big] \nonumber \\
& \quad=  \E\Big[\frac{\cc_{0}(\cc_{1}+\cc_{2}+\cdots+\cc_{N})\log (\cc_{1}+\cc_{2}+\cdots+\cc_{N})}{\cc_{0}+\cc_{1}+\cdots+\cc_{N}-1}\Big]-\E\Big[\frac{\cc_{0}(\cc_{1}+\cc_{2}+\cdots+\cc_{N})\log (\cc_{0})}{\cc_{0}+\cc_{1}+\cdots+\cc_{N}-1}\Big] \nonumber \\ 
& \quad= \E\big[f(\cc_{1}+\cc_{2}+\cdots+\cc_{N})\big] - \E\big[g(\cc_{1}+\cc_{2}+\cdots+\cc_{N})\big], \label{eq:tech-fg}
\end{align}
where we have set, for every $x\geq 1$,
$$ f(x) = \E \left[ \frac{ \mathcal{C}_0x}{ \mathcal{C}_0+x-1}\log x\right] \quad \mbox{ and }\quad g(x) = \E \left[ \frac{ \mathcal{C}_0x}{ \mathcal{C}_0+x-1}\log \mathcal{C}_0\right]. $$ 
We can replace $\E[f(\cc_{1}+\cc_{2}+\cdots+\cc_{N})]$ by $\E[f(\mathcal{C}_{1})] + \E[ \mathcal{C}_{1}( \mathcal{C}_{1}-1)f'( \mathcal{C}_{1})] $ using \eqref{eq:star}, and similarly for $g$, to obtain 
\begin{displaymath} 
\E\big[f(\cc_{1}+\cc_{2}+\cdots+\cc_{N})\big] - \E\big[g(\cc_{1}+\cc_{2}+\cdots+\cc_{N})\big] =\ \frac{1}{2} \left( \E[ \mathcal{C}_{0}]^2 - \E \left[ \frac{ \mathcal{C}_{0} \mathcal{C}_{1}}{ \mathcal{C}_{0}+ \mathcal{C}_{1}-1} \right] \right) .
\end{displaymath}
Plugging this into~(\ref{eq:tech-fg}) yields the required formula~(\ref{eq:beta-value}), and hence finishes the proof of Proposition~\ref{prop:dim-formula}.

\subsection{A second approach to Theorem~\ref{thm:dim-stable}}
\label{sec:second-appr-conti}
In this section, we outline a different approach to Theorem~\ref{thm:dim-stable}, which contains certain intermediate results of independent interest. This approach involves an invariant measure for the environment seen by Brownian motion on the CTGW tree $\Gamma\a$ at the last visit of a fixed height. This is similar to Section~3 of~\cite{CLG13}, and for this reason we will leave the proofs to Section~\ref{sec:appendix}.

We fix the index $\alpha\in(1,2]$, and we first introduce some additional notation. For $\t\in \T$ and $r>0$, if $x\in\t_r$, let $\t[x]$ denote the subtree of descendants of $x$ in $\t$. To define it formally, we write $v_x$ for the unique element of $\v$ such that $x=(v_x,r)$, and define the shifted discrete tree $\Pi[v_{x}]=\{v \in \mathcal{V} \colon v_{x}v\in \Pi\}$. Then $\t[x]$ is the infinite continuous tree corresponding to the pair 
$$\Big(\Pi[v_{x}], (z_{v_xv}-r)_{v\in\Pi[v_{x}]}\Big).$$

For a fixed $r>0$, we know that $\Gamma\a$ has a.s.~no branching point at height $r$. As there is a unique point $x\in \Gamma\a_r$ such that $x\prec W_\infty$, we write $\Gamma\a \langle r\rangle = \Gamma\a[x]$ for the subtree above level~$r$ selected by the harmonic measure.

To describe the distribution of $\Gamma\a \langle r\rangle$, recall that for every $x\geq 0$,
$$\varphi_{\alpha}(x)=\E\big[\exp(-x\,\mathcal{C}\a/2)\big]= \Theta_{\alpha}\big(\exp(-x\,\cc(\t)/2)\big).$$

\begin{proposition}
\label{prop:tree-selected-law} 
The distribution under $\P \otimes P$ of the subtree $\Gamma^{(\alpha)}\langle r\rangle$ above level $r$ selected by the harmonic measure is
\begin{displaymath}
\Phi\a_r(c)\,\Theta_{\alpha}(\mathrm{d}\t),
\end{displaymath}
where, for every $c>0$, 
$$\Phi\a_r(c)\colonequals  E_{(c)}\Big[\exp-\int_0^r \mathrm{d}s\,\Big(m_{\alpha}\big(1-\varphi_{\alpha}(X_s)\big)^{\alpha-1}-\frac{1}{\alpha-1}\Big)\Big].$$
Here $X=(X_s)_{0\leq s\leq r}$ stands for the solution of the stochastic differential equation 
\begin{displaymath}
\mathrm{d}X_s = 2\sqrt{X_s} \,\mathrm{d}\eta_s + (2- X_s)\mathrm{d}s
\end{displaymath}
that starts under the probability measure $P_{(c)}$ with an exponential distribution of parameter $c/2$. In the previous SDE, $(\eta_s)_{s\geq 0}$ denotes a standard linear Brownian motion.
\end{proposition}

Now we define shifts $(\tau_r)_{r\geq 0}$ on $\T^*$ in the following way. For $r=0$, $\tau_0$ is the identity mapping of $\T^*$. For $r>0$ and $(\t,\mathbf{v})\in\T^*$, we write $\mathbf{v}=(v_1,v_2,\ldots)$ and $\mathbf{v}_n=(v_1,\ldots,v_n)$ for every $n\geq 0$ (by convention, $\mathbf{v}_{0}=\varnothing$). Also let $x_{r,\mathbf{v}}$ be the unique element of $\t_r$ such that $x_{r,\mathbf{v}}\prec \mathbf{v}$. Then we set
$$\tau_r(\t,\mathbf{v})= \Big(\t[x_{r,\mathbf{v}}]\,,\,(v_{k+1},v_{k+2},\ldots)\Big),$$
where $k=\min\{n\geq 0\colon z_{\mathbf{v}_n}\geq r\}$. Informally,  $\tau_r(\t,\mathbf{v})$ is obtained by taking the subtree of $\t$ consisting of descendants of the vertex at height $r$ on the distinguished geodesic ray, and keeping in this subtree the ``same'' geodesic ray. It is straightforward to verify that $\tau_r\circ \tau_s=\tau_{r+s}$ for every $r,s\geq 0$. 

The next proposition gives an invariant measure absolutely continuous with respect to $\Theta_{\alpha}^*$ under the shifts $\tau_{r}$. To simplify notation, we set first
\begin{equation*}
C_1(\alpha)\colonequals 2\int_0^\infty \mathrm{d}s\,\varphi_{\alpha}'(s)^2\,e^{s/2} = \int\!\!\int \gamma_{\alpha}(\mathrm{d}\ell)\gamma_{\alpha}(\mathrm{d}\ell') \frac{\ell \ell'}{\ell+\ell'-1}.
\end{equation*}

\begin{proposition}
\label{invariant-meas}
For every $c>0$,
$$\lim_{r\to +\infty} \Phi\a_r(c)= \Phi\a_\infty(c)\colonequals \frac{1}{C_1(\alpha)} \int \gamma_{\alpha}(\mathrm{d}s)\frac{cs}{c+s-1}. $$
The probability measure $\Lambda_{\alpha}^{*}$ on $\T^{*}$ defined as
$$\Lambda_{\alpha}^*(\mathrm{d}\t\mathrm{d}\mathbf{v})\colonequals \Phi\a_\infty(\cc(\t))\,\Theta_{\alpha}^*(\mathrm{d}\t \mathrm{d}\mathbf{v})$$
is invariant under the shifts $\tau_r$, $r\geq 0$. 
\end{proposition}

Furthermore, one can easily adapt the proof of Proposition 13 in~\cite{CLG13} to show that for every $r>0$, the shift $\tau_r$ acting on the probability space $(\T^*, \Lambda_{\alpha}^*)$ is ergodic. Applying Birkhoff's ergodic theorem to a suitable functional (see Section 3.4 of~\cite{CLG13}) leads to the convergence~(\ref{eq:loc-dim-har}) in Theorem~\ref{thm:dim-stable}, with $\beta_{\alpha}$ given by formula~(\ref{eq:beta1}). See Section~\ref{sec:appendix} for more details.

\section{The discrete setting}
\renewcommand{\t}{\mathsf{T}}

\subsection{Galton-Watson trees}
\label{sec:tree-discrete}

Let us first introduce discrete (finite) rooted ordered trees, which are also called plane trees in combinatorics. A plane tree $\mathsf{t}$ is a finite subset of $\mathcal{V}$ such that the following holds:
\begin{enumerate}
\item[(1)] $\varnothing\in \mathsf{t}\,$.
\item[(2)] If $u=(u_1,\ldots,u_n)\in \mathsf{t}\backslash\{\varnothing\}$, then $\widehat u =(u_1,\ldots,u_{n-1})\in \mathsf{t}\,$.
\item[(3)] For every $u=(u_1,\ldots,u_n)\in \mathsf{t}$, there exists an integer $k_u(\mathsf{t})\geq 0$ such that, for every $j\in\N$,  $(u_1,\ldots,u_n,j)\in \mathsf{t}$ if and only if $1\leq j\leq k_u(\mathsf{t})$.
\end{enumerate}
In this section we will say tree instead of plane tree for short. The same notation and terminology introduced at the beginning of \cref{sec:treedelta} will be used in this section: $|u|$ is the generation of~$u$, $uv$ denotes the concatenation of $u$ and $v$, $\prec$ stands for the genealogical order and $u\wedge v$ is the maximal element of $\{w\in\mathcal{V}\colon w\prec u\hbox{ and }w\prec v\}$. A vertex with no child is called a leaf.

The height of a tree $\mathsf{t}$ is
$$h(\mathsf{t})\colonequals \max\{|v| \colon v\in \mathsf{t}\}.$$
We write $\mathscr{T}$ for the set of all trees, and $\mathscr{T}_n$ for the set of all trees with height $n$.

We view a tree $\mathsf{t}$ as a graph whose vertices are the elements of $\mathsf{t}$ and whose edges are the pairs $\{\widehat u,u\}$ for all $u\in  \mathsf{t}\backslash\{\varnothing\}$. The set $\mathsf{t}$ is equipped with the distance
$$ d(u,v) \colonequals \frac{1}{2}(|u|+|v| -2|u\wedge v|).$$
Notice that this is half the usual graph distance. We will write $B_{\mathsf{t}}(v,r)$, or simply $B(v,r)$ if there is no ambiguity, for the closed ball of radius $r$ centered at $v$, with respect to the distance $d$ in the tree $\mathsf{t}$. 

The set of all vertices of $\mathsf{t}$ at generation $n$ is denoted by
$$\mathsf{t}_n\colonequals \{v\in \mathsf{t} \colon |v|=n\}.$$
If $v\in \mathsf{t}$, the subtree of descendants of $v$ is 
$$\wt{\mathsf{t}} [v]\colonequals \{v'\in \mathsf{t}\colon v\prec v'\}.$$
Note that $\wt{\mathsf{t}}[v]$ is not a tree under the previous definition, but we can turn it into a tree by relabeling its vertices as
$$\mathsf{t}[v] \colonequals \{w\in \mathcal{V} \colon vw\in \mathsf{t}\}.$$

If $v\in \mathsf{t}$, then for every $i\in\{0,1,\ldots,|v|\}$ we write $\langle v\rangle_i$ for the ancestor of $v$ at generation~$i$. Suppose that $|v|=n$. Then $B_{\mathsf{t}}(v,i)\cap \mathsf{t}_n= \wt{\mathsf{t}}\,[\langle v\rangle_{n-i}]\cap \mathsf{t}_n$, for every $i\in\{0,1,\ldots,n\}$. This simple observation will be used repeatedly below.

Let $\rho$ be a non-trivial probability measure on $\Z_{+}$ with mean one, which belongs to the domain of attraction of a stable distribution of index $\alpha\in(1,2]$. Therefore property (\ref{eq:stable-attraction}) holds. For every integer $n\geq 0$, we let $\t^{(n)}$ be a Galton-Watson tree with offspring distribution $\rho$, conditioned on non-extinction at generation $n$, viewed as a random subset of $\mathcal V$ (see e.g.~\cite{LG05} for a precise definition of Galton-Watson trees).  In particular, $\t^{(0)}$ is just a Galton-Watson tree with offspring distribution $\rho$. We suppose that the random trees $\t^{(n)}$ are defined under the probability measure $\P$. 

We let $\t^{*n}$ be the reduced tree associated with $\t^{(n)}$, which consists of all vertices of $\t^{(n)}$ that have (at least) one descendant at generation $n$. Note that $|v|\leq n$ for every $v\in \t^{*n}$. A priori $\t^{*n}$ is not a tree in the sense of the preceding definition. However we can relabel the vertices of  $\t^{*n}$, preserving both the lexicographical order and the genealogical order, so that $\t^{*n}$ becomes a tree in the sense of our definitions. We will always assume that this relabeling has been done.

Conditionally on $\t^{(n)}$, the hitting distribution of generation $n$ is the same for simple random walk on $\t^{(n)}$ and that on the reduced tree $\t^{*n}$. In view of studying properties of this hitting distribution, we can consider directly a simple random walk on $\t^{*n}$ starting from the root $\varnothing$, which we denote by $Z^n=(Z^n_k)_{k\geq 0}$. This random walk is defined under the probability measure $P$. Let
$$H_n\colonequals \inf\{k\geq 0\colon |Z^n_k|=n\}$$
be the first hitting time of generation $n$ by $Z^n$, and set $\Sigma_n = Z^n_{H_n}$ to be the hitting point. The discrete harmonic measure $\mu_n$ is the law of $\Sigma_n$ under $P$, which is a (random) probability measure on the level set $\t^{*n}_{n}$.

Set $q_{n}=\P\big(h(\t^{(0)})\geq n\big)$. If $L$ is the slowly varying function appearing in (\ref{eq:stable-attraction}), it has been established in~\cite[Lemma~2]{S68} that
\begin{equation}
\label{eq:survivalpro1}
q_{n}^{\alpha-1}L(q_{n})\sim \frac{1}{(\alpha-1)n} \quad \mbox{ as } n \to \infty.
\end{equation}
By the asymptotic inversion property of slowly varying functions (see e.g.~\cite[Section 1.5.7]{BGT87}), it follows that
\begin{equation}
\label{eq:survivalpro}
q_{n}\sim n^{-\frac{1}{\alpha-1}}\ell(n) \quad \mbox{ as } n \to \infty,
\end{equation} 
for a function $\ell$ slowing varying at $\infty$. Moreover, it is shown in~\cite[Theorem 1]{S68} that, as $n\to \infty$, $q_{n}\#\t^{*n}_{n}$ converges in distribution to the positive random variable $\mathcal{W}(\Gamma\a)$ introduced in the proof of~\cref{prop:beta<}.

We will need to estimate the size of level sets in $\t^{*n}$. The following lemma is an analogue of Lemma 15 in~\cite{CLG13}. 

\begin{lemma}
\label{lem:level-size}
For every $r\geq 1$, there exists a constant $C=C(r,\rho)$ depending on $r$ and the offspring distribution $\rho$ such that, for every integer $n\geq 2$ and every integer $p\in[1,n/2]$,
$$\E\big[(\log \#\t^{*n}_{n-p})^{r}\big]^{\frac{1}{r}} \leq C\,\log\frac{n}{p} \qquad \mbox{and} \qquad  \E\big[(\log \#\t^{*n}_{n})^{r}\big]^{\frac{1}{r}} \leq C\,\log n.$$
\end{lemma}

\begin{proof}
We can find $a=a(r)>0$ such that the function $x\mapsto (\log(a+x))^r$ is concave over $[1,\infty)$. Then as in the proof of~\cite[Lemma 15]{CLG13}, 
\begin{displaymath}
\E\big[(\log \#\t^{*n}_{n-p})^{r}\big]^{\frac{1}{r}}\leq \E\big[(\log (a+\#\t^{*n}_{n-p}))^{r}\big]^{\frac{1}{r}} \leq \log \big(a+\E[\#\t^{*n}_{n-p}] \big) =\log (a+\frac{q_p}{q_n}).
\end{displaymath}
Using Potter's bounds on slowly varying function (see e.g.~\cite[Theorem 1.5.6]{BGT87}), one can deduce from~(\ref{eq:survivalpro}) that there exists a constant $C'=C'(\rho)>0$ such that for every $n\geq 2 $ and every $p\in [1, n/2]$, 
\begin{displaymath}
\log\Big(\frac{q_{p}}{q_{n}}\Big)\leq C'\log\Big(\frac{n}{p}\Big),
\end{displaymath} 
from which the first bound of the lemma easily follows. The second estimate can be shown in a similar way.
\end{proof}

\smallskip

The goal of this section is to prove Theorem~\ref{thm:dim-discrete}. We will assume in the rest of this section that the critical offspring distribution $\rho$ satisfies~(\ref{eq:stable-attraction}) with a fixed $\alpha \in(1,2]$. Accordingly, we will omit the superscripts and subscripts concerning $\alpha$ if there is no ambiguity. 

\subsection{Convergence of discrete reduced trees}
\label{sec:conv-reduced-tree}
We first define truncations of the discrete reduced tree $\t^{*n}$. For every $s\in[0,n]$, we set
$$R_s(\t^{*n})\colonequals \big\{v\in \t^{*n}\colon |v|\leq n-\lfloor s\rfloor\big\}.$$

Recall from \cref{sec:treedelta} the definition of the continuous reduced tree ${\Delta}$ of index $\alpha$. For every $\ve\in(0,1)$, we have set $\Delta_\ve=\{x\in\Delta\colon H(x)\leq 1-\ve\}$. We will implicitly use the fact that, for every fixed $\ve$, there is a.s.~no branching point of $\Delta$ at height $1-\ve$. The skeleton of $\Delta_\ve$ is defined as the following plane tree
$$\hbox{Sk}(\Delta_\ve)\colonequals \{\varnothing\}\cup\big\{v\in\Pi\backslash\{\varnothing\}\colon Y_{\hat v}\leq 1-\ve\big\}= \{\varnothing\}\cup \big\{v\in\Pi\backslash\{\varnothing\}\colon (\widehat v,Y_{\hat v})\in \Delta_\ve\big\}.$$ 
A vertex $v$ of $\hbox{Sk}(\Delta_\ve)$ is a leaf of $\hbox{Sk}(\Delta_\ve)$ if and only if $Y_{v}>1-\varepsilon$.

Let $\mathsf{t}$ be a tree. We write $\mathcal{S}(\mathsf{t})$ for the set of all vertices of $\mathsf{t}$ whose number of children is different from 1. Then we can find a unique tree $[\mathsf{t}]\in \mathscr{T}$ such that there exists a bijection from $[\mathsf{t}]$ onto $\mathcal{S}(\mathsf{t})$ that preserves the genealogical order and the lexicographical order of vertices. Denote the inverse of this canonical bijection by $u\in \mathcal{S}(\mathsf{t}) \mapsto [u]\in [\mathsf{t}]$. In a less formal way, $[\mathsf{t}]$ is just the tree obtained from $\mathsf{t}$ by removing all vertices that have exactly one child. 

\begin{proposition}
\label{conv-reduced-tree}
We can construct the reduced trees $\t^{*n}$ and the (continuous) reduced stable tree $\Delta$ on the same probability space $(\Omega,\mathcal{F},\P)$, so that the following assertions hold for every fixed $\ve\in(0,1)$ with $\P$-probability one. 
\begin{enumerate}
\item[\rm(1)] For every sufficiently large integer $n$, there exists an injective mapping $\Psi_{n}^{\ve}\colon u \mapsto w^{n,\ve}_{u}$ from $\mathrm{Sk}(\Delta_{\varepsilon})$ into $\mathcal{S}(R_{\varepsilon n}( \t^{*n}))$ satisfying the following properties.

\begin{enumerate}
\item[\rm(1.a)] The mapping $\Psi_{n}^{\ve}$ preserves both the lexicographical order and the genealogical order.
\item[\rm(1.b)] If $u$ is a leaf of $\mathrm{Sk}(\Delta_{ \varepsilon})$, $[w^{n,\ve}_{u}]$ is a leaf of $[R_{\ve n}(\t^{*n})]$ and $|w^{n,\ve}_{u}|=n-\lfloor \ve n \rfloor$. The restricted mapping 
$$\Psi_{n}^{\ve}\restriction_{\sf Leaves}\,\colon \mbox{Leaves of }\mathrm{Sk}(\Delta_{ \varepsilon}) \longrightarrow \big\{v \in \mathcal{S}(R_{\varepsilon n}( \t^{*n}))\colon [v] \mbox{ is a leaf of }[R_{\ve n}(\t^{*n})]\big\}$$
is bijective.
\item[\rm(1.c)] For every vertex $u$ of $\mathrm{Sk}(\Delta_{ \varepsilon})$, 
\begin{eqnarray*}
\lim_{n \to \infty}\frac{1}{n} |w^{n,\ve}_u|&= &Y_{u} \wedge (1-\ve)\,,  \\
\lim_{n \to \infty}\frac{1}{n} |\overline w^{n,\ve}_u| &=& Y_{\hat u}\,,
\end{eqnarray*}
where $\widehat u$ denotes the parent of $u$ in $\mathrm{Sk}(\Delta_{\varepsilon})$, and $\overline w^{n,\ve}_u$ stands for the vertex in $\mathcal{S}(R_{\varepsilon n}( \t^{*n}))$ such that $[\overline w^{n,\ve}_u]$ is the parent of $[w^{n,\ve}_u]$ in $[R_{\varepsilon n}( \t^{*n})]$. (Notice that $\overline w^{n,\ve}_u$ does not necessarily coincide with $w^{n,\ve}_{\hat u}$.)
\end{enumerate}

\item[\rm(2)] The mapping $\Psi_{n}^{\ve}$ is asymptotically unique in the sense that, if $\widetilde \Psi_{n}^{\ve}$ is another mapping such that the preceding properties hold, then for $n$ sufficiently large,
\begin{displaymath}
\Psi_{n}^{\ve}(u)=\widetilde \Psi_{n}^{\ve}(u) \quad \mbox{ for every } u\in \mathrm{Sk}(\Delta_{\varepsilon}).
\end{displaymath}
\end{enumerate}
\end{proposition}

\begin{figure}[!h]
\begin{center}
\includegraphics[width=17cm]{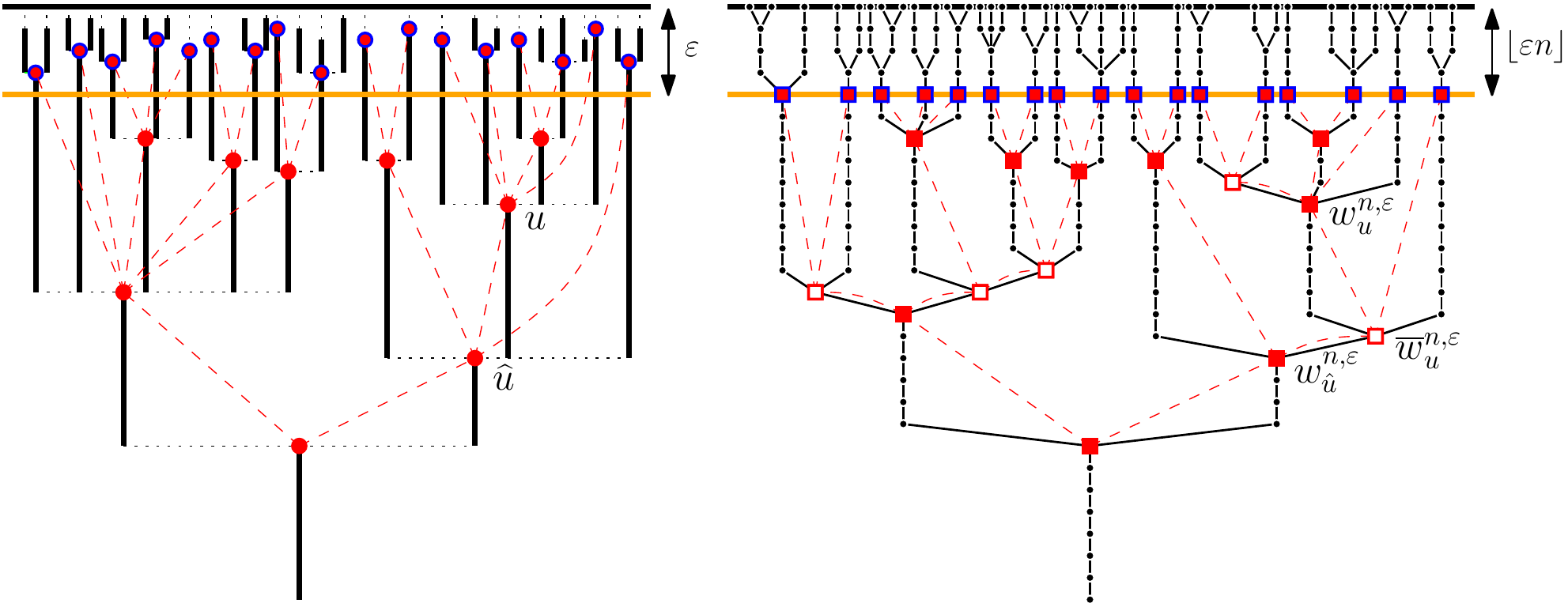}
\caption{On the left, the tree $\Delta$, its truncation $\Delta_{ \varepsilon}$ and its skeleton $\mathrm{Sk}( \Delta_{ \varepsilon})$. On the right, a large reduced tree $\t^{*n}$ of height $n$, its truncation $R_{\varepsilon n}(\t^{*n})$ and the associated tree $[R_{\varepsilon n}(\t^{*n})]$. The vertices depicted as filled red disks on the left correspond to the vertices depicted as filled red squares on the right, via the mapping $\Psi_{n}^{\ve}$.}
\label{fig:conv-reduced}
\end{center}
\end{figure}

\cref{conv-reduced-tree} (see~\cref{fig:conv-reduced} for an illustration)  essentially results from the convergence in distribution of the rescaled contour functions associated with the trees $\t^{(n)}$ towards the excursion of the stable height process with height greater than $1$ (see~\cite[Section 2.5]{DLG02}). By using the Skorokhod representation theorem, one may assume that the trees $\t^{(n)}$ and the excursion of the stable height process are constructed so that the latter convergence holds almost surely. The various assertions of \cref{conv-reduced-tree} then easily follow (cf.~\cite[Section 2.6]{DLG02}), using the relation between the excursion of the stable height process with height greater than $1$ and the limiting reduced tree $\Delta$, which can be found in~\cite[Section 2.7]{DLG02}.

\smallskip
\noindent \textbf{Remark 1.} Let us take $0<\delta<\ve$. If $u$ is not a leaf of $\hbox{Sk}(\Delta_\ve)$, we must have $w^{n,\ve}_u=w^{n,\delta}_u$ for sufficiently large $n$. On the other hand, if $u$ is a leaf of $\hbox{Sk}(\Delta_\ve)$, then for large $n$, $w^{n,\ve}_u$ must be an ancestor of $w^{n,\delta}_u$. 

\smallskip
\noindent \textbf{Remark 2.} We expect that a result more precise than~\cref{conv-reduced-tree} should hold. For all sufficiently large $n$, the mapping $\Psi_{n}^{\ve}$ should be a bijection, and the equality $\overline w^{n,\ve}_u=w^{n,\ve}_{\hat u}$ should hold for all $u\in \hbox{Sk}(\Delta_\ve)$ (in other words, there should be no white square in the right part of~\cref{fig:conv-reduced}). However this refinement does not easily follow from the results of~\cite{DLG02}, and we will omit it since it is not needed for our purposes. 

\subsection{Convergence of harmonic measures}
\label{sec:conv-harmonic}

Recall that $\mu$ is the continuous harmonic measure on the boundary $\partial \Delta$ of the reduced stable tree, and that $\mu_{n}$ is the discrete harmonic measure on $\t^{*n}_{n}$. For every $x\in\partial \Delta_\ve$, we set
$$\mu^\ve(x)=\mu(\{y\in\partial\Delta \colon x\prec y\})=P(x \prec B_{T_{-}}).$$
Similarly, we define a probability measure $\mu^\ve_n$ on $\t^{*n}_{n-\lfloor \ve n\rfloor}$ by setting
$$\mu^\ve_n(u)= \mu_n(\{ v \in \t_{n}^{*n} \colon u \prec v\}),$$
for every $u\in\t^{*n}_{n-\lfloor \ve n\rfloor}$. Clearly, $\mu_n^\ve$ is the distribution of $\langle \Sigma_n \rangle_{n-\lfloor \ve n\rfloor}$. 

\begin{proposition}
\label{convergence-harmonic}
Suppose that the reduced trees $\t^{*n}$ and the (continuous) tree $\Delta$ have been constructed so that the properties of \cref{conv-reduced-tree} hold, and recall the notation $(w^{n,\ve}_u)_{u\in{\rm Sk}(\Delta_\ve)}$ introduced therein. Then $\P$-a.s.~for every $x=(u,1-\ve)\in\partial \Delta_\ve$, 
$$\lim_{n\to\infty} \mu^\ve_n(w^{n,\ve}_u) = \mu^\ve(x).$$
\end{proposition}

\begin{proof}
Let $\delta\in(0,\ve)$ and set $T_\delta=\inf\{t\geq 0: H(B_t)=1-\delta\} <T$. Define a probability measure $\mu^{\ve,(\delta)}$ on $\partial\Delta_\ve$ by setting for every $x\in \partial\Delta_\ve$,
$$\mu^{\ve,(\delta)}(x)= P(x\prec B_{T_\delta}).$$
Similarly, we write $\mu^{(\delta)}_n$ for the distribution of the hitting point of generation $n-\lfloor \delta n\rfloor$ by random walk on $\t^{*n}$ started from $\varnothing$. Then we define a probability measure $\mu^{\ve,(\delta)}_n$ on $\t^{*n}_{n-\lfloor \ve n\rfloor}$ by setting
$$\mu^{\ve,(\delta)}_n (v)=\mu_n^{(\delta)} (\{w\in \t^{*n}_{n-\lfloor \delta n\rfloor}\colon v\prec w\}),$$
for every $v\in \t^{*n}_{n-\lfloor \ve n\rfloor}$.

As in the proof of~\cite[Proposition 18]{CLG13}, we have $\P$-a.s.
\begin{eqnarray*}
\lim_{\delta\to 0} \Big(\sup_{x \in \partial \Delta_{\ve}} \big|\mu^{\ve,(\delta)}(x)- \mu^\ve(x)\big|\Big) &= &0, \\
\lim_{\delta\to 0}\bigg(\limsup_{n\to\infty} \Big(\sup_{v\in \t^{*n}_{n-\lfloor \ve n\rfloor}} \big|\mu^{\ve,(\delta)}_n(v)- \mu^\ve_n(v)\big|\Big)\bigg)  &= & 0.
\end{eqnarray*}
So the convergence of the proposition will follow if we can verify that for every fixed $\delta\in(0,\ve)$, we have $\P$-a.s.~for every $x=(u,1-\ve)\in \partial \Delta_\ve$,
\begin{equation}
\label{convharmo1}
\lim_{n\to\infty} \mu^{\ve,(\delta)}_n(w^{n,\ve}_u) = \mu^{\ve,(\delta)}(x).
\end{equation}
To this end, we may and will assume that the reduced trees $\t^{*n}$ and the (continuous) tree $\Delta$ have been constructed so that the properties of \cref{conv-reduced-tree} hold simultaneously for $\ve$ and for $\delta$.

Firstly, by considering the successive passage times of Brownian motion stopped at time $T_\delta$ in the set $\{(u,Y_u\wedge (1-\delta))\colon u \in {\rm Sk}(\Delta_\delta)\}$, we get a Markov chain $X^{(\delta)}$, which is absorbed in the set $\{(v,1-\delta)\colon v\hbox{ is a leaf of }{\rm Sk}(\Delta_\delta)\}$, and whose transition kernels are explicitly described in terms of the quantities $Y_u, u\in{\rm Sk}(\Delta_\delta)$ by series and parallel circuits calculation. 

Secondly, let $n$ be sufficiently large so that assertions (1) and (2) of~\cref{conv-reduced-tree} hold with $\ve$ as well as with $\delta$, and consider simple random walk on $\t^{*n}$ started from $\varnothing$ and stopped at the first hitting time of generation $n-\lfloor \delta n\rfloor$. By considering the successive passage times of this random walk in the set $\mathcal{S}(R_{\delta n}(\t^{*n}))$, we again get a Markov chain $X^{(\delta),n}$, which is absorbed in the set 
$$\big\{v \in \mathcal{S}(R_{\delta n}( \t^{*n}))\colon [v] \mbox{ is a leaf of }[R_{\delta n}(\t^{*n})]\big\}. $$
By property (1.b) of~\cref{conv-reduced-tree}, this set is exactly $\{w^{n,\delta}_v \colon v\hbox{ is a leaf of }{\rm Sk}(\Delta_\delta)\}$. As previously, the transition kernels of this Markov chain $X^{(\delta),n}$ can be written explicitly in terms of the quantities $|v|, v\in \mathcal{S}(R_{\delta n}(\t^{*n}))$. 

Recall that by~\cref{conv-reduced-tree},
$$\Psi^{\delta}_{n}({\rm Sk}(\Delta_\delta))=\{w_{u}^{n,\delta}\colon u\in {\rm Sk}(\Delta_\delta)\}$$
is a subset of $\mathcal{S}(R_{\delta n}(\t^{*n}))$, and that the mapping $\Psi^{\delta}_{n}$ is injective. If we let $\widetilde X^{(\delta),n}$ be the Markov chain restricted to the subset $\Psi^{\delta}_{n}({\rm Sk}(\Delta_\delta))$, then after identifying both sets $\{(u,Y_u\wedge (1-\delta))\colon u\in {\rm Sk}(\Delta_\delta)\}$ and $\Psi^{\delta}_{n}({\rm Sk}(\Delta_\delta))$ with ${\rm Sk}(\Delta_\delta)$, we can view both $X^{(\delta)}$ and $\widetilde X^{(\delta),n}$ as Markov chains with values in the set ${\rm Sk}(\Delta_\delta)$. Using property (1.c) of~\cref{conv-reduced-tree}, we see that the transition kernels of $\widetilde X^{(\delta),n}$ converge to those of $X^{(\delta)}$.

Write $X^{(\delta)}_\infty$ for the absorption point of $X^{(\delta)}$, and similarly write $X^{(\delta),n}_\infty$ for that of $X^{(\delta),n}$. Notice that $X^{(\delta),n}_\infty$ is also the absorption point of the restricted Markov chain $\widetilde X^{(\delta),n}$. We thus obtain that the distribution of $X^{(\delta),n}_\infty$ converges to that of $X^{(\delta)}_\infty$ (recall that both $X^{(\delta),n}_\infty$ and $X^{(\delta)}_\infty$ are viewed as taking values in the set of leaves of ${\rm Sk}(\Delta_\delta)$). Consequently, for every $u\in\mathcal{V}$ such that $x=(u,1-\ve)\in \partial\Delta_\ve$, we have
$$\lim_{n\to \infty} P(u\prec X^{(\delta),n}_\infty) = P(u\prec X^{(\delta)}_\infty).$$
However, from our definitions, we have
$$P(u\prec X^{(\delta)}_\infty) = \mu^{\ve,(\delta)}(x),$$
and, for $n$ sufficiently large, since $w^{n,\ve}_u$ coincides with the ancestor of $w^{n,\delta}_u$ at generation $n-\lfloor \ve n\rfloor$ (see Remark 1 after~\cref{conv-reduced-tree}),
$$P(u\prec X^{(\delta),n}_\infty)= \mu_n^{\ve,(\delta)}(w^{n,\ve}_u).$$
This completes the proof of (\ref{convharmo1}) and of the proposition. 
\end{proof}

Recall that $\beta$ is the Hausdorff dimension of the continuous harmonic meaure $\mu$.
\begin{proposition} 
\label{pro:comparaison} 
Let $r\geq 1$ and $\xi\in (0,1)$. We can find $ \varepsilon_{0}\in(0,1/2)$ such that the following holds.  For every $\varepsilon \in (0, \varepsilon_{0})$, there exists $n_{0} \geq 0$ such that for every $n \geq n_{0}$,
\begin{eqnarray*}  
\mathbb{E} \otimes E \bigg[ \Big|\log  \mu_{n}^ \varepsilon \big( \langle \Sigma_{n}\rangle_{n- \lfloor \varepsilon n \rfloor} \big) - \beta \log \varepsilon \Big|^{r}\bigg] & \leq &  \xi\, |\log \varepsilon|^{r}.  
\end{eqnarray*}
\end{proposition}
 
\proof 
Recall our notation $\mathcal{B}_\mathbf{d}(x,r)$ for the closed ball of radius $r$ centered at $x\in\Delta$. Fix $\eta\in(0,1)$. Since $B_{T_{-}}$ is distributed according to $\mu$, it follows from~\cref{thm:dim-stable} that there exists $\varepsilon_{0}\in(0,1/2)$ such that for every $ \varepsilon \in (0, \varepsilon_{0})$ we have 
\begin{equation}
\label{compatech1}
\mathbb{P} \otimes P \left( \Big|\log \mu ( \mathcal{B}_\mathbf{d}(B_{T_{-}}, 2\varepsilon)) - \beta \log \varepsilon\Big| > 
 (\eta/2)|\log \varepsilon|\right) < \eta/2.
\end{equation} 

Let us fix $\ve\in(0,\ve_0)$ and assume that the reduced trees $\t^{*n}$ and the (continuous) tree $\Delta$ have been constructed so that the properties of~\cref{conv-reduced-tree} hold. We now claim that, under $\P\otimes P$,
\begin{equation}
\label{convharmocont}
\mu_{n}^ \varepsilon( \langle \Sigma_{n}\rangle_{n- \lfloor \varepsilon n \rfloor})\build{\longrightarrow}_{n\to\infty}^{(\mathrm{d})} \mu (\mathcal{B}_\bd(B_{T_{-}}, 2\varepsilon)).
\end{equation}
To see this, let $f$ be a continuous function on $[0,1]$. Since  the distribution of $\langle \Sigma_n \rangle_{n-\lfloor \ve n\rfloor}$ under $P$ is $\mu_n^\ve$, we have
$$\E\otimes E\Big[ f(\mu_{n}^ \varepsilon( \langle \Sigma_{n}\rangle_{n- \lfloor \varepsilon n \rfloor}))\Big]= \E\bigg[ \sum_{v\in \t^{*n}_{n-\lfloor \ve n\rfloor}} \mu_n^\ve(v)\, f(\mu_n^\ve(v))\bigg].$$
By property (1.b) of~\cref{conv-reduced-tree}, we know that $\P$-a.s.~for $n$ sufficiently large,
$$\sum_{v \in \t^{*n}_{n-\lfloor \ve n\rfloor}} \mu_n^\ve(v)\, f(\mu_n^\ve(v))
=\sum_{x=(u,1-\ve)\in \partial \Delta_\ve} \mu^\ve_n(w^{n,\ve}_u)\,f(\mu^\ve_n(w^{n,\ve}_u)),$$
and by~\cref{convergence-harmonic} the latter quantities converge as $n\to\infty$ towards
$$\sum_{x\in\partial \Delta_\ve} \mu^\ve(x)\,f(\mu^\ve(x)) = E\big[f( \mu (\mathcal{B}_\bd(B_{T_{-}}, 2\varepsilon)))\big],$$
which establishes the convergence \eqref{convharmocont} as claimed.

By \eqref{compatech1} and \eqref{convharmocont}, we can 
find $ n_{0}=n_{0}( \varepsilon) \geq \varepsilon^{-1}$  such that for $n\geq n_{0}$,
\begin{eqnarray*} 
\mathbb{P} \otimes P\left( \left| \log  \mu_{n}^ \varepsilon\big(  \langle \Sigma_{n}\rangle_{n - \lfloor  \varepsilon n \rfloor}\big) - \beta \log \varepsilon \right| > \eta\,|{\log \varepsilon}| \right) < \eta.  
\end{eqnarray*}
Using the Cauchy-Schwarz inequality, we have then
\begin{eqnarray} 
\label{compatech2}
 &&\mathbb{E} \otimes E \left[ \left|\log  \mu_{n}^ \varepsilon\big(  \langle \Sigma_{n}\rangle_{n - \lfloor  \varepsilon n \rfloor}\big) - \beta \log \varepsilon \right|^{r}\right]\notag\\ 
 &&\qquad\leq  \eta^{r} |\log \varepsilon|^{r}  + \eta^{\frac{1}{2}} \mathbb{E} \otimes E\Big[\left|\log  \mu_{n}^ \varepsilon\big(  \langle \Sigma_{n}\rangle_{n - \lfloor  \varepsilon n \rfloor}\big) - \beta \log \varepsilon \right|^{2r}\Big]^{1/2}\notag\\
 &&\qquad \leq \eta^{r} |\log \varepsilon|^{r}+2^{r}\eta^{\frac{1}{2}}|\beta\log\ve |^{r}  + 2^{r}\eta^{\frac{1}{2}} \mathbb{E} \otimes E\Big[\left|\log  \mu_{n}^ \varepsilon\big(  \langle \Sigma_{n}\rangle_{n - \lfloor  \varepsilon n \rfloor}\big) \right|^{2r}\Big]^{1/2}. 
\end{eqnarray}
Since $r\geq 1$, the function
$$g(x)\colonequals(x \wedge e^{-2r})\,|\log(x \wedge e^{-2r})|^{2r}$$ 
is nondecreasing and concave over $[0,1]$. Thus, we obtain
\begin{align*}
E\Big[\left|\log  \mu_{n}^ \varepsilon\big(  \langle \Sigma_{n}\rangle_{n - \lfloor  \varepsilon n \rfloor}\big) \right|^{2r}\Big]&
=\sum_{v\in \t^{*n}_{n-\lfloor \ve n\rfloor}} \mu_n^\ve(v) |\log \mu_n^\ve(v)|^{2r}\\
&\leq \sum_{v\in \t^{*n}_{n-\lfloor \ve n\rfloor}} g(\mu_n^\ve(v))+ (2r)^{2r}\\
&\leq \#\t^{*n}_{n-\lfloor \ve n\rfloor}\times  g\Big( ( \#\t^{*n}_{n-\lfloor \ve n\rfloor} )^{-1}\Big) + (2r)^{2r}\\
&\leq \Big| \log  \#\t^{*n}_{n-\lfloor \ve n\rfloor} \Big|^{2r} + 2(2r)^{2r}.
\end{align*}
We now use \cref{lem:level-size} to see
$$\E\otimes E\Big[\left|\log  \mu_{n}^ \varepsilon\big(  \langle \Sigma_{n}\rangle_{n - \lfloor  \varepsilon n \rfloor}\big) \right|^{2r}\Big]
\leq  \E\Big[\Big| \log  \#\t^{*n}_{n-\lfloor \ve n\rfloor} \Big|^{2r}\Big]+ 2(2r)^{2r} \leq 
C^{2r} \Big(\log\frac{n}{\lfloor \ve n\rfloor}\Big)^{2r}+2(2r)^{2r}.$$
By combining the last estimate with (\ref{compatech2}), we get that, for every $n\geq n_0(\ve)$,
$$\mathbb{E} \otimes E \left[ \left|\log  \mu_{n}^ \varepsilon\big(  \langle \Sigma_{n}\rangle_{n - \lfloor  \varepsilon n \rfloor}\big) - \beta \log \varepsilon \right|^{r}\right] \leq (\eta^{r}+ 2^{r}\eta^{\frac{1}{2}}\beta^{r})|\log\ve|^{r} + 2^{r+1}\eta^{\frac{1}{2}} \big( (2r)^{r} + C^{r} |\log\ve|^{r}\big).$$
The statement of the proposition follows since $\eta$ was arbitrary. 
\endproof

\subsection{The flow property of discrete harmonic measure}
We briefly recall the flow property of the discrete harmonic measure $\mu_{n}$ presented in~\cite[Section 4.3.1]{CLG13}. Let $\mathsf{t}\in\mathscr{T}_n$ be a plane tree of height $n$ and $Z^{(\mathsf{t})}=(Z^{(\mathsf{t})}_k)_{k\geq 0}$ be simple random walk on $\mathsf{t}$ starting from $\varnothing$. We set 
$$H_n^{(\mathsf{t})}\colonequals \inf\{k\geq 0\colon |Z^{(\mathsf{t})}_k|=n\}\quad\mbox{ and } \quad \Sigma_{n}^{(\mathsf{t})}\colonequals  Z^{(\mathsf{t})}_{H^{(\mathsf{t})}_n}.$$ 
We write $\mu^{(\mathsf{t})}_n$ for the distribution of $\Sigma_{n}^{(\mathsf{t})}$, considered as a measure on $\mathsf{t}$ supported on $\mathsf{t}_n$. 

For $0\leq p\leq n$, we set
$$L^{(\mathsf{t})}_p  \colonequals \sup\{k\leq H_n^{(\mathsf{t})}\colon |Z^{(\mathsf{t})}_k| =p\}.$$
Clearly $\Sigma_{n}^{(\mathsf{t})} \in \wt{\mathsf{t}}[Z^{(\mathsf{t})}_{L^{(\mathsf{t})}_p}]$, and therefore $ Z^{(\mathsf{t})}_{L^{(\mathsf{t})}_p}=\langle\Sigma_{n}^{(\mathsf{t})}\rangle_p$. 

\begin{lemma}[Lemma 20 in~\cite{CLG13}]
\label{conditioning-subtree}
Let $p\in \{0,1,\ldots,n-1\}$ and $z\in \mathsf{t}_p$. Then, conditionally on $\langle\Sigma_{n}^{(\mathsf{t})}\rangle_p=z$, the process
$$\Big( Z^{(\mathsf{t})}_{(L^{(\mathsf{t})}_p+k)\wedge H^{(\mathsf{t})}_n}\Big)_{k\geq 0}$$
is distributed as simple random walk on $\wt{\mathsf{t}}[z]$ starting from $z$ and conditioned to hit $\wt{\mathsf{t}}[z]\cap \mathsf{t}_n$ before returning to $z$, and stopped at this hitting time. Consequently, for every integer $q\in\{0,1,\ldots,n-p\}$, the conditional distribution of 
$$ \frac{\mu_n^{(\mathsf{t})}\big(B_{\mathsf{t}}(\Sigma_{n}^{(\mathsf{t})},q)\big)}{\mu_n^{(\mathsf{t})}\big(B_{\mathsf{t}}(\Sigma_{n}^{(\mathsf{t})},n-p)\big)}$$
knowing that $\langle\Sigma_{n}^{(\mathsf{t})}\rangle_p=z$ is equal to the distribution of
$$\mu^{(\mathsf{t}[z])}_{n-p}\big(B_{\mathsf{t}[z]}(\Sigma_{n-p}^{(\mathsf{t}[z])},q)\big).$$
\end{lemma}

\subsection{The subtree selected by the discrete harmonic measure}
We begin by introducing the conductance of discrete trees. Let $i$ be a positive integer and let $\mathsf{t} \in\mathscr{T}$ be a tree such that $h(\mathsf{t})\geq i$. Consider the new graph $\mathsf{t}'$ obtained by adding to the graph $\mathsf{t}$ an edge between the root $\varnothing$ and an extra vertex~$\partial$. We denote by $\mathcal{C}_i(\mathsf{t})$ the effective conductance between $\partial$ and generation $i$ of $\mathsf{t}$ in the graph $\mathsf{t}'$. In probabilistic terms, it is equal to the probability that simple random walk on $\mathsf{t}'$ starting from~$\varnothing$ hits generation $i$ of $\mathsf{t}$ before hitting the vertex $\partial$.

Recall that for $i\in\{1,\ldots,n-1\}$, $\wt\t^{*n}[\langle\Sigma_n\rangle_{n-i}]$ is the subtree of $\t^{*n}$ above generation $n-i$ that is selected by harmonic measure, and $\t^{*n}[\langle\Sigma_n\rangle_{n-i}]$ is the tree obtained by relabeling the vertices of $\wt\t^{*n}[\langle\Sigma_n\rangle_{n-i}]$ as explained above. 

\begin{lemma}
\label{lem:tree-selected}
For every integer $i\in \{1,\ldots,n-1\}$ and every nonnegative function $F$ on $\mathscr{T}$,
$$\E\otimes E \big[ F\big(\t^{*n}[\langle\Sigma_n\rangle_{n-i}]\big) \big] \leq (i+1) \, \E\big[ \mathcal{C}_i(\t^{*i})\, F(\t^{*i}) \big].$$
\end{lemma}
This lemma is proved in~\cite{CLG13} under the assumption that $\rho$ has finite variance. Actually the proof uses only the branching property of Galton-Watson trees and remains valid under our assumptions on $\rho$.

Meanwhile, we have the following moment estimate for the conductance $\mathcal{C}_n(\t^{*n})$. 

\begin{lemma}
\label{moment-conductance}
For every $r\in (0,\alpha)$, there exists a constant $K=K(r,\rho)\geq 1$ depending on $r$ and the offspring distribution $\rho$ such that, for every integer $n\geq 1$,
$$\E\big[\mathcal{C}_n(\t^{*n})^{r}\big] \leq \frac{K}{(n+1)^{r}}\,.$$
\end{lemma}

\begin{proof}
We can assume $n\geq 2$, and set $j=\lfloor n/2\rfloor \geq 1$. An application of the Nash--Williams inequality \cite[Chapter 2]{LP10} gives
$$\mathcal{C}_n(\t^{*n}) \leq \frac{ \# \t^{*n}_j}{j}.$$
On the other hand,
\begin{align*}
\E\big[(\# \t^{*n}_j)^{r}\big]&= \E\big[(\#\{v\in \t^{(0)}_j\colon h(\t^{(0)}[v])\geq n-j\})^{r}\mid h(\t^{(0)})\geq n\big] \\
&= q_n^{-1} \, \E\big[(\#\{v\in \t^{(0)}_j\colon h(\t^{(0)}[v])\geq n-j\})^{r}\big].
\end{align*}
Notice that given $\#\t^{(0)}_{j}=k$, the conditional distribution of $\#\{v\in\t^{(0)}_{j} \colon h(\t^{(0)}[v])\geq n-j\}$ is the binomial distribution $\mathcal{B}(k,q_{n-j})$.
Using Jensen's inequality, we get 
\begin{eqnarray*}
\E\Big[(\#\{v\in \t^{(0)}_j\colon h(\t^{(0)}[v])\geq n-j\})^{r}\Big] &\leq& \E\Big[\E\Big[(\#\{v\in \t^{(0)}_j\colon h(\t^{(0)}[v])\geq n-j\})^{2}\mid \#\t^{(0)}_{j}\Big]^{\frac{r}{2}}\Big]\\
&=& \E\Big[ \Big(q_{n-j}^2 (\#\t^{(0)}_j)^2 + (q_{n-j}-q_{n-j}^2) \#\t^{(0)}_j\Big)^{\frac{r}{2}} \Big]\\
&\leq& q_{n-j}^{r}\,\E\big[(\#\t^{(0)}_j)^{r}\big] + (q_{n-j}-q_{n-j}^2)^{\frac{r}{2}}\E\big[(\#\t^{(0)}_j)^{\frac{r}{2}}\big].
\end{eqnarray*}
At this point, we need the following result proved in~\cite[Lemma 11]{FVW07} for the unconditioned Galton-Watson tree. For any $\gamma\in (0,\alpha)$, there is a finite constant $C(\gamma)$ such that for every $m\geq 1$,
\begin{equation}
\label{eq:moment-esti}
\E\big[(\#\t^{(0)}_{m})^{\gamma}\big]\leq C(\gamma)\,q_{m}^{1-\gamma}\,.
\end{equation}
The original statement of the latter bound in~\cite{FVW07} was given for any $\gamma \in [1,\alpha)$, while the case $\gamma\in (0,1)$ follows from the (trivial) case $\gamma=1$ by applying the H\"{o}lder inequality to 
$$\E\big[\mathbf{1}_{\{\t^{(0)}_m\neq \emptyset\}}(\# \t^{(0)}_m)^{\gamma}\big]$$
(we can in fact take $C(\gamma)=1$ for any $\gamma \in(0,1)$).

With the help of~(\ref{eq:moment-esti}), we conclude that
$$\E\big[\mathcal{C}_n(\t^{*n})^{r}\big] \leq j^{-r}q_n^{-1}\,\Big(C(r)q_{n-j}^{r}q_{j}^{1-r}+C(r/2)(q_{n-j}-q_{n-j}^{2})^{\frac{r}{2}}q_{j}^{1-\frac{r}{2}}\Big),$$
and the statement of the lemma readily follows from~(\ref{eq:survivalpro}).
\end{proof}

\subsection{Proof of Theorem \ref{thm:dim-discrete}}

Following~\cite{CLG13}, we will show 
\begin{equation}  
\label{eq:goal} 
\mathbb{E} \otimes E \big[ |\log \mu_{n}( \Sigma_{n}) + \beta \log n |\big] = o(\log n) \qquad \mbox{ as }n \to \infty,
\end{equation}
which is sufficient for establishing Theorem \ref{thm:dim-discrete}. The proof given below is adapted from~\cite[Section 4.3.2]{CLG13}. For later convenience, we introduce the notation
$$\bar \alpha \colonequals\frac{\alpha+1}{2}\in(1,\frac{3}{2})$$
and its H\"{o}lder conjugate 
$$\alpha^{*}\colonequals \frac{\bar\alpha}{\bar\alpha-1}\in (3,\infty).$$ 
  
Fix $\xi >0$. Let $ \varepsilon>0$ and $n_{0} \geq 1$ be such that the conclusion of \cref{pro:comparaison} holds for every $n\geq n_0$ with the exponent $r=\alpha^{*}$. Without loss of generality, we may and will assume that $\ve =1/N$, for some
integer $N\geq 4$, which is fixed throughout the proof. We also fix a constant $\gamma>0$, such that $\gamma \log N <1/2$.

Let $n> N$ be sufficiently large so that $N^{\lfloor \gamma \log n\rfloor}\geq n_0$. Then we let $\ell\geq 1$ be the unique integer such that $N^\ell < n \leq N^{\ell+1}$, and write 
\begin{equation}
\label{decomp-harmo}
\log \mu_n(\Sigma_n)= \log\frac{\mu_n(\Sigma_n)}{\mu_n(B(\Sigma_n,N))}
+\sum_{j=2}^\ell \log\frac{\mu_n(B(\Sigma_n,N^{j-1}))}{\mu_n(B(\Sigma_n,N^j))} + \log \mu_n(B(\Sigma_n, N^\ell)).
\end{equation}
To simplify notation, we set 
\begin{eqnarray*} 
A_{1}^n &\colonequals& \log\frac{\mu_n(\Sigma_n)}{\mu_n(B(\Sigma_n,N))} +\beta\log N,\\ 
A^n_j&\colonequals& \log\frac{\mu_n(B(\Sigma_n,N^{j-1}))}{\mu_n(B(\Sigma_n,N^j))} + \beta \log N \quad \mbox{for every }j\in\{2,\ldots,\ell\},\\
A_{\ell+1}^n &\colonequals& \log \mu_n(B(\Sigma_n, N^\ell)) + \beta \log(n/N^\ell)\,.
\end{eqnarray*} 
From \eqref{decomp-harmo}, we see that
\begin{equation}
\label{pfdiscrete1}
\mathbb{E} \otimes E \Big[ \big|\log \mu_{n}( \Sigma_{n}) + \beta \log n\big|\Big] = \mathbb{E} \otimes E\Big[ \Big|\sum_{j=1}^{\ell+1} A_{j}^n\Big|\Big]\leq  \sum_{i=1}^{\ell+1} \mathbb{E} \otimes E [|A^n_j|] .
\end{equation}
We will bound each term in the sum of the right-hand side.

\smallskip

\textsc{First step: A priori bounds.} We verify that, for every $j\in\{1,2,\ldots, \ell +1\}$,
\begin{equation}
\label{aprioribound}
\E\otimes E\big[|A^n_j|\big] \leq (C K^{1/\bar\alpha}+ \beta) \,\log N,
\end{equation}
where $C=C(\alpha^{*},\rho)$ is the constant in \cref{lem:level-size} for the exponent $r=\alpha^{*}$, and $K=K(\bar \alpha,\rho)$ is the constant in \cref{moment-conductance} for the exponent $r=\bar \alpha$. 

Suppose first that $2\leq j \leq \ell$. Using the second assertion of \cref{conditioning-subtree},  with $p=n-N^j$ and $q=N^{j-1}$, we obtain that, for every $z\in \t^{*n}_{n-N^j}$, the conditional distribution of $A^n_j$ under $P$, knowing that $\langle\Sigma_n\rangle_{n-N^j}=z$, is the same as the distribution of
$$\log \mu_{N^j}^{(\t^{*n}[z])}(B(\Sigma_{N^j}^{(\t^{*n}[z])},N^{j-1})) + \beta \log N.$$
Recalling that $\mu_{N^j}^{(\t^{*n}[z])}$ is the distribution of $\Sigma_{N^j}^{(\t^{*n}[z])}$ under $P$, we get
\begin{eqnarray} 
\label{aprioritech1}
E\big[ |A^n_j| \mid \langle\Sigma_n\rangle_{n-N^j}=z\big] 
&\leq&   E\left[\big|\log \mu_{N^j}^{(\t^{*n}[z])}\big(B(\Sigma_{N^j}^{( \t^{*n}[z])},N^{j-1})\big)\big|\right]
+ \beta \log N  \notag\\ 
& =&G_j(\t^{*n}[z]) +  \beta \log N , 
\end{eqnarray}
where for any tree $\mathsf{t}\in\mathscr{T}_{N^j}$,
$$G_j(\mathsf{t})\colonequals \int \mu^{(\mathsf{t})}_{N^j}(\mathrm{d}y)\,\big|\log \mu^{(\mathsf{t})}_{N^j}(B_{\mathsf{t}}(y, N^{j-1}))\big|= \sum_{z\in \mathsf{t}_{N^j-N^{j-1}}} \mu^{(\mathsf{t})}_{N^j}(\wt{\mathsf{t}}[z])\,\big|\log \mu^{(\mathsf{t})}_{N^j}(\wt{\mathsf{t}}[z])\big|.$$ 
As explained in \cite{CLG13}, we have the entropy bound $G_j(\mathsf{t})\leq \log \# \mathsf{t}_{N^j-N^{j-1}}$ for any tree $\mathsf{t}\in\mathscr{T}_{N^j}$. So we get from~\eqref{aprioritech1} that
\begin{align*}
\E\otimes E\big[ |A^n_j|\big]  &\leq \E\otimes E\big[\log \#\t_{N^j-N^{j-1}}^{*n}[\langle\Sigma_n\rangle_{n-N^j}]\big]+ \beta \log N\\
&\leq  (N^j +1)\, \E\left[\mathcal{C}_{N^j}(\t^{*N^j})\, \log\#\t^{*N^j}_{N^j-N^{j-1}}\right] + \beta \log N\\
&\leq (N^j+1)\, \E\left[\big(\mathcal{C}_{N^j}(\t^{*N^j})\big)^{\bar\alpha}\right]^{1/\bar\alpha}\, \E\left[\big(\log\#\t^{*N^j}_{N^j-N^{j-1}}\big)^{\alpha^{*}}\right]^{1/\alpha^{*}} 
+ \beta \log N\\
&\leq K^{1/\bar\alpha}\,\E\left[\big(\log\#\t^{*N^j}_{N^j-N^{j-1}}\big)^{\alpha^{*}}\right]^{1/\alpha^{*}} + \beta \log N,
\end{align*}
using successively \cref{lem:tree-selected}, the H\"{o}lder inequality and \cref{moment-conductance}. Finally, \cref{lem:level-size} gives
$$\E\left[\big(\log\#\t^{*N^j}_{N^j-N^{j-1}}\big)^{\alpha^{*}}\right]^{1/\alpha^{*}}\leq C\,\log N, $$
and this completes the proof of \eqref{aprioribound} when $2\leq j \leq \ell$. The cases $j=1$ and $j=\ell+1$ can be treated in a similar manner. For details we refer the reader to~\cite[Section 4.3.2]{CLG13}. 

\smallskip

\textsc{Second step: Refined bounds.} Let us prove that, if $\lfloor \gamma \log n\rfloor \leq j \leq \ell$, 
\begin{equation}
\label{refinedbound}
\E\otimes E\big[|A^n_j|\big] \leq K^{1/\bar\alpha}\xi^{1/\alpha^{*}}\log N.
\end{equation}
Recall that for $j\in\{\lfloor \gamma \log n\rfloor,\ldots,\ell\}$ we have $N^j\geq n_0$. From \eqref{aprioritech1}, we have
\begin{equation}
\label{estimtech1}
E\big[|A^n_j|\big]= E\big[F_j(\t^{*n}[\langle\Sigma_n\rangle_{n-N^j}])\big],
\end{equation}
where, if $\mathsf{t}\in\mathscr{T}_{N^j}$,
$$F_j(\mathsf{t}) \colonequals  \big|\beta \log N - G_j(\mathsf{t})\big|= \left|\int \mu^{(\mathsf{t})}_{N^j}(\mathrm{d}y)\,\left(\log \mu^{(\mathsf{t})}_{N^j}(B_{\mathsf{t}}(y, N^{j-1})) + \beta \log N\right)\right|.$$
Using \cref{lem:tree-selected} as in the first step, we have 
$$\E\otimes E\big[ |A^n_j|\big]= \mathbb{E} \otimes E\big[F_j(\t^{*n}[\langle\Sigma_n\rangle_{n-N^j}])\big] \leq  (N^j+1) \mathbb{E} \left[ \mathcal{C}_{N^j}(\t^{*N^j})\, F_{j}(\t^{*N^j}) \right].$$
We then apply the H\"{o}lder inequality together with the bound of \cref{moment-conductance} for $r=\bar\alpha$ to get
\begin{align*}
\mathbb{E} \otimes E\left[|A_{j}^n|\right] 
&\leq K^{1/\bar\alpha}\,\mathbb{E}\big[F_{j}(\t^{*N^j})^{\alpha^{*}}\big]^{1/\alpha^{*}}\\
&\leq K^{1/\bar\alpha}\,\E\bigg[\left(\int \mu_{N^j}(\mathrm{d}y)\,\left|\log \mu_{N^j}(B(y, N^{j-1})) + \beta \log N\right|\right)^{\alpha^{*}}\bigg]^{1/\alpha^{*}}\\
&\leq K^{1/\bar\alpha}\, \E\left[\int \mu_{N^j}(\mathrm{d}y)\,\left|\log \mu_{N^j}(B(y, N^{j-1})) + \beta \log N\right|^{\alpha^{*}}\right]^{1/\alpha^{*}}\\
&= K^{1/\bar\alpha} \cdot \mathbb{E} \otimes E \left[ \left|\log  \mu_{N^j}^{1/N}\Big( \langle \Sigma_{N^j}\rangle_{N^j-N^{j-1}}\Big) + \beta \log N\right|^{\alpha^{*}}\right]^{1/\alpha^{*}},
\end{align*}
where the last equality follows from the definition of the measure $\mu^\ve_n$ at the beginning of \cref{sec:conv-harmonic}. Now recall that $1/N=\ve$ and note that $N^j-N^{j-1}=N^j -\ve N^{j}$. Since we have $N^j\geq n_0$, we can apply \cref{pro:comparaison} with $r=\alpha^{*}$ and get that the right-hand side of the preceding display is bounded above by $K^{1/\bar\alpha} \xi^{1/\alpha^{*}}\log N$, which finishes the proof of \eqref{refinedbound}. 

\smallskip
By combining \eqref{aprioribound} and \eqref{refinedbound}, and using \eqref{pfdiscrete1}, we arrive at the bound
\begin{align*}
\E\otimes E\left[ \big |\log \mu_n(\Sigma_n) + \beta \log n \big|\right]
&\leq \lfloor \gamma\log n\rfloor (K^{1/\bar\alpha}C +\beta)\log N + \ell\,K^{1/\bar\alpha} \xi^{1/\alpha^{*}}\log N\\
&\leq \big(\gamma (K^{1/\bar\alpha}C +\beta)\log N +K^{1/\bar\alpha} \xi^{1/\alpha^{*}}\big)\log n,
\end{align*}
which holds for every sufficiently large $n$. By choosing $\xi$ and then $\gamma$ arbitrarily small, we see that our claim \eqref{eq:goal} follows from the last bound, and this completes the proof of~\cref{thm:dim-discrete}.

\section{Comments and questions}
Following~\cite[Section 5.2]{CLG13}, let us consider the supercritical offspring distribution $\theta_{\alpha}^{(n)}$ of index $\alpha \in (1,2]$, defined as $ \theta_{\alpha}^{(n)}(1) = 1- \frac{1}{n}$ and
\begin{displaymath}
\theta_{\alpha}^{(n)}(k)= \frac{1}{n}\theta_{\alpha}(k) \quad \mbox{ for every }k\geq 2.
\end{displaymath} 
We let $\mathrm{T}_{\alpha}^{(n)}$ be an infinite Galton-Watson tree with offspring distribution $\theta_{\alpha}^{(n)}$, then $n^{-1}\mathrm{T}_{\alpha}^{(n)}$ viewed as a metric space with the graph distance rescaled by factor $n^{-1}$, converges in distribution in an appropriate sense (e.g.~for the local Gromov-Hausdorff topology) to the CTGW tree $\Gamma\a$, as $n\to \infty$.
 
Consider then the biased random walk $(Z^{(n)}_{k})_{k \geq 0}$ on $\mathrm{T}_{\alpha}^{(n)}$ with bias parameter $\lambda^{(n)} = 1- \frac{1}{n}$ towards root (see \cite{LPP96} or~\cite{Aid11} for a precise definition of this process). Then the rescaled process 
$$\Big(n^{-1}Z^{(n)}_{ \lfloor n^2 t\rfloor}\Big)_{t\geq 0}$$
will converge in distribution, as $n\to \infty$, to Brownian motion $(W(t))_{t\geq 0}$ with drift $1/2$ on the CTGW tree $\Gamma\a$, in a sense that can easily be made precise. Furthermore, the rescaled conductance $n\,\cc(\mathrm{T}_{\alpha}^{(n)}, \lambda^{(n)})$ converges in distribution to the conductance $\cc^{(\alpha)}=\cc(\Gamma\a)$. 

Following this informal passage to the limit, we can find heuristically a candidate for the limit of $n \mathbf{V}_{\alpha}^{(n)}$ as $n \to \infty$, where $\mathbf{V}_{\alpha}^{(n)}$ stands for the speed of the biased random walk $Z^{(n)}$ on $\mathrm{T}_{\alpha}^{(n)}$. One can either directly employ an explicit formula of $\mathbf{V}_{\alpha}^{(n)}$ stated in~\cite[Theorem 1.1]{Aid11}, or use the invariant measure for the environment seen from the random walker (\cite[Theorem 4.1]{Aid11}) to calculate the speed as the proportion of last-exit points. Both methods give rise to the following quantity which should be interpreted as the speed of Brownian motion $W$ with drift $1/2$ on $\Gamma\a$,
\begin{equation}
\label{eq:speed}
\mathbf{V}_{\alpha}\colonequals \frac{\E\Big[\frac{\cc^{(\alpha)}_{0}\cc^{(\alpha)}_{1}}{\cc^{(\alpha)}_{0}+\cc^{(\alpha)}_{1}-1}\Big]}{\E\Big[\frac{2\cc^{(\alpha)}_{0}}{\cc^{(\alpha)}_{0}+\cc^{(\alpha)}_{1}-1}\Big]}\,,
\end{equation}
where $\cc^{(\alpha)}_{0}$ and $\cc^{(\alpha)}_{1}$ are two independent copies of $\cc^{(\alpha)}$ under the probability measure $\P$. 

Since the conductance $\cc^{(\alpha)}$ is a.s.~strictly larger than 1, we see immediately from~(\ref{eq:speed}) that $\mathbf{V}_{\alpha}< \frac{1}{2}$ for any $\alpha \in (1,2]$. On the other hand, according to the coupling explained in~\cref{sec:coupling}, the denominator of the right-hand side of~(\ref{eq:speed}) 
\begin{displaymath}
\E\bigg[\frac{2\cc^{(\alpha)}_{0}}{\cc^{(\alpha)}_{0}+\cc^{(\alpha)}_{1}-1}\bigg]=\E\bigg[\frac{\cc^{(\alpha)}_{0}+\cc^{(\alpha)}_{1}}{\cc^{(\alpha)}_{0}+\cc^{(\alpha)}_{1}-1}\bigg]=1+\E\bigg[\frac{1}{\cc^{(\alpha)}_{0}+\cc^{(\alpha)}_{1}-1}\bigg]
\end{displaymath}
is increasing with respect to $\alpha$.

\smallskip
\noindent {\bf Question 1.} \textit{If we apply the coupling explained in~\cref{sec:coupling}, does the derivative $\frac{\mathrm{d}}{\mathrm{d}\alpha}\cc^{(\alpha)}$ of the conductance with respect to $\alpha$ exist almost surely?}
\smallskip

An affirmative answer to Question 1 would allow us to take the derivative of the numerator in~(\ref{eq:speed}) with respect to $\alpha$, and to see that 
\begin{displaymath}
\frac{\mathrm{d}}{\mathrm{d}\alpha}\E\Big[\frac{\cc^{(\alpha)}_{0}\cc^{(\alpha)}_{1}}{\cc^{(\alpha)}_{0}+\cc^{(\alpha)}_{1}-1}\Big] = \E\bigg[\frac{\cc^{(\alpha)}_{0}(\cc^{(\alpha)}_{0}-1)\frac{\mathrm{d}}{\mathrm{d}\alpha}\cc^{(\alpha)}_{1}+\cc^{(\alpha)}_{1}(\cc^{(\alpha)}_{1}-1)\frac{\mathrm{d}}{\mathrm{d}\alpha}\cc^{(\alpha)}_{0}}{\big(\cc^{(\alpha)}_{0}+\cc^{(\alpha)}_{1}-1\big)^{2}}\bigg] \leq 0 
\end{displaymath}
because a.s.~$\frac{\mathrm{d}}{\mathrm{d}\alpha}\cc^{(\alpha)}\leq 0$. Hence, the numerator in the right-hand side of~(\ref{eq:speed}) would be decreasing with respect to $\alpha$, and so would be the speed $\mathbf{V}_{\alpha}$.

\smallskip
\noindent {\bf Question 2.} \textit{Does the speed $\mathbf{V}_{\alpha}$ decrease with respect to~$\alpha$?}

\smallskip
A similar question was raised in~\cite{BFS11}, concerning the monotonicity of the speed with respect to the offspring distribution for biased random walk on Galton-Watson trees with no leaves. It has been proved in~\cite{MSZ13} that this monotonicity holds for high values of bias.

\smallskip

Finally, we also want to ask the same question for the Hausdorff dimension of the continuous harmonic measure. 

\smallskip
\noindent {\bf Question 3.} \textit{Does the Hausdorff dimension $\beta_{\alpha}$ decrease with respect to~$\alpha$?}

\section{Appendix: proofs postponed from Section~\ref{sec:second-appr-conti}}
\label{sec:appendix}
\renewcommand{\t}{\mathcal{T}}

Before starting the proofs, we state first a useful ``spine'' decomposition of the CTGW tree $\Gamma\a$ for $\alpha \in(1,2]$, which is a reformulation of the standard results about the size-biased Galton-Watson trees, see e.g.~\cite{CRW91}. Recall that $m_{\alpha}=\frac{\alpha}{\alpha-1}$ is the mean of the $\alpha$-offspring distribution~$\theta_{\alpha}$. The size-biased $\alpha$-offspring distribution $\widehat \theta_{\alpha}$ is then defined as
\begin{equation}
\label{eq:size-bias-theta}
\widehat \theta_{\alpha}(k)=\frac{(k+1)\theta_{\alpha}(k+1)}{m_{\alpha}}\quad\mbox{ for every }k\geq1.
\end{equation}

We take $\t=(\Pi,(z_v)_{v\in\Pi})\in \mathbb{T}$. If $\llbracket 0, x\rrbracket$ denotes the geodesic segment in $\t$ between the root and $x$, we can define the subtrees of~$\t$ branching off $\llbracket 0, x\rrbracket$. To this end, set $n_x=|v_x|$ and let $v_{x,0}=\varnothing, v_{x,1},\ldots, v_{x,n_x}=v_x$ be the successive ancestors of $v_x$ from generation $0$ to generation $n_x$. For every $1\leq i\leq n_x$ set $r_{x,i}= z_{v_{x,i-1}}$, and write $k_{x,i}=k_{v_{x,i-1}}-1$ as the number of siblings of $v_{x,i}$ in $\Pi$.  Then, for every $1\leq i\leq n_x$ and $1\leq j\leq k_{x,i}$, the $j$-th subtree branching off the ancestral line $\llbracket 0, x\rrbracket$ at $v_{x,i-1}$, which is denoted by $\t_{x,i,j}$, corresponds to the pair
$$\left(\Pi[\tilde v_{x,i,j}], (z_{\tilde v_{x,i,j}v}- r_{x,i})_{v \in \Pi[\tilde v_{x,i,j}] }\right),$$
where $\tilde v_{x,i,j}$ is the $j$-th child of $v_{x,i-1}$ different from $v_{x,i}$. To simplify notation, we introduce the point measure 
\begin{displaymath}
\xi_{r,x}(\t)=\sum\limits_{i=1}^{n_{x}}\sum\limits_{j=1}^{k_{x,i}}\delta_{(r_{x,i},\t_{x,i,j})},
\end{displaymath}
which belongs to the set $\mathcal{M}_{p}(\R_{+}\times \T)$ of all finite point measures on $\R_{+}\times \T$.

\begin{figure}[!h]
\begin{center}
\includegraphics[width=6.5cm]{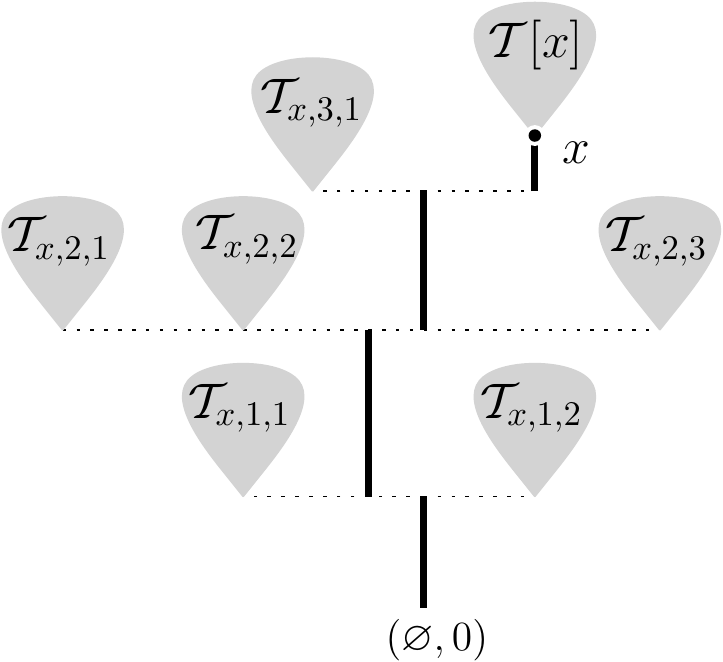}
\caption{Schematic representation of the spine decomposition}
\end{center}
\end{figure}

\begin{lemma}
\label{spine-decomposition}
Fix $\alpha \in (1,2]$. Let $F$ be a nonnegative measurable function on $\T$, and let $H$ be a nonnegative measurable function on $\mathcal{M}_p(\R_+\times \T)$. For $r>0$,
$$\E\Bigg[ \sum_{x\in\Gamma\a_r} F\big(\Gamma^{(\alpha)}[x]\big)\,H\big(\xi_{r,x}(\Gamma\a )\big)\Bigg]
= e^{(m_{\alpha}-1)r}\,\E\big[F\big(\Gamma^{(\alpha)}\big)\big] \times \E\bigg[H\bigg(\sum\limits_{i\in I}\sum\limits_{j=1}^{k_{i}}\delta_{(s_{i},\t_{i,j})}\bigg)\bigg],$$
where we assume that, under the probability measure~$\P$, 
$$\mathcal{N}_{\alpha} \colonequals \sum\limits_{i\in I}\delta_{(s_{i},k_{i},(\t_{i,j},j\geq 1))}$$
is a Poisson point measure on $\R_+\times\N\times \T^{\N}$ with intensity 
$$m_{\alpha}{\bf 1}_{[0,r]}(s)\mathrm{d}s\,\widehat\theta_{\alpha}(\mathrm{d}k)\prod_{j=1}^{\infty} \Theta_{\alpha}(\mathrm{d} \t_{j}).$$ 
\end{lemma}

From now on we fix the stable index $\alpha \in(1,2]$. Unless otherwise specified, we will omit the superscripts and subscripts concerning $\alpha$ in the following proofs.

\subsection{Proof of Proposition~\ref{prop:tree-selected-law}}

Let $F$ be a nonnegative measurable function on $\T$, and consider the quantity
\begin{equation}
\label{law-selected}
I_r \colonequals \E\otimes E\big[F(\Gamma\langle r\rangle )\big]=\E\otimes E\Bigg[ \sum_{x\in \Gamma_r} F(\Gamma[x])\,{\bf 1}_{\{x\prec W_\infty\}}\Bigg],
\end{equation}
where the notation $\E\otimes E$ means that we consider the expectation first under the probability measure $P$ (under which the Brownian motion $W$ is defined) and then under $\P$.

Let us fix $x\in\Gamma_r$ and $R>r$. We write $\wt\Gamma[x] \colonequals \{y\in \Gamma \colon x\prec y\}$ to denote the subset of $\Gamma$ composed of all descendants of $x$ in $\Gamma$. Define
$$\Gamma^{x,R}\colonequals \{y\in \Gamma\backslash \wt\Gamma[x] \colon H(y)\leq R\} \cup \wt\Gamma [x].$$

Let $W^{x,R}$ be Brownian motion with drift $1/2$ on $\Gamma^{x,R}$. We assume that $W^{x,R}$ is reflected both at the root and at the leaves of $\Gamma^{x,R}$, which are the points $y$ of $\Gamma\backslash \wt\Gamma[x]$ such that $H(y)=R$. Write $(\ell^{x,R}_t)_{t\geq 0}$ for the local time process of $W^{x,R}$ at $x$. From excursion theory, $\ell^{x,R}_\infty$ has an exponential distribution with parameter $\cc(\Gamma[x])/2$. For details, we refer the reader to~\cite[Section 3.1]{CLG13}. 

We then consider for every $a\in[0,r]$ the local time process $(L^{a,R}_t)_{t\geq 0}$ of $W^{x,R}$ at the unique point of $\llbracket 0, x\rrbracket$ at distance $a$ from the root. Note in particular that $L^{r,R}_t= \ell^{x,R}_t$. As a consequence of a classical Ray-Knight theorem, conditionally on $\ell^{x,R}_\infty=\ell$, the process $(L^{r-a,R}_\infty)_{0\leq a\leq r}$ is distributed as the process $(X_a)_{0\leq a\leq r}$ which solves the stochastic differential equation
\begin{equation}
\label{EDSRK}
\left\{\begin{array}{l}
\mathrm{d}X_a = 2\sqrt{X_a} \mathrm{d}\eta_a + (2- X_a)\mathrm{d}a \\
X_0=\ell
\end{array}
\right.
\end{equation}
where $(\eta_a)_{a\geq 0}$ is a standard linear Brownian motion. In what follows, we will write $P_\ell$ for the probability measure under which the process $X$ starts from $\ell$, and $P_{(c)}$ for the probability measure under which the process $X$ starts with an exponential distribution with parameter $c/2$. 

Now write $\{x_i, 1\leq i\leq n_{x}\}$ for the branching points of $\Gamma^{x,R}$ (or equivalently of $\Gamma$) that belong to $\llbracket 0,x\rrbracket$, and set $a_i=H(x_i)$ for $1\leq i\leq n_{x}$. We denote by $k_{i}$ the number of subtrees branching off $\llbracket 0,x\rrbracket$ at
$x_i$, and write 
$$(\Gamma^{x,R}_{i,j})_{1\leq j\leq k_{i}}$$ 
for the finite subtrees of $\Gamma^{x,R}$ that branch off $\llbracket 0,x\rrbracket$ at $x_i$. See \cref{figure:treecutoff} for an illustration.

\begin{figure}[!h]
\begin{center}
\includegraphics[width=12cm]{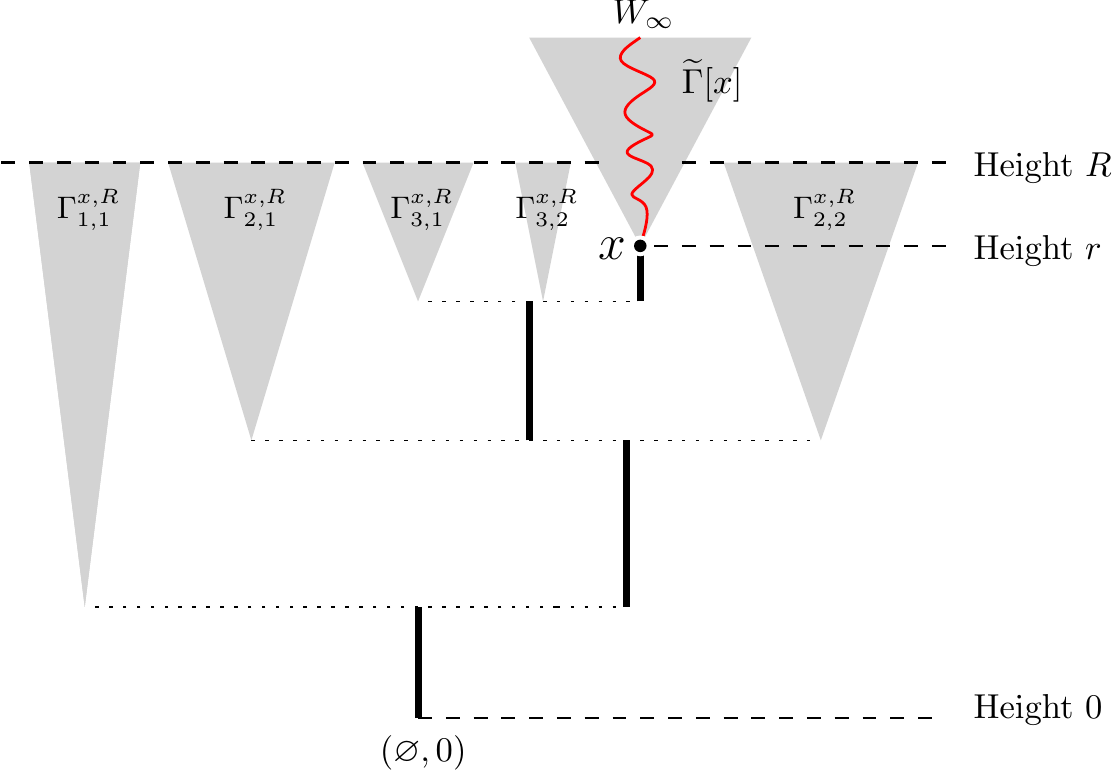}
\caption{The infinite tree $\Gamma^{x,R}$ and the finite subtrees that branch off $\llbracket 0,x\rrbracket$ \label{figure:treecutoff}}
\end{center}
\end{figure}

Let $A_{x,R}$ be the event that $W^{x,R}$ never hits the leaves of $\Gamma^{x,R}$, or equivalently that $W^{x,R}$ escapes to infinity in $\wt\Gamma[x]$ before hitting any leaf of $\Gamma^{x,R}$. Excursion theory
shows that 
$$P(A_{x,R}\mid (L^{a,R}_\infty)_{0\leq a\leq r})= \exp\Big(- \frac{1}{2}\sum_{i=1}^{n_{x}}\sum_{j=1}^{k_{i}}\cc(\Gamma^{x,R}_{i,j})\,L^{a_i,R}_\infty\Big),$$
where  $\cc(\Gamma^{x,R}_{i,j})$ refers to the conductance of $\Gamma^{x,R}_{i,j}$ between its root $x_i$ and the set of its leaves (this conductance is defined by an easy adaptation of the definition given at the beginning of \cref{sec:conductance}).

From the preceding observations, we have thus
\begin{eqnarray}
P(A_{x,R})&=& E\Big[\exp\Big(-\frac{1}{2}\sum_{i=1}^{n_{x}}\sum_{j=1}^{k_{i}}\cc(\Gamma^{x,R}_{i,j})\,L^{a_i,R}_\infty\Big)\Big]\notag\\
&=&
\label{escape-proba}
E_{(\cc(\Gamma[x]))}\Big[\exp\Big(-\frac{1}{2}\sum_{i=1}^{n_{x}}\sum_{j=1}^{k_{i}}\cc(\Gamma^{x,R}_{i,j})\,X_{r-a_i}\Big)\Big].
\end{eqnarray}
At this point, we let $R$ tend to infinity. It is easy to verify that $P(A_{x,R})$ increases to $P(A_x)$, where $A_x\colonequals \{x\prec W_\infty\}$. Furthermore, for every $i\in\{1,\ldots,n_{x}\}$ and $j \in \{1,\ldots, k_{i}\}$, $\cc(\Gamma^{x,R}_{i,j})$ decreases to $\cc(\Gamma_{x,i,j})$, where $\Gamma_{x,i,j}$ is the $j$-th subtree of $\Gamma$ branching off
$\llbracket 0,x\rrbracket$ at $x_i$. Consequently, we obtain 
$$P(x\prec W_\infty) = E_{(\cc(\Gamma[x]))}\Big[\exp\Big(-\frac{1}{2}\sum_{i=1}^{n_{x}}\sum_{j=1}^{k_{i}} \cc(\Gamma_{x,i,j})\,X_{r-a_i}\Big)\Big].$$

We can now return to the computation of the quantity $I_r$ defined in \eqref{law-selected}. 
\begin{align*}
I_r&=\E\Big[ \sum_{x\in \Gamma_r} F(\Gamma[x])\,P(x\prec W_\infty)\Big]\\&
= \E\Big[\sum_{x\in \Gamma_r} F(\Gamma[x]) 
E_{(\cc(\Gamma[x]))}\Big[\exp\Big(-\frac{1}{2}\sum_{i=1}^{n_{x}}\sum_{j=1}^{k_{i}} \cc(\Gamma_{x,i,j})\,X_{r-a_i}\Big)\Big]\Big].
\end{align*}
Note that the quantity inside the sum over $x\in \Gamma_r$ is a function of $\Gamma[x]$ and of the subtrees of $\Gamma$ branching off the segment $\llbracket 0,x\rrbracket$. We can thus apply \cref{spine-decomposition} to get
$$I_r=e^{(m-1)r} \int \Theta(\mathrm{d}\t)\,F(\t)\,E_{(\cc(\t))}\bigg[ \mathbf{E}\Big[\exp\Big(-\frac{1}{2}\int \n\Big(\mathrm{d}s\,\mathrm{d} k\prod_{j=1}^{\infty}\mathrm{d}\t_{j}\Big)\,\sum\limits_{j=1}^{k}\cc(\t_{j})\,X_{r-s}\Big)\Big]\bigg],$$
where the constant $m=\frac{\alpha}{\alpha-1}$ is the mean of $\alpha$-offspring distribution $\theta$. Under the probability measure $\mathbf{P}$, the random measure $\n$ is a Poisson point measure on $\mathbb{R}_{+}\times \N\times \T^{\N}$ with intensity 
$$m{\bf 1}_{[0,r]}(s)\mathrm{d}s\,\widehat \theta(\mathrm{d}k)\prod_{j=1}^{\infty}\Theta(\mathrm{d}\t_{j}),$$
where the size-biased offspring distribution $\widehat\theta$ is defined by~(\ref{eq:size-bias-theta}).

Now we can use the exponential formula for Poisson measures to arrive at
\begin{eqnarray}
I_r &=&e^{(m-1)r} \int \Theta(\mathrm{d}\t)\,F(\t)\,E_{(\cc(\t))}\bigg[\exp\Big(-m\int_0^r \mathrm{d}s\,\sum\limits_{k=1}^{\infty}\tilde\theta(k)\big(1-\varphi(X_s)^{k}\big)\Big)\bigg]\notag\\
\label{firststepeq}
&=& \int \Theta(\mathrm{d}\t)\,F(\t)\,E_{(\cc(\t))}\bigg[\exp-\int_0^r \mathrm{d}s\,\Big(1-m\sum\limits_{k=1}^{\infty}\tilde\theta(k)\varphi(X_s)^{k}\Big)\bigg],
\end{eqnarray}
where we recall that for every $x\geq 0$,
$$\varphi(x)=\E[\exp(-x\,\mathcal{C}/2)]= \Theta\big(\exp(-x\,\cc(\t)/2)\big)$$
is the Laplace transform (evaluated at $x/2$) of the distribution of the conductance $\mathcal{C}(\Gamma)$. Observe that for any $r\in (0,1)$, the identity 
\begin{displaymath}
\sum\limits_{k=1}^{\infty}\widehat \theta(k)r^{k}=1-(1-r)^{\alpha-1}
\end{displaymath}
follows by differentiating~(\ref{eq:gene-fct}). Applying this to~(\ref{firststepeq}), we have thus proved Proposition~\ref{prop:tree-selected-law}.

\subsection{Proof of Proposition~\ref{invariant-meas}}
In order to study the asymptotic behavior of $\Phi_r$ when $r$ tends to $+\infty$, we first observe that, in terms of the law $\gamma(\mathrm{d}s)$ of $\mathcal{C}(\Gamma)$, we have
$$\varphi(\ell) = \int_{[1,\infty)} e^{-\ell s/2}\,\gamma(\mathrm{d}s)\ ,\quad \varphi'(\ell) = -\,\frac{1}{2} \int_{[1,\infty)} s\,e^{-\ell s/2}\,\gamma(\mathrm{d}s).$$
It follows that $\varphi(\ell)\leq e^{-\ell/2}$ and $|\varphi'(\ell)|\leq \frac{1}{2}(\int s\gamma(\mathrm{d}s))e^{-\ell/2}$. 
By differentiating \eqref{eq:laplace}, we have
\begin{equation}
\label{eq:edphi}
2\ell\,\varphi'''(\ell) + (2+\ell) \varphi''(\ell) + \frac{\alpha}{\alpha-1}\varphi'(\ell)\big(1-(1-\varphi(\ell))^{\alpha-1}\big)=0.
\end{equation}

\begin{lemma}
\label{asympto-density}
For every $\ell\geq 0$,
$$\lim_{r\to\infty} E_\ell\Big[ \exp-\int_0^r \mathrm{d}s\,\Big(m\big(1-\varphi(X_s)\big)^{\alpha-1}-\frac{1}{\alpha-1}\Big)\Big] = -\,\frac{\varphi'(\ell) e^{\ell/2}}{\int_0^\infty \mathrm{d}s\,\varphi'(s)^2\,e^{s/2}}.$$
Additionally, there exists a constant $A<\infty$ such that, for every $\ell\geq 0$ and $r>0$,
$$E_\ell\Big[ \exp-\int_0^r \mathrm{d}s\,\Big(m\big(1-\varphi(X_s)\big)^{\alpha-1}-\frac{1}{\alpha-1}\Big)\Big]\leq A.$$
\end{lemma}

\proof 
Firstly, under $\int \Theta(\mathrm{d}\t)\,P_{(\cc(\t))}$, the density of $X_{0}$ is
\begin{equation}
\label{asymptech1}
q(\ell)=\int_{[1,\infty)} \gamma(\mathrm{d}s)\,\frac{s}{2}\,e^{-s\ell/2}= -\varphi'(\ell).
\end{equation}
So from \cref{prop:tree-selected-law}, we have
\begin{align*}
1&= \int \Theta(\mathrm{d}\t)\,E_{(\cc(\t))}\Big[\exp-\int_0^r \mathrm{d}s\,\Big(m\big(1-\varphi(X_s)\big)^{\alpha-1}-\frac{1}{\alpha-1}\Big)\Big]\\
&=-\int \mathrm{d}\ell \,\varphi'(\ell)\,E_\ell\Big[\exp-\int_0^r \mathrm{d}s\,\Big(m\big(1-\varphi(X_s)\big)^{\alpha-1}-\frac{1}{\alpha-1}\Big)\Big].
\end{align*}

We can generalize the last identity via a minor extension of the calculations of the preceding subsection. Let $L^0_\infty$ be the total local time accumulated by the process $W$ at the root of $\Gamma$. Fix $r>0$ and take a nonnegative measurable function $F$ on $\T$. Let $h$ be a bounded nonnegative continuous function on $(0,\infty)$. As an analogue of $I_r$ in the preceding subsection, we set
$$I_r^h:=\E\otimes E\Bigg[ h(L^0_\infty)\,\sum_{x\in \Gamma_r} F(\Gamma[x])\,{\bf 1}_{\{x\prec W_\infty\}}\Bigg].$$
The same calculations that led to~(\ref{escape-proba}) give, for every $x\in\Gamma_r$ and $R>r$,
\begin{eqnarray*}
E[ h(L^{0,R}_\infty)\,{\bf 1}_{A_{x,R}}]&=& E\Big[h(L^{0,R}_\infty)\,\exp\Big(-\frac{1}{2}\sum_{i=1}^{n_{x}}\sum\limits_{j=1}^{k_{i}} \cc(\Gamma^{x,R}_{i,j})\,L^{a_i,R}_\infty\Big)\Big]\notag\\
&=&
E_{(\cc(\Gamma[x]))}\Big[h(X_r)\,\exp\Big(-\frac{1}{2}\sum_{i=1}^{n_{x}}\sum\limits_{j=1}^{k_{i}} \cc(\Gamma^{x,R}_{i,j})\,X_{r-a_i}\Big)\Big].
\end{eqnarray*}
When $R\to\infty$, $L^{0,R}_\infty$ converges to $L^0_\infty$, and so we get
$$E\big[ h(L^{0}_\infty)\,{\bf 1}_{\{x\prec W_\infty\}}\big]=E_{(\cc(\Gamma[x]))}\Big[h(X_r)\exp\Big(-\frac{1}{2}\sum_{i=1}^{n_{x}}\sum\limits_{j=1}^{k_{i}} \cc(\Gamma^{x,R}_{i,j})\,X_{r-a_i}\Big)\Big].$$
We then sum over $x\in \Gamma_r$ and integrate with respect to $\P$. By the same manipulations as in the previous proof, we arrive at
\begin{equation}
\label{firststepeq2}
I^h_r= \int \Theta(\mathrm{d}\t)\,F(\t)\,E_{(\cc(\t))}\Big[h(X_r)\exp-\int_0^r \mathrm{d}s\,\Big(m\big(1-\varphi(X_s)\big)^{\alpha-1}-\frac{1}{\alpha-1}\Big)\Big].
\end{equation}
Note that if $F=1$,
$$I^h_r= \E\otimes  E\big[h(L^0_\infty)\big] = - \int_0^\infty \mathrm{d}\ell\,\varphi'(\ell) h(\ell)$$
since given $\Gamma=\t$ the local time $L^0_\infty$ follows an exponentiel distribution with parameter $\cc(\t)/2$, and we use the same calculation as in (\ref{asymptech1}). Hence the case $F=1$ of (\ref{firststepeq2}) gives
\begin{equation}
\label{asymptech2}
\int_0^\infty \mathrm{d}\ell \,\varphi'(\ell)\,E_\ell\Big[h(X_r)\exp-\int_0^r \mathrm{d}s\,\Big(m\big(1-\varphi(X_s)\big)^{\alpha-1}-\frac{1}{\alpha-1}\Big)\Big]
=\int_0^\infty \mathrm{d}\ell\,\varphi'(\ell)\, h(\ell).
\end{equation}
By a standard truncation argument, this identity also holds if $h$ is unbounded.

\begin{lemma}
\label{martingalelemma}
The process
$$M_t\colonequals -\varphi'(X_t)\,\exp\bigg(\frac{X_t}{2}-\int_0^t \mathrm{d}s\,\Big(m\big(1-\varphi(X_s)\big)^{\alpha-1}-\frac{1}{\alpha-1}\Big)\bigg), \quad t\geq 0$$
is a martingale under $P_\ell$, for every $\ell\geq 0$.
\end{lemma}

\noi{\bf Proof of \cref{martingalelemma}.} From the stochastic differential equation (\ref{EDSRK}), an application of It\^o's formula shows that the finite variation part of the semimartingale $-M_t$ is
\begin{equation}
\label{eq:finite-var}
\int_0^t\Big(2X_s\varphi'''(X_s) + (2+X_s)\varphi''(X_s) + \varphi'(X_s)\big(\frac{\alpha}{\alpha-1}-m(1-\varphi(X_{s})^{\alpha-1}\big)\Big)Y_{s}\, \mathrm{d}s,
\end{equation}
where for any $s\geq 0$,
$$Y_{s}\colonequals \exp\Big(\frac{X_s}{2}-\int_0^s\mathrm{d}u\Big(m\big(1-\varphi(X_s)\big)^{\alpha-1}-\frac{1}{\alpha-1}\Big)\Big). $$
Recall that $m=\frac{\alpha}{\alpha-1}$, and hence~(\ref{eq:finite-var}) vanishes thanks to~(\ref{eq:edphi}), whereupon $M$ is a local martingale. Furthermore, we have already noticed that, for every $\ell\geq 0$, $|\varphi'(\ell)|\leq Ce^{-\ell/2}$, where $C\colonequals \frac{1}{2}\int s\gamma(\mathrm{d}s)$. It follows that $|M|$ is bounded by $C\exp(\frac{t}{\alpha-1})$ over the time interval $[0,t]$, and thus $M$ is a (true) martingale. \endproof

\smallskip
We return to the proof of \cref{asympto-density}.
Let $\ell\geq 0$ and $t>0$. On the probability space where $X$ is defined, we introduce a new probability measure $Q^t_\ell$ by setting
$$Q^t_\ell \colonequals \frac{M_t}{M_0}\cdot P_\ell.$$
The fact that $Q^t_\ell$ is a probability measure follows from the martingale property derived in \cref{martingalelemma}. By definition of $M_{t}$, we have $P_\ell$-a.s.
$$ \frac{M_t}{M_0} = \frac{\varphi'(X_t)}{\varphi'(\ell)} \exp\Big(\frac{X_t-\ell}{2} - \int_0^t \mathrm{d}s\,\Big(m\big(1-\varphi(X_s)\big)^{\alpha-1}-\frac{1}{\alpha-1}\Big)\Big),$$
so that the martingale part of $\log\frac{M_t}{M_0}$ is
$$\int_0^t \sqrt{X_s}\,\mathrm{d}\eta_s + 2\int_0^t \frac{\varphi''(X_s)}{\varphi'(X_s)} \sqrt{X_s}\,\mathrm{d}\eta_s,$$
where $\eta$ is the linear Brownian motion in (\ref{EDSRK}). An application of Girsanov's theorem shows that under $Q^t_\ell$, the process
$$\wt\eta_s \colonequals  \eta_s -\int_0^s \sqrt{X_u}\Big(1+ \frac{2\varphi''(X_u)}{\varphi'(X_u)}\Big)\,\mathrm{d}u\;,\quad 0\leq s\leq t,$$
is a linear Brownian motion over the time interval $[0,t]$. Furthermore, on the same time interval $[0,t]$, the process $X$ satisfies the stochastic differential equation
$$\mathrm{d}X_s = 2\sqrt{X_s}\,\mathrm{d}\wt\eta_s + 2X_s\Big(1+ \frac{2\varphi''(X_s)}{\varphi'(X_s)}\Big)\mathrm{d}s + (2-X_s)\,\mathrm{d}s,$$
or equivalently, using (\ref{eq:laplace}),
\begin{equation}
\label{EDSbis}
\mathrm{d}X_s = 2\sqrt{X_s}\,\mathrm{d}\wt\eta_s + \Big(2-X_s + \frac{2}{\alpha-1}\frac{1-\varphi(X_{s})-(1-\varphi(X_{s}))^{\alpha}}{\varphi'(X_{s})}\Big)\mathrm{d}s.
\end{equation}
Notice that the function 
$$\ell \mapsto \frac{1-\varphi-(1-\varphi)^{\alpha}}{\varphi'}(\ell)$$
is continuously differentiable over $[0,\infty)$, takes negative values on $(0,\infty)$ and vanishes at $0$. Pathwise uniqueness, and therefore also weak uniqueness, holds for (\ref{EDSbis}) by an application of the classical Yamada-Watanabe criterion. The preceding considerations show that, under the probability measure $Q^t_\ell$ and on the time interval $[0,t]$, the process $X$ is distributed as the diffusion process on $[0,\infty)$ started from $\ell$, with generator
$$\mathcal{L} = 2r\,\frac{\mathrm{d}^2}{\mathrm{d}r^2} + \Big(2-r + \frac{2}{\alpha-1}\frac{1-\varphi-(1-\varphi)^{\alpha}}{\varphi'}(r)\Big)\,\frac{\mathrm{d}}{\mathrm{d}r}.$$
Write $\wt X$ for this diffusion process, and assume that $\wt X$ starts from $\ell$ under the probability measure $P_\ell$. Note that $0$ is an entrance point for $\wt X$, but independently of its starting point, $\wt X$ does not visit $0$ at a positive time. By comparing the solutions of (\ref{EDSRK}) and (\ref{EDSbis}), we know that $\wt X$ is recurrent on $(0,\infty)$. 

We next observe that, by (\ref{asymptech2}) and a few lines of calculations, the finite measure $\lambda$ on $(0,\infty)$ defined by
$$\lambda(\mathrm{d}\ell) \colonequals \varphi'(\ell)^2\,e^{\ell/2}\,\mathrm{d}\ell$$
is invariant for $\wt X$. We normalize $\lambda$ by setting
$$\wh\lambda = \frac{\lambda}{\lambda((0,\infty))}.$$

It is then easy to prove that the distribution of $\wt X_t$ under $P_\ell$ converges weakly to $\wh\lambda$ as $t\to\infty$, for any $\ell \geq 0$. Consequently, for any bounded continuous function $g$ on $[0,\infty)$, and every $\ell\geq 0$,
\begin{equation}
\label{conv-mesu-invar}
E_\ell\big[g(\wt X_t)\big] \build{\la}_{t\to\infty}^{} \int g\,\mathrm{d}\wh\lambda.
\end{equation}
By the same argument as in \cite[Section 3.2]{CLG13}, the preceding convergence remains true if $g$ is a continuous, increasing and nonnegative function such that $\int g\,\mathrm{d}\wh\lambda<\infty$.

We can thus apply (\ref{conv-mesu-invar}) to the function
$$g(\ell)= -\frac{1}{\varphi'(\ell)}\,e^{-\ell/2},$$
which satisfies the desired properties and in particular $\int g\,\mathrm{d}\lambda =-\int \varphi'(\ell)\,\mathrm{d}\ell = 1$. For this function $g$,
$$E_\ell\big[g(\wt X_t)\big] = Q^t_\ell\big[g(X_t)\big]= -\frac{e^{-\ell/2}}{\varphi'(\ell)}\,E_\ell\Big[ \exp-\int_0^t \mathrm{d}s\,\Big(m\big(1-\varphi(X_s)\big)^{\alpha-1}-\frac{1}{\alpha-1}\Big)\Big].$$
It follows from (\ref{conv-mesu-invar}) that, for every $\ell\geq 0$, 
$$\lim_{t\to\infty} -\frac{e^{-\ell/2}}{\varphi'(\ell)}\,E_\ell\Big[
\exp-\!\int_0^t \mathrm{d}s\,\Big(m\big(1-\varphi(X_s)\big)^{\alpha-1}-\frac{1}{\alpha-1}\Big)\Big] =\int g\mathrm{d}\wh\lambda= \frac{1}{\lambda((0,\infty))}= \frac{1}{\int_{0}^{\infty}\mathrm{d}s \varphi'(s)^2e^{s/2}},$$
which gives the first assertion of the lemma. The second assertion of \cref{asympto-density} can be shown in the same way as in~\cite{CLG13}.\hfill$\square$

\smallskip 
By definition, we have
$$\Phi_r(c)= \frac{c}{2}\int_0^\infty \mathrm{d}\ell\,e^{-c\ell/2} \,E_\ell\Big[\exp-\int_0^r \mathrm{d}s\,\Big(m\big(1-\varphi(X_s)\big)^{\alpha-1}-\frac{1}{\alpha-1}\Big)\Big].$$
From \cref{asympto-density}
and an application of the dominated convergence theorem, we get
$$\lim_{r\to +\infty}\Phi_r(c)=  c\int_0^\infty \mathrm{d}\ell\,e^{-c\ell/2} \times  \Big(-\frac{\varphi'(\ell) e^{\ell/2}}{C_1}\Big),$$
where 
$$C_1\colonequals 2\int_0^\infty \mathrm{d}s\,\varphi'(s)^2\,e^{s/2} = \int\!\!\int \gamma(\mathrm{d}\ell)\gamma(\mathrm{d}
\ell') \frac{\ell \ell'}{\ell+\ell'-1}.$$
By a straightforward calculation, the preceding limit is identified with $\Phi_\infty(c)$ defined in the statement of Proposition~\ref{invariant-meas}. 

Finally, with all the ingredients prepared above in this appendix, we can show the invariance of $\Lambda^{*}(\mathrm{d}\t \mathrm{d}\mathbf{v})$ under the shifts $(\tau_{r},r\geq 0)$ in the same way as in~\cite[Proposition 12]{CLG13}, and the proof of Proposition~\ref{invariant-meas} is therefore completed. 

\subsection{Another derivation of formula~(\ref{eq:beta1})}
Recall that $\nu_\t$ stands for the harmonic measure of a tree $\t\in \T$. For every $r>0$, we consider the nonnegative measurable function $G_r$ defined on $\T^*$ by the formula
$$G_r(\t,\mathbf{v})\colonequals -\log  \nu_\t(\mathcal{B}_\t(\mathbf{v},r)),$$
where $\mathcal{B}_\t(\mathbf{v},r)$ denotes the set of all geodesic rays of $\t$ that coincide with the ray $\mathbf{v}$ over the interval $[0,r]$. The flow property of harmonic measure (cf.~Lemma 7 in~\cite{CLG13}) implies that, for every $r,s>0$, we have
$$G_{r+s}= G_r + G_s\circ \tau_r.$$
Since the shift $\tau_r$ acting on $(\T^{*},\Lambda^{*})$ is ergodic, the Birkhoff ergodic theorem implies that $\Lambda^*$-a.s.
$$\lim\limits_{s \to \infty}\frac{G_s}{s}= \Lambda^*(G_1).$$
Recall that $\Lambda^*$ has a strictly positive density with respect to $\Theta^*$. So the latter convergence also holds $\Theta^*$-a.s., which gives the convergence~(\ref{eq:1}) with $\beta = \Lambda^*(G_1)>0$. 

For $\ve >0$, we define a nonnegative function $H_\ve$ on $\T^*$ by setting
$$H_\ve(\t,\mathbf{v})\colonequals  \left\{
\begin{array}{ll}
0&\hbox{if } z_{\varnothing} \geq \ve,\\
-\log \nu_\t(\{\mathbf{v}'\in \N^{\N}\colon \mathbf{v}_1\prec \mathbf{v}'\})\quad &\hbox{if } z_{\varnothing} < \ve,
\end{array}
\right.
$$
where we write $\t=(\Pi, (z_v)_{v\in\Pi})$ as in Section~\ref{sec:ctgwtree}. Clearly, $H_\ve(\t,\mathbf{v})\leq G_\ve(\t,\mathbf{v})$, and $H_\ve(\t,\mathbf{v})=G_\ve(\t,\mathbf{v})$ if $z_{\mathbf{v}_1}\geq\ve$. More generally, $H_\ve\circ \tau_r(\t,\mathbf{v})=G_\ve\circ \tau_r(\t,\mathbf{v})$ if there is at most one index $i\geq 0$ such that $r\leq z_{\mathbf{v}_i} < r+\ve$. It follows from these remarks that, for every integer $n\geq 1$,
\begin{equation}
\label{ergodic-tech1}
G_1\geq \sum_{k=0}^{n-1} H_{1/n} \circ \tau_{k/n},
\end{equation}
and for every $(\t,\mathbf{v})\in \T^*$,
\begin{equation}
\label{ergodic-tech2}
G_1(\t,\mathbf{v})=\lim_{n\to\infty} \sum_{k=0}^{n-1} H_{1/n} \circ \tau_{k/n}(\t,\mathbf{v}).
\end{equation}

Let us then investigate the behavior of $\Lambda^*(H_\ve)$ when $\ve\to 0$. By considering the subtrees $\t_{(1)},\ldots,\t_{(k_{\varnothing})}$ of $\t$ obtained at the first branching point, we can write
\begin{equation}
\label{eq:integro1}
\Lambda^*(H_\ve)=-\int \Theta(\mathrm{d}\t)\,\Phi_\infty(\cc(\t))\,{\bf 1}_{\{z_\varnothing<\ve\}} \sum\limits_{i=1}^{k_{\varnothing}}\frac{\cc(\t_{(i)})}{\sum_{j=1}^{k_{\varnothing}}\cc(\t_{(j)})} \log \frac{\cc(\t_{(i)})}{\sum_{j=1}^{k_{\varnothing}}\cc(\t_{(j)})}.
\end{equation}
Recall the branching property of the CTGW tree, and notice that 
$$\cc(\t)= \frac{\sum_{j=1}^{k_{\varnothing}}\cc(\t_{(j)})}{e^{-z_\varnothing} + (1-e^{-z_\varnothing})\big(\sum_{j=1}^{k_{\varnothing}}\cc(\t_{(j)})\big)}.$$
Substituting this into (\ref{eq:integro1}), we see that $\Lambda^*(H_\ve)$ can be expanded as
\begin{eqnarray*}
\lefteqn{-\sum\limits_{k=2}^{\infty}\theta(k)\int\! \Theta(\mathrm{d}\t_{1}) \int\! \Theta(\mathrm{d}\t_{2})\cdots \int\!\Theta(\mathrm{d}\t_{k})\sum\limits_{i=1}^{k}\frac{\cc(\t_{i})}{\sum_{j=1}^{k}\cc(\t_{j})} \log \frac{\cc(\t_{i})}{\sum_{j=1}^{k}\cc(\t_{j})}}\hspace{3cm}\\
  &\times &\int_{0}^{\ve}\mathrm{d}z \,e^{-z}\,\Phi_\infty\Big(\frac{\sum_{j=1}^{k}\cc(\t_{j})}{e^{-z} + (1-e^{-z})\big(\sum_{j=1}^{k}\cc(\t_{j})\big)}\Big)\\
\lefteqn{=-\sum\limits_{k=2}^{\infty} k\,\theta(k)\int\! \Theta(\mathrm{d}\t_{1}) \int\! \Theta(\mathrm{d}\t_{2})\cdots \int\!\Theta(\mathrm{d}\t_{k})\,\frac{\cc(\t_{1})}{\sum_{j=1}^{k}\cc(\t_{j})} \log \frac{\cc(\t_{1})}{\sum_{j=1}^{k}\cc(\t_{j})}}\hspace{3.4cm}\\
 &\times &\int_{0}^{\ve}\mathrm{d}z \,e^{-z}\,\Phi_\infty\Big(\frac{\sum_{j=1}^{k}\cc(\t_{j})}{e^{-z} + (1-e^{-z})\big(\sum_{j=1}^{k}\cc(\t_{j})\big)}\Big)\\
\end{eqnarray*}
by a symmetry argument. Since $\Phi_\infty$ is a bounded continuous function, and 
\begin{displaymath}
\left|\frac{\cc(\t_{1)})}{\sum_{j=1}^{k}\cc(\t_{j})} \log \frac{\cc(\t_{1})}{\sum_{j=1}^{k}\cc(\t_{j})}\right|\leq 1,
\end{displaymath}
we can let $\ve\to 0$ in the preceding expression and get
\begin{equation}
\label{ergodic-tech3}
\lim_{\ve\to 0}\frac{\Lambda^*(H_\ve)}{\ve} 
= -\sum\limits_{k=2}^{\infty} k\theta(k)\!\int\! \Theta(\mathrm{d}\t_{1}) \cdots \!\int\!\Theta(\mathrm{d}\t_{k})\frac{\cc(\t_{1})}{\sum_{j=1}^{k}\cc(\t_{j})} \log \frac{\cc(\t_{1})}{\sum_{j=1}^{k}\cc(\t_{j})}\Phi_\infty\Big(\sum_{j=1}^{k}\cc(\t_{j})\Big).
\end{equation}
Note that we used the fact that $\theta$ has a finite first moment. Sine the limit in the preceding display is finite, we can use (\ref{ergodic-tech2}) and Fatou's lemma to get that $\Lambda^*(G_1)<\infty$, and then use (\ref{ergodic-tech1}) (to justify dominated convergence) and (\ref{ergodic-tech2}) again to obtain that
$$\beta=\Lambda^*(G_1) =\lim_{n\to \infty} n\, \Lambda^*(H_{1/n})$$
coincides with the right-hand side of (\ref{ergodic-tech3}). Using the expression of $\Phi_\infty$, we can therefore reformulate $\beta$ as in formula~(\ref{eq:beta1}).

\end{document}